\documentclass[10pt,a4paper]{amsart}

\usepackage[a4paper,top=3.5 cm,bottom=3 cm,left=2.5 cm,right=2.5 cm]{geometry}
\allowdisplaybreaks
\usepackage[T1]{fontenc}
\usepackage[english]{babel}
\usepackage{url}
\usepackage{amsmath,amssymb,amsthm}

\usepackage{graphicx}
\usepackage{xcolor} 
\usepackage{bbm}
\usepackage{dsfont}
\usepackage{stmaryrd}
\usepackage{esint}
\usepackage{floatrow}
\usepackage{epsfig}

\usepackage[normalem]{ulem}
\usepackage{mathrsfs}
\usepackage{tikz}
\usetikzlibrary{decorations.markings,backgrounds}
\usetikzlibrary{arrows.meta}
\usepackage{pgfplots}\pgfplotsset{compat=1.18}
\usepackage{pst-plot}
\usepackage{subfig}
\usepackage{caption}
\usepackage{float}
\usepackage[colorinlistoftodos]{todonotes}
\usepackage{import}
\usepackage{enumitem}
\usepackage{verbatim,fancyvrb}
\usepackage{epigraph}

\usepackage{mathtools}
\usepackage{hyperref}
\usepackage[nameinlink,capitalise,sort]{cleveref}
\usepackage{autonum}
\crefname{equation}{}{} 
\crefname{enumi}{}{} 

\definecolor{mblue}{rgb}{         0    0.4470    0.7410}
\definecolor{mred}{rgb}{    0.8500    0.3250    0.0980}  
\definecolor{myellow}{rgb}{    0.9290    0.6940    0.1250}

\theoremstyle{plain}
\newtheorem{lemma}{Lemma}[section]
\newtheorem{proposition}[lemma]{Proposition}
\newtheorem{theorem}[lemma]{Theorem}
\newtheorem{corollary}[lemma]{Corollary}

\newtheorem*{idea*}{Idea}

\theoremstyle{definition}
\newtheorem{definition}[lemma]{Definition}

\theoremstyle{remark}
\newtheorem{remark}[lemma]{Remark}

\numberwithin{equation}{section}

\newcommand{\loc}{\mathrm{loc}}

\newcommand{\R}{\mathbb{R}}
\newcommand{\N}{\mathbb{N}}
\newcommand{\T}{\mathbb{T}}
\newcommand{\Z}{\mathbb{Z}}
\newcommand{\eps}{{\varepsilon}}

\DeclareMathOperator{\supp}{supp}
\DeclareMathOperator{\dive}{div}
\DeclareMathOperator{\curl}{curl}          
\DeclareMathOperator{\dist}{dist}

\DeclareMathOperator{\tr}{tr}              

\newcommand{\abs}[1]{\left\lvert#1\right\rvert}
\newcommand{\norm}[1]{\left\lVert#1\right\rVert}
\newcommand\ip[1]{\left\langle#1\right\rangle}

\newcommand{\dd}{\,\mathrm{d}}

\newcounter{eq}[section]

\title[On Lions' density patch problem at a critical level of regularity]{On Lions' density patch problem at a critical level of regularity}

\date{}

\author[S. \v{S}kondri\'c]{Stefan \v{S}kondri\'c}
\address[S. \v{S}kondri\'c]{Chair of Analysis, Department of
Mathematics, Friedrich-Alexander-Universität Erlangen-Nürnberg, Cauerstraße 11, 91058 Erlangen, Germany.}
\email{stefan.skondric@fau.de}

\author[A.~Violini]{Alessandro~Violini}
\address[A.~Violini]{Universit\"at Basel, Department Mathematik und Informatik, Spiegelgasse 1, 4051 Basel, Switzerland. }
\email{alessandro.violini@unibas.ch}


\begin{document}

\begin{abstract}
In this article, we study Lions' density patch problem in two space dimensions at critical regularity. We prove global existence, uniqueness, and stability for a fluid occupying a bounded Lipschitz region surrounded by vacuum and evolving according to the incompressible Navier--Stokes equations, with initial velocity in $\dot{B}^0_{2,1}(\R^2)$. Moreover, we show that the Lipschitz regularity of the patch is preserved, and that its long-time dynamics is a rigid motion leading to the emergence of an asymptotic domain.
\end{abstract}
\maketitle

\tableofcontents

\section{Introduction}

We study the evolution of a homogeneous incompressible fluid occupying a bounded region \(D \subset \mathbb{R}^2\), surrounded by vacuum, and evolving according to the incompressible Navier--Stokes equations. This configuration is commonly referred to as the \emph{density patch problem}.

The region occupied by the fluid evolves in time into a family of domains \(D_t\), describing the motion of the interface. This point of view naturally leads to a free-boundary formulation, where the evolution of the fluid is described through the motion of the domain \(D_t\) and its boundary.

An equivalent Eulerian formulation is obtained by introducing a density transported by the flow, which formally takes the form  
\begin{align}
    \rho(t) = \mathbf{1}_{D_t}.
\end{align}

In this formulation, the problem can be described by the inhomogeneous incompressible Navier--Stokes system in  
\(\mathbb{R}_+ \times \mathbb{R}^2\)
\begin{align}\label{eq:INS}
    \begin{cases}
        \partial_t (\rho u) + \dive (\rho u \otimes u) - \nu \Delta u + \nabla P = 0, \\
        \partial_t \rho + \dive (\rho u) = 0, \\
        \dive u = 0,
    \end{cases}
\end{align}
where \(\rho \geq 0\) denotes the density, \(u\) the velocity field, \(P\) the pressure, and \(\nu\) the viscosity. The system is supplemented with the initial data
\begin{align}\label{eq:initial_data}
\begin{cases}
\rho(0,x) = \mathbf{1}_D,\\
u(0,x) = u_0(x) \qquad \text{with } \dive u_0 = 0.
\end{cases}
\end{align}

In this framework, the evolution of the interface \(D_t\) can be studied from a Lagrangian perspective. Indeed, the domains are given by
\begin{align}
D_t = X(t,D),
\end{align}
where \(X(t,x)\) denotes the flow associated with the velocity field \(u\). Since \(u\) is incompressible, i.e.\ \(\dive u = 0\), the flow \(X(t,\cdot)\) is measure-preserving, and therefore
\begin{align}
    |D_t| = |D|.
\end{align}

In particular, the evolution preserves the size of the patch, but this information alone does not provide any control on the geometry of the boundary. This naturally leads to the central question, raised by Lions in his monograph
\emph{Mathematical Topics in Fluid Mechanics}~\cite[p.~34]{Lions1996}, of whether the regularity of the domains \(D_t\) is preserved by the time evolution.

More precisely, starting from an initial domain \(D\) with a given geometric regularity, one may ask whether the transported domains \(D_t\) retain the same structure, or whether singularities may develop at the interface as time evolves. Since the domain \(D\) is transported by the flow \(X(t,\cdot)\), its evolution ultimately depends on the regularity properties of the velocity field \(u\) and on the initial data \(u_0\).

\vspace{0.3cm}

Our main result provides a positive answer to this question in a critical regularity framework. We establish existence, uniqueness, and stability of solutions, together with the propagation of Lipschitz regularity of the interface. In addition, we describe the long-time behavior of the flow, showing convergence towards a rigid motion and the emergence of an asymptotic domain.

Before stating a simplified version of our main result, we introduce the quantity
\begin{align}\label{eq:Mdef}
    M := \frac{1}{|D|}\int_{\R^2}\rho_0(x)u_0(x)\,\dd x
    = \frac{1}{|D|}\int_{D} u_0(x)\,\dd x.
\end{align}
This quantity can be interpreted as the average momentum per unit mass inside the patch \(D\). In particular, it has the same scaling as the velocity field \(u\), and represents the average initial velocity of the fluid in \(D\).

\begin{theorem}\label{thm:main}
Let $\rho_0 = \mathbf{1}_{D}$ be the characteristic function of a domain 
$D \subset \R^2$ with Lipschitz boundary, 
and let the initial velocity satisfy 
\begin{align}
    u_0 \in \dot{B}^0_{2,1}(\R^2), 
    \qquad \dive u_0 = 0.
\end{align}
Then there exists a unique global solution $(\rho,u)$ to \eqref{eq:INS} satisfying the energy equality such that 
\begin{align}   \sqrt{\rho }\dot u, \nabla^2 u , \nabla P \in L^1 \big((0,\infty);L^2(\R^2)\big), \qquad
     \nabla u \in L^1 \big((0,\infty);L^\infty(\R^2)\big),
     \qquad 
     \rho(t) = \mathbf{1}_{D_t},
\end{align}
where $D_t = X(t,D) \subset \R^2$ is a Lipschitz domain for all $t\geq 0$, and $X(t,x)$ is the flow associated with $u$.

Moreover, the quantity $M$ is conserved in time, and there exists a bi-Lipschitz, measure-preserving map 
$X_\infty : D \to \R^2$ such that
\begin{align}
\|X(t,\cdot)-Mt - X_\infty(\cdot)\|_{L^\infty(D)} \to 0
\qquad \text{as } t\to\infty .
\end{align}
In particular, up to the rigid translation $Mt$, the domains $D_t$ converge in the Hausdorff sense to the limiting domain $D_\infty \coloneqq X_\infty(D)$. 

If, in addition, $u_0 \in \dot{B}^\gamma_{2,1}$ for some $\gamma \in (0,1)$, then 
$\nabla u \in L^1((0,\infty);C^{\gamma}(\R^2))$. In this case, if the initial domain 
$D$ is of class $C^{1,\gamma}$, then the evolved domains $D_t$ remain of class 
$C^{1,\gamma}$ for all $t \geq 0$.
\end{theorem}
The previous theorem shows that the evolution of the patch splits into a finite deformation and a rigid motion, namely a translation with constant velocity $M$, directed along the average of the initial velocity inside the patch. This decomposition highlights a striking rigidity in the long-time behavior of the flow. Indeed, the velocity converges exponentially fast to its spatial average inside the patch, namely
\begin{align}
\|u(t)-M\|_{L^2(D_t)}
\lesssim
\|u_0-M\|_{L^2(D)}\,e^{-\lambda t},
\end{align}
as proved in \cref{lem:expdecay}. The rate \(\lambda\) has the same scaling as the two-dimensional heat equation. This decay is a consequence of the dissipative structure of the system, which tends to damp out spatial variations of the velocity. At the same time, the conservation of the total momentum prevents the velocity from vanishing, and instead selects the constant vector \(M\) as the asymptotic state. As a consequence, the velocity becomes asymptotically constant and the motion is governed by the uniform drift \(M\). 

Moreover, the condition \(\nabla u \in L^1_t L^\infty_x\) implies that the associated flow has finite total distortion. In particular, all the deformation of the domain is accumulated over time and remains finite. More precisely, after subtracting the translation \(Mt\), the flow converges uniformly to a limiting map \(X_\infty\), which admits the representation
\begin{align}
    X_\infty(x) = x + \int_0^\infty \bigl(u(t, X(t,x)) - M\bigr)\,\dd t,
\end{align}
see \cref{lem:finaldomain}. In particular, \(X_\infty\) encodes the total deformation of the domain. As a consequence, the domains \(D_t - Mt\) converge in the Hausdorff sense to the limiting domain \(D_\infty\), and no further deformation occurs as \(t \to \infty\).

The key mechanism behind the previous result is the integrability condition
\begin{align}\label{eq:LIp}
    \nabla u \in L^1\big((0,\infty);L^\infty\big),
\end{align}
which ensures that the associated flow has finite total distortion and allows one to control the evolution of the interface. This condition is critical with respect to the natural scaling of the system, and therefore represents the threshold for controlling the geometry of the flow.

We recall that the natural scaling of~\eqref{eq:INS} is given by
\begin{equation}\label{eq:scaling}
(\rho_\lambda,u_\lambda,P_\lambda)(t,x) 
:= \big(\rho(\lambda^2 t,\lambda x), \; \lambda u(\lambda^2 t,\lambda x), \; \lambda^2 P(\lambda^2 t,\lambda x)\big),
\qquad \forall \lambda >0.
\end{equation}
The space $\dot{B}^0_{2,1}$ is invariant under the scaling \eqref{eq:scaling}. Moreover, one has the continuous embedding
\begin{align}
    \dot{B}^0_{2,1} \hookrightarrow L^2,
\end{align}
while the space $L^2$ alone is not sufficient to ensure the propagation of the regularity of the patch.

This obstruction is already visible at the level of the linear heat equation: in two dimensions, the heat flow with initial data in $\dot{B}^0_{2,1}$ satisfies \eqref{eq:LIp} 
whereas there exist explicit initial data $u_0 \in L^2$ for which this property fails. In particular, one cannot expect to obtain \eqref{eq:LIp} for solutions to \eqref{eq:INS} starting from $L^2$ initial data. As a consequence, the preservation of the Lipschitz regularity of the patch cannot hold in this setting.

For this reason, $\dot{B}^0_{2,1}$ appears as the natural (critical) functional framework for the density patch problem with Lipschitz boundary.

\subsubsection*{Overview of the literature.}  The mathematical analysis of~\eqref{eq:INS} dates back to the works of Kazhikhov~\cite{Kazhikhov1974} and Simon~\cite{Simon1990}, who established global weak solutions under suitable assumptions on the density. A major breakthrough was achieved by Lions~\cite{Lions1996}, who constructed global weak solutions for initial velocity in $L^2$ and bounded nonnegative initial density.

\underline{Well-posedness.}
The well-posedness theory in the absence of vacuum has been extensively developed in a series of works, including Danchin~\cite{Danchin2003}, Abidi~\cite{Abidi2007}, Abidi--Gui--Zhang~\cite{AbidiGuiZhang2012}, Abidi--Gui~\cite{AbidiGui2021}, Danchin--Mucha~\cite{DanchinMucha2012}, Paicu--Zhang--Zhang~\cite{PaicuZhangZhang2013}, Zhang~\cite{Zhang2020}, and Danchin--Wang~\cite{DanchinWang2023}. Recently, global well-posedness in two dimensions for critical data was established by Danchin~\cite{Danchin2025} for $u_0 \in \dot{B}^0_{2,1}$, and shortly after by Hao, Shao, Wei and Zhang~\cite{HaoShaoWeiZhang2025} for $u_0 \in L^2$. The latter result was further extended to the full Leray--Hopf class in~\cite{Skondric2025}. In this sense, the well-posedness theory in the absence of vacuum is by now well understood.

The analysis in the presence of vacuum for bounded densities is more delicate, since the momentum equation degenerates where the density vanishes. Danchin and Mucha~\cite{DanchinMucha2019} proved global well-posedness in two dimensions for $u_0 \in H^1(\T^2)$ and general bounded nonnegative densities. The extension to the whole space $\R^2$ was obtained by Prange and Tan~\cite{PrangeTan2023} and by Hao, Shao, Wei and Zhang~\cite{HaoShaoWeiZhang2026}, under structural or geometric assumptions on the density, covering in particular density patch configurations. In both works, the initial velocity belongs in $H^1$.

\underline{Propagation of regularity.} The propagation of boundary regularity was first studied in a relaxed setting compared to that of Lions, where the initial density takes the form
\begin{align}
\rho_0(x) = \rho^{\mathrm{in}}_0\mathbf{1}_{D}(x) + \rho^{\mathrm{out}}_0\mathbf{1}_{D^c}(x), \qquad \rho^{\mathrm{in}}_0, \rho^{\mathrm{out}}_0 > c \geq 0.
\end{align}

In this framework, Liao and Zhang~\cite{LiaoZhang2016,LiaoZhang2019} proved the propagation of $W^{3,p}(\R^2)$ regularity of the interface for $p \in (2,4)$, first under a smallness assumption on the density jump and later without this restriction. See also Danchin and Zhang~\cite{DanchinZhang2017} for related results. Gancedo and García-Juárez~\cite{GancedoGarcia2018} established the propagation of $C^{1+\gamma}$ regularity of the interface for initial velocities with sharp Sobolev regularity $u_0 \in H^{s}(\R^2)$, with $s \in (\gamma,1)$, without any smallness assumptions.

In the presence of vacuum, that is when $\rho_0^{\mathrm{out}} = 0$, only few results are available. Danchin and Mucha~\cite{DanchinMucha2019} proved the preservation of $C^{1,\gamma}$ regularity of the boundary $D$ for initial velocities $u_0 \in H^1(\T^2)$ (see also \cite{DanchinMuchaPiasecki2024}). More recently, Gancedo, García-Juárez, and Luna-Velasco~\cite{GancedoGarciaJuarezLunaVelasco2025} extended this result to initial data $u_0 \in H^s(\R^2)$ with $s \in (\gamma,1)$.

\cref{thm:main} advances the theory initiated in~\cite{DanchinMucha2019} by establishing global well-posedness together with the propagation of boundary regularity in the critical scaling-invariant framework, in the whole space and in the presence of vacuum.  In addition, the Lipschitz control \eqref{eq:LIp} provides a quantitative description of the flow and allows us to characterize its long-time asymptotic behavior. To the best of our knowledge, this is the first result in the literature establishing such asymptotic properties for density patches.

\subsubsection*{Conventions on the Notation.}

We write $A \lesssim B$ to denote an inequality up to a multiplicative constant whose dependence is not relevant. When we need to emphasize the dependence on a parameter $p$, we write $A \lesssim_p B$. The measure of integration is omitted whenever it is clear from the context. Similarly, we omit the domain in norms when no ambiguity arises.

For simplicity, except when studying the asymptotics, we set $\nu = 1$.

\subsection{Notion of solutions and Main results}
Consider the following initial data:
\begin{align}\label{ass:weakData}
    u_0 \in L_\sigma^2(\R^2), \qquad 0 \leq \rho_0 \in L^\infty(\R^2),
\end{align}
where $L^2_\sigma(\R^2)$ denotes the subspace of $L^2(\R^2)$ consisting of weakly divergence-free vector fields. Under these assumptions, one can introduce the classical notion of weak solution satisfying the energy inequality, namely a Leray--Hopf-type weak solution.
 
\begin{definition}[Weak solution]\label{def:Distsol}
A pair $(\rho,u)$ is called a weak solution of \eqref{eq:INS}
with initial data $(\rho_0,u_0)$ satisfying \eqref{ass:weakData} if the following properties hold:

\begin{itemize}
    \item[(i)] The solution satisfies the integrability conditions
    \begin{align}
        \sqrt{\rho} u \in L^\infty\big((0,\infty);L^2(\R^2)\big),\qquad
       \rho \in L^\infty\big((0,\infty)\times\R^2\big), \qquad
        \nabla u \in L^2\big((0,\infty)\times\R^2\big).
    \end{align}

    \item[(ii)] The pair $(\rho,u)$ is a distributional solution of \eqref{eq:INS}, in the following sense:
    \begin{itemize}
        \item For every $\varphi \in C_c^\infty\big([0,\infty)\times\R^2;\R\big)$,
        \begin{align}
            \int_0^\infty \!\!\int_{\R^2} \rho\,\partial_t \varphi \,dx\,dt
            + \int_0^\infty \!\!\int_{\R^2} \rho u \cdot \nabla \varphi \,dx\,dt
            = - \int_{\R^2} \rho_0 \,\varphi(0,x)\,dx.
        \end{align}
        \item For every $\varphi \in C_{c,\sigma}^\infty\big([0,\infty)\times\R^2;\R^2\big)$,
        \begin{align}
            \int_0^\infty \!\!\int_{\R^2}
            \rho\,\partial_t \varphi \cdot u
            + \rho\,u\otimes u : \nabla \varphi \,dx\,dt
            - \int_0^\infty \!\!\int_{\R^2} \nabla u : \nabla \varphi \,dx\,dt
            = - \int_{\R^2} \rho_0 u_0 \cdot \varphi(0,x)\,dx.
        \end{align}
        \item For every $\varphi \in C_c^\infty\big(\R^2;\R\big)$ and for almost every $t \in (0,\infty)$,
        \begin{align}
             \int_{\R^2} u(t) \cdot \nabla \varphi \,dx = 0.
        \end{align}
    \end{itemize}

    \item[(iii)] The energy inequality holds for every $t \geq 0$:
    \begin{equation}\label{eq:energy}
    \frac12 \int_{\R^2} \rho(t,x)\,|u(t,x)|^2 \, dx
    + \int_0^t \int_{\R^2} |\nabla u(s,x)|^2 \, dx\,ds
    \;\le\;
    \frac12 \int_{\R^2} \rho_0(x)\,|u_0(x)|^2 \, dx.
    \end{equation}
\end{itemize}
\end{definition}
Solutions of \eqref{eq:INS} with initial data in $L^2$ typically satisfy higher-order energy estimates, at the expense of a singular behavior near $t=0$, reflecting the parabolic regularization of the system. In particular, weak solutions constructed via approximation schemes (and, in a suitable sense, mild solutions associated with the heat semigroup) satisfy additional \emph{a priori} bounds. It is natural to expect that such estimates can be propagated to the limit.

Let $u$ be a weak solution in the sense of \cref{def:Distsol}. For a (sufficiently regular) function $v$, we define the \emph{material derivative} along $u$ by
\begin{align}
    \dot v := D_u v := \partial_t v + u \cdot \nabla v,
\end{align}
which represents the rate of change of $v$ along the flow generated by $u$.

For $\eta \in [0,1]$ and $t \geq 0$, we introduce the time-weighted energy functionals 
\begin{align}\label{eq:Aidef}
\begin{aligned}
 A^\eta_0(t,v) &\coloneqq \| \sqrt{\rho}\,v(t)\|_{L^2(\R^2)}^2 
     + \int_0^t \| \nabla v(s)\|_{L^2(\R^2)}^2 \, \mathrm{d}s, \\
 A^\eta_1(t,v,u) &\coloneqq t^{1-\eta} \|\nabla v(t)\|_{L^2(\R^2)}^2 
     + \int_0^t s^{1-\eta} 
     \| \nabla^2 v,\, \nabla P,\, \sqrt{\rho}\,\dot{v} \|_{L^2(\R^2)}^2 \, \mathrm{d}s, \\
 A^\eta_2(t,v,u) &\coloneqq t^{2-\eta} 
     \bigl(\| \sqrt{\rho}\,\dot{v}(t) \|_{L^2(\R^2)}^2 + \|\nabla^2 v(t)\|_{L^2(\R^2)}^2 \bigr) 
     + \int_0^t s^{2-\eta} \| \nabla \dot{v}(s) \|_{L^2(\R^2)}^2 \, \mathrm{d}s, \\
 A^\eta_3(t,v,u) &\coloneqq t^{3-\eta} \| \nabla \dot{v}(t) \|_{L^2(\R^2)}^2 
     + \int_0^t s^{3-\eta} 
     \| \nabla^2 \dot{v} \|_{L^2(\R^2)}^2 
     \, \mathrm{d}s.
\end{aligned}
\end{align}
If $u = v$, we simply write $A_i^\eta(t,u)$ for $i \in \{0,1,2,3\}$.

\begin{definition}[Immediately strong solution]\label{def:immstrongsol-INS}
A pair $(\rho,u)$ is called an \emph{immediately strong weak solution} of \eqref{eq:INS}
with initial data $(\rho_0,u_0)$ satisfying \eqref{ass:weakData} if $(\rho,u)$ is a weak
solution and there exists a constant $C_{(\rho,u)}(u)$ such that
\begin{align}\label{eq:BarC-INS}
   \sup_{i \in \{0,1,2,3\}} \sup_{t>0} A_i^0(t,u) \le C_{(\rho,u)}(u).
\end{align}
\end{definition}
For a detailed discussion of this class of solutions and related equivalent notions, we refer to \cite{CrinBaratDeNittiSkondricViolini2025}.

We now focus on the case of density patches, namely
\begin{align}\label{ass:weakData2}\tag{$L^2$-patch}
    u_0 \in L_\sigma^2(\R^2), \qquad
    \rho_0 = \mathbf{1}_D,
\end{align}
where $D \subseteq \R^2$ is a bounded Lipschitz domain. In this setting, we establish the following existence result. To the best of our knowledge, all the properties listed below are new in this level of generality.

\begin{theorem}[Existence for $L^2$ data]\label{thm:existL2}
Let $(\rho_0,u_0)$ be as in \eqref{ass:weakData2}. Then there exists an immediately strong weak solution $(\rho,u)$ to \eqref{eq:INS} with initial data $(\rho_0,u_0)$.

Moreover, the following properties hold:
\begin{itemize}
    \item[(i)] Continuity in time:
    \begin{align}
        \rho u \in C([0,\infty); L^2(\R^2)), \qquad
        \rho \in C([0,\infty); L^p_{\loc}(\R^2)), \quad \forall p \in [1,\infty).
    \end{align}
    \item[(ii)] Energy equality: the identity \eqref{eq:energy} holds with equality.
    \item[(iii)] Conservation of momentum: for almost every $t \geq 0$,
    \begin{equation}
        \int_{\mathbb{R}^2} \rho(t,x)\,u(t,x)\,dx
        =
        \int_{\mathbb{R}^2} \rho_0(x)\,u_0(x)\,dx.
    \end{equation}
    \item[(iv)] Propagation of the support: there exists a constant $C>0$ such that for almost every $t \geq 0$,
    \begin{align}
        \operatorname{supp}\rho(t) \subseteq D + B_{C\sqrt{t}}(0).
    \end{align}
\end{itemize}
\end{theorem}

We refer to the next subsection for a discussion of the result and an outline of the proof. For the moment, we only point out that properties (iii) and (iv) hold for general immediately strong weak solutions (see Section~3), while (i) and (ii) require an explicit construction of the solution (see Section~4).

Combining the estimates $A_i^0(t,u)$ is not sufficient to show that $\nabla u \in L^1(0,\infty;L^\infty(\R^2))$. As discussed above, the obstruction already appears at the level of the heat equation, which does not allow one to control critical quantities in $L^1$ in time starting from $L^2$ initial data.

For this reason, we restrict to the case of Besov initial data:
\begin{align}\label{ass:weakData3}
    u_0 \in \dot{B}^0_{2,1}(\R^2), \qquad \rho_0 = \mathbf{1}_D.
\end{align}
In this setting, we establish global existence and uniqueness of Lipschitz solutions.

\begin{theorem}[Well-posedness for $\dot{B}^0_{2,1}$ data]\label{thm:existB}
Let $(\rho_0,u_0)$ be as in \eqref{ass:weakData3}. Then there exists an immediately strong weak solution $(\rho,u)$ to \eqref{eq:INS} with initial data $(\rho_0,u_0)$ satisfying all the properties of \cref{thm:existL2}. In addition, the following quantity is bounded:
\begin{align}\label{eq:L1estimates}
   \int_0^\infty \left( \|\nabla u(t)\|_{L^\infty(\R^2)} + \|\rho \dot{u}(t)\|_{L^2(\R^2)} + t^{3/4}
   \| \nabla \dot{u}(t)\|_{L^4(\R^2)} +
  \sqrt{t}\,\|\nabla \dot{u}(t)\|_{L^2(\R^2)} \right)\,\dd t,
\end{align}
and $  \sqrt{t}\,\|\nabla u(t)\|_{L^2(\R^2)} \to 0$ as $t \to 0$.

Let $\bar{u}_0 \in L^2(\R^2)$ and let $(\bar{\rho},\bar{u})$ be an immediately strong weak solution arising from $(\rho_0,\bar{u}_0)$. Fix $t>0$. Then there exist a function $\gamma \in L^1(0,\infty)$ and a constant $C>0$ such that
\begin{align}\label{eq:stabWS}
\sup_{s\in(0,t)} \|\bar{\rho}(s)\,\left(\bar{u}(s)-u(s)\right)\|_{L^2(\R^2)}^2
\le C
\|\rho_0\,\left(\bar{u}_0-u_0\right)\|_{L^2(\R^2)}^2
\exp\!\Bigl(C\int_0^t \gamma(s)\,\dd s\Bigr).
\end{align}
In particular, if $\bar{u}_0 = u_0$, then $(\bar{\rho},\bar{u}) = (\rho,u)$ almost everywhere in $(0,\infty)\times \R^2$.
\end{theorem}

We emphasize that the estimate on each of the critical quantities in \eqref{eq:L1estimates} relies crucially on the assumption $u_0 \in \dot{B}^0_{2,1}(\R^2)$, as discussed in Section~5.  The stability estimate \eqref{eq:stabWS} requires one of the solutions to be the Besov solution constructed above, while the other is only assumed to be an arbitrary immediately strong weak solution. For this reason, the result can be interpreted as a weak--strong uniqueness statement.

In the case of more regular initial data or a more regular initial patch, a slight adaptation of the previous theorem naturally yields the following proposition, which, for the first time, covers the critical level of regularity for $C^{1,\gamma}$ patches with $\gamma \in (0,1)$.

\begin{theorem}\label{thm:gammareg}
Let $(\rho,u)$ be the Besov solution constructed in \cref{thm:existB}. Then\footnote{Here, $C_0(\R^2)$ denotes the space of bounded, continuous functions vanishing at infinity.}
\begin{align}
    \int_0^\infty \|\nabla u(t)\|_{C_0(\R^2)} \dd t < \infty.
\end{align}
In particular, for every $t \geq 0$, the associated flow satisfies $X(t)\in C^1(\R^2)$, and if $D$ is a $C^1$ domain, then $D_t$ is also of class $C^1$.

If in addition $u_0 \in \dot{B}^\gamma_{2,1}(\R^2)$ for some $\gamma \in (0,1)$, then
\begin{align}
    \int_0^\infty \|\nabla u(t)\|_{C^\gamma(\R^2)} \dd t < \infty.
\end{align}
Consequently, for every $t \geq 0$ one has $X(t)\in C^{1,\gamma}(\R^2)$, and if $D$ is a $C^{1,\gamma}$ domain, then $D_t$ remains a $C^{1,\gamma}$ domain for all $t\geq 0$.
\end{theorem}

Compared to \cite[Theorem 1.1]{GancedoGarciaJuarezLunaVelasco2025}, where the patch regularity is preserved provided $u_0 \in H^{s}(\R^2)$ with $s > \eta$, \cref{thm:gammareg} establishes the end point case $s = \eta$ in the Besov setting.

Once the global Lipschitz estimate is established, the total distortion of the flow map is finite:
\begin{equation}\label{eq:defdist}
L_\infty 
:= \sup_{t >0} \|\nabla X(t)\|_{L^\infty(\R^2)} 
\le 
\exp\!\left(
\int_0^{\infty}
\|\nabla u(s)\|_{L^\infty(\R^2)}\,\dd s
\right)
<\infty .
\end{equation}
As a consequence, the asymptotic behavior can be studied.

\begin{theorem}[Asymptotics]\label{thm:asymptotics}
    Let $(\rho,u)$ be the Besov solution constructed in \cref{thm:existB}. Then
    \begin{align}
        \|u(t)-M\|_{L^2(D_t)}
        \le
        \|u_0-M\|_{L^2(D)}\,e^{-\lambda t},
        \qquad
        \lambda \coloneqq \frac{\nu}{C_D^2 L_\infty^2},
    \end{align}
    where $C_D$ is the Poincar\'e constant of $D$, $M$ is the average momentum defined in \eqref{eq:Mdef}, and $L_\infty$ is the upper bound on the distortion defined in \eqref{eq:defdist}. 
    
    Moreover, there exists a measure-preserving map $X_\infty$ such that
    \begin{align}
       \lim_{t \to \infty} \bigl\|( X(t,\cdot ) - Mt) - X_\infty (\cdot) \bigr\|_{L^\infty(D)} = 0.
    \end{align}
    Furthermore, for every $x \in D$, we have
    \begin{align}
        X_\infty(x)
        =
        x + \int_0^\infty \bigl(u(t, X(t,x)) - M\bigr)\,\dd t,
    \end{align}
    where $u(t, X(t,x)) - M \in L^1(0,\infty)$ for every $x \in D$.
\end{theorem}
In the next subsection, we outline the main difficulties and the strategy of the proofs of \cref{thm:existL2,thm:existB,thm:asymptotics}.

\subsection{Main difficulties and ideas}

In this section, we discuss the main difficulties of the problem and outline some of the key ideas of the analysis.

The analysis relies on three key mechanisms: a localized, scaling-critical substitute for Gagliardo--Nirenberg inequalities; a gluing procedure based on atomic decompositions, which allows one to recover strong structural properties of weak solutions; and a refined analysis of Lipschitz patches, which yields local-in-time Lipschitz control of the velocity and, when combined with a Galilean transformation removing the rigid motion component, leads to a global-in-time control.

\subsubsection*{Lack of integrability outside of the patch and a dynamic Gagliardo--Nirenberg inequality.}

The first issue we address is the lack of integrability of the velocity field outside the patch.
Indeed, the energy estimates only control
\begin{align}
    \|\sqrt{\rho}\, u(t)\|_{L^2(\R^2)}
    = \|u(t)\|_{L^2(X(t,D))},
    \qquad \text{and} \qquad
    \|\nabla u\|_{L^2((0,\infty);L^2(\R^2))},
\end{align}
but they do not provide any bound on \(u(t)\) in \(L^p\) outside the patch, for any \(p \in [1,\infty]\).
In dimension three, this difficulty can be bypassed thanks to the Sobolev embedding
\begin{align}
    \|u(t)\|_{L^6(\R^3)} \;\lesssim\; \|\nabla u(t)\|_{L^2(\R^3)},
\end{align}
which yields integrability of $u(t)$ directly from the dissipation.
In dimension two, however, the critical embedding only gives
\begin{align}
    \|u(t)\|_{BMO(\R^2)} \;\lesssim\; \|\nabla u(t)\|_{L^2(\R^2)},
\end{align}
which does not imply any $L^p$-integrability at the critical level.

To overcome this obstacle, Prange and Tan~\cite{PrangeTan2023}, and later Hao, Shao, Wei and Zhang~\cite{HaoShaoWeiZhang2026}
(see in particular~\cite[Lemma A.2]{HaoShaoWeiZhang2026}), proved that
\begin{align}\label{eq:prangetN}
    \|u(t)\|_{L^2(\R^2)} \;\lesssim  
    \|\nabla u(t)\|_{L^2(\R^2)} + \|\sqrt{\rho}\,u(t)\|_{L^2(\R^2)},
\end{align}
under suitable geometric assumptions on the density.

However, this estimate is not compatible with the critical scaling, since the two terms on the right-hand side scale differently. As a consequence, it cannot be used in a critical framework.

Our goal is therefore to obtain a critical substitute for \eqref{eq:prangetN}. 
The price we pay is that the estimate holds only locally in space, namely in a $\sqrt{t}$-neighborhood of the initial patch:
\begin{align}\label{eq:Ddef}
    D_{R,t} := D + B_{R\sqrt{t}}(0).
\end{align}
The mechanism relies on two ingredients.

First, \cref{prop:ImprovedDecay} shows that \eqref{eq:Ddef} is the natural region where the dynamics takes place. Indeed, there exists $C>0$ such that
\begin{align}\label{eq:speedC}
X(t,D) = \operatorname{supp}\rho(t) \subseteq D_{C,t}.
\end{align}
Second, \cref{cor:Lpstrip} yields that for any measurable $v$, any $t>0$, any $p \in [2,\infty)$, and any $R>0$,
\begin{align}\label{eq:Lpsppeed}
t^{\frac12-\frac1p}\,
\|v\|_{L^p(D_{R,t})}
\;\lesssim_{D,R,p}\;
t^{\frac12}\,\|\nabla v\|_{L^2(\R^2)}
+
\|\rho(t)\,v\|_{L^2(\R^2)}.
\end{align}
The implicit constant deteriorates as $p \to \infty$ and $R \to \infty$.
This inequality is obtained by exploiting the transport structure of \eqref{eq:INS}, together with the Lipschitz regularity of the domain $D$ and a $\rho(t)$-weighted version of the Poincaré inequality.

Together, \eqref{eq:speedC} and \eqref{eq:Lpsppeed} show that the relevant dynamics is effectively confined to regions of the form \eqref{eq:Ddef}, where one recovers critical integrability. This provides a localized, time-weighted substitute for the classical Gagliardo--Nirenberg inequality adapted to the moving domain. Indeed, combining \eqref{eq:Lpsppeed} and \eqref{eq:speedC}, we obtain that for any measurable function $v$ and any $p\in[2,\infty)$,
\begin{align}\label{eq:GNsub}
\|v\|_{L^p(X(t,D))} \leq \|v\|_{L^p(D_{R,t})}
\;\lesssim \;
t^{\frac1p}\,\|\nabla v\|_{L^2(\R^2)}
+
t^{-\frac12+\frac1p}\,
\|v\|_{L^2(X(t,D))}.
\end{align}

This estimate should be compared with the classical Gagliardo--Nirenberg inequality on Lipschitz domains, which states that for any Lipschitz domain $\Omega \subset \R^2$ and any $p\in[2,\infty)$,
\begin{align}\label{eq:GNclass}
\|v\|_{L^p(\Omega)}
\;\lesssim \;
\|\nabla v\|_{L^2(\R^2)}^{1-\frac{2}{p}}
\|v\|_{L^2(\Omega)}^{\frac{2}{p}}.
\end{align}

If the flow $X$ were Lipschitz, then $X(t,D)$ would remain a Lipschitz domain uniformly in time, and \eqref{eq:GNclass} applied with $\Omega = X(t,D)$ would yield \eqref{eq:GNsub} via a suitable Young inequality with time weights. However, in our setting the flow is not known to be Lipschitz (we work with immediately strong solutions), so this argument is not available. Moreover, \eqref{eq:GNsub} provides control of $v$ on the larger set $D_{R,t} \supset X(t,D)$, a feature which is crucial in the analysis. Thus, \eqref{eq:GNsub} can be viewed as a substitute for \eqref{eq:GNclass} compatible with the available regularity: although weaker in structure, it preserves the correct scaling and encodes the same interpolation mechanism through time weights.

Having established this substitute, we now present some applications. First, it yields new $L^p$ estimates for quantities such as $\rho u$ and $\rho \dot{u}$. 
Since $\rho(t) = \mathbf{1}_{X(t,D)}$, choosing $p = 4$ and $v = \dot{u}(t)$ in \eqref{eq:GNsub}, we obtain
\begin{align}
\|\rho\dot{u}(t)\|_{L^4(\R^2)}
\;\lesssim \;
t^{\frac14}\,\|\nabla \dot{u}(t)\|_{L^2(\R^2)}
+
t^{-\frac14}\,
\|\rho\dot{u}(t)\|_{L^2(\R^2)}.
\end{align}
Multiplying by $t^{3/4}$, integrating in time, and using the bounds encoded in $A_1^0(t,u)$ and $A_2^0(t,u)$, we obtain
\begin{align}
t^{3/4}\,\|\rho\dot{u}(t)\|_{L^4(\R^2)} \in L^2_{\loc}[0,\infty),
\end{align}
which is a scale-invariant estimate.

Similarly, choosing $v = u(t)$ and $p = 2$, we obtain for every $R>0$
\begin{align}
\|u(t)\|_{L^2(D_{R,t})} 
\;\lesssim \;
t^{\frac12}\,\|\nabla u(t)\|_{L^2(\R^2)}
+
\|\rho(t)u(t)\|_{L^2(\R^2)}.
\end{align}
Compared to \eqref{eq:prangetN}, both terms on the right-hand side now scale in the same way.

Another crucial application of \eqref{eq:GNsub} is uniqueness. 
Consider two sufficiently regular solutions $(\rho_1,u_1)$ and $(\rho_2,u_2)$ with the same initial data $(\rho_0,u_0)$, and set $(\delta \rho, \delta u) = (\rho_1-\rho_2, u_1-u_2)$. Then, for almost every $t>0$ (see \cref{thm:relaEnergyNL}),
\begin{align}
  \frac{1}{2} \| \sqrt{\rho_1}\,\delta u(t)\|_{L^2(\R^2)}^2
  + \int_0^t \|\nabla \delta u(s)\|_{L^2(\R^2)}^2 \,ds
  \le  
  - \int_0^t\!\!\int_{\R^2} \rho_1\,\delta u\!\otimes\!\delta u : \nabla u_2
  - \int_0^t\!\!\int_{\R^2} \delta \rho\,\delta u \cdot \dot{u}_2 .
\end{align}
The first term on the right-hand side can be handled by Gronwall, since
\begin{align}
\int_0^t\!\!\int_{\R^2} \rho_1\,\delta u\!\otimes\!\delta u : \nabla u_2
\le 
\int_0^t 
\|\sqrt{\rho_1}\,\delta u(s)\|_{L^2(\R^2)}^2
\|\nabla u_2(s)\|_{L^\infty(\R^2)}\,ds,
\end{align}
and $\nabla u_2 \in L^1((0,\infty);L^\infty(\R^2))$.

The second term is more delicate, due to the lack of integrability of $\delta u$ and $\dot{u}_2$. However, thanks to \eqref{eq:speedC}, there exist $C_1,C_2>0$ such that for almost every $s>0$,
\begin{align}
 \operatorname{supp}\rho_1(s) \subseteq D_{C_1,s}, 
 \qquad  
 \operatorname{supp}\rho_2(s) \subseteq D_{C_2,s}.
\end{align}
Setting $C=\max\{C_1,C_2\}$, we obtain
\begin{align}
\int_0^t\!\!\int_{\R^2} \delta \rho\,\delta u \cdot \dot{u}_2
=
\int_0^t\!\!\int_{D_{C,s}} \delta \rho\,\delta u \cdot \dot{u}_2 .
\end{align}
In this localized region we can exploit \eqref{eq:Lpsppeed}. For instance, applying it with $v=\delta u(s)$  yields
\begin{align}
\|\delta u(s)\|_{L^2(D_{C,s})}
\;\lesssim\;
s^{\frac12}\,\|\nabla \delta u(s)\|_{L^2(\R^2)}
+
\|\sqrt{\rho_1(s)}\,\delta u(s)\|_{L^2(\R^2)}.
\end{align}
This provides the correct structure to close the Gronwall argument. 
The key point is that the dynamics can be localized in $\sqrt{t}$-neighborhoods of the initial patch, where critical integrability is recovered; combined with $\dot{H}^{-1}$-type stability for the density, this yields uniqueness.

\subsubsection*{Existence by gluing}

The second difficulty is the construction of sufficiently regular weak solutions. 
To the best of our knowledge, the only existence result for initial data $u_0 \in L^2(\R^2)$ in the presence of vacuum can be found in \cite[Section~4]{HaoShaoWeiZhang2025}, where the existence of a distributional solution to \eqref{eq:INS} is established. 

Their approach (see also \cite{DanchinMucha2019}) relies on a lifting procedure for the initial density. More precisely, for $n \in \mathbb{N}$, one considers the regularized data
\begin{align}
    \rho_0^n = \rho_0 + \frac{1}{n},
\end{align}
which is bounded away from vacuum, and constructs a unique solution $(\rho_n,u_n)$ associated with the initial data $(\rho_0^n,u_0)$. One then shows that the sequence $(\rho_n,u_n)$ is sufficiently compact to pass to the limit in the distributional formulation, yielding a solution $(\rho,u)$ of \eqref{eq:INS} with initial data $(\rho_0,u_0)$.

In \cref{prop:existenceL2}, we present a proof based on a similar lifting scheme which, thanks to the critical control provided by \eqref{eq:Lpsppeed} and the a priori bounds encoded in $A_3^0(t,u_n)$, yields a distributional solution that additionally satisfies the energy inequality and enjoys the following convergences in $\mathcal{D}'\bigl((0,\infty)\times\R^2\bigr)$:
\begin{align}\label{eq:distrconvergences}
 u_n \to  u,
\qquad
\sqrt{\rho_n} u_n \to \sqrt{\rho} u,
\qquad
\dot u_n \to \dot u,
\qquad
\sqrt{\rho_n} \dot u_n \to \sqrt{\rho} \dot u.
\end{align}
By identifying the limits in all nonlinear terms and using the lower semicontinuity of weak convergence, we show that $(\rho,u)$ is not only a Leray--Hopf solution, but in fact an immediately strong solution.

However, the immediately strong solution obtained via lifting does not possess sufficient structure to establish uniqueness and long-time asymptotics. In particular, one needs the strong continuity
\begin{align}\label{eq:strongcont}
    \sqrt{\rho} u \in C([0,\infty);L^2(\R^2)),
\end{align}
which is closely related to the energy equality. Both \eqref{eq:strongcont} and the energy equality hold at the level of the approximating sequence $(\rho_n,u_n)$, but cannot be directly passed to the limit, since one cannot improve the convergence of $\sqrt{\rho_n} u_n$ to $\sqrt{\rho} u$ in a strong enough sense (see \cref{rem:nconvergence}). A similar obstruction appears when considering mollification of the initial data.

To overcome this difficulty, we exploit a decomposition of the initial data. More precisely, we write
\begin{align}
    u_0 = \sum_{j \in \mathbb{Z}} u_{0,j},
\end{align}
where the partial sums converge strongly to $u_0$ in $L^2(\R^2)$, and each atom $u_{0,j}$ is divergence free and belongs to the subcritical space $\dot{H}^\eta(\R^2)$ for some $\eta \in (0,1)$. Such a decomposition can be obtained, for instance, via a Littlewood--Paley decomposition or through an atomic decomposition in the spirit of \cite{LionsPeetre1964} via real interpolation. In this work, we adopt the latter approach.

For each atom $u_{0,j}$, we then consider the linearized system transported by $(\rho,u)$, where $(\rho,u)$ is the solution constructed via lifting:
\begin{align}\label{eq:linearized-system}
\begin{cases}
\partial_t (\rho u_j) + \dive (\rho u \otimes u_j) + \nabla P_j = \Delta u_j,\\
\dive u_j = 0,\\
u_j(0) = u_{0,j}.
\end{cases}
\end{align}
In \cref{lem:exLS}, we prove that \eqref{eq:linearized-system} admits a solution such that $\rho u_j \in C([0,\infty);L^2(\R^2))$ and satisfies the corresponding energy equality. This relies on the additional regularity of the initial data, namely $u_{0,j} \in \dot{H}^{\eta}(\R^2)$.

We then define the partial sums
\begin{align}\label{eq:partial-sum}
u_J := \sum_{|j|\le J} u_j,
\end{align}
which, being finite sums of solutions to \eqref{eq:linearized-system}, inherit the same properties. We next study the convergence $u_J \to u$.

Since $u_J - u$ still solves \eqref{eq:linearized-system}, we have
\begin{align}\label{eq:uJu}
\partial_t (\rho (u_J - u)) + \dive (\rho u \otimes (u_J - u)) + \nabla (P_J - P) = \Delta (u_J - u),
\end{align}
and testing \eqref{eq:uJu} with $(u_J - u)$ yields, for almost every $t>0$,
\begin{align}\label{eq:energy-difference}
\frac{1}{2}\|\sqrt{\rho(t)}(u_J(t)-u(t))\|_{L^2(\R^2)}^2
+
\int_0^t \|\nabla(u_J(s)-u(s))\|_{L^2(\R^2)}^2\,\dd s
\le \frac{1}{2}
\|u_{0,J}-u_0\|_{L^2(\R^2)}^2 .
\end{align}
Since the right-hand side converges to zero as $J \to \infty$ (by \eqref{eq:atomic-decomp}), we deduce the strong convergences
\begin{align}\label{eq:strong-convergence}
\sqrt{\rho} u_J \to \sqrt{\rho} u 
\quad \text{in } C([0,\infty);L^2(\R^2)), 
\qquad 
\nabla u_J \to \nabla u
\quad \text{in } L^2((0,\infty);L^2(\R^2)).
\end{align}
These convergences are sufficient to pass to the limit in the energy equality for $u_J$, thereby obtaining the energy equality for $u$.

We emphasize that testing \eqref{eq:uJu} with $u_J - u$ is not immediate, due to the lack of decay of $u_J - u$ as $|x| \to \infty$. The corresponding justification, including admissibility of test functions, is addressed at the end of Section~3 and in Section~7.

Finally, we believe that this construction is robust and may be adapted to other contexts.

\subsubsection*{Galilean Transform and faster decays for negative Sobolev data}

The impossibility of having $\nabla u \in L^1((0,\infty);L^\infty(\R^2))$ for initial data $u_0 \in L^2(\R^2)$ can already be observed at the level of the heat equation. Indeed, one has
\begin{align}
    \|\nabla (e^{t\Delta} u_0)\|_{L^\infty(\R^2)} \lesssim \frac{1}{t}\, \|u_0\|_{L^2(\R^2)},
\end{align}
where $e^{t\Delta}$ denotes the heat semigroup. Clearly, the function $t \mapsto \frac{1}{t}$ is not integrable near $t=0$, nor at infinity. To obtain time-integrability of $\|\nabla u(t)\|_{L^\infty(\R^2)}$, the optimal space for the heat equation is $\dot{B}^{-1}_{\infty,1}(\R^2)$ (\cite[Theorem 2.34]{BCD2011}). However, $L^2(\R^2) \not\subset \dot{B}^{-1}_{\infty,1}(\R^2)$, and if $u_0 \in \dot{B}^{-1}_{\infty,1}(\R^2)$, then the corresponding solution $e^{t\Delta}u_0$ does not have finite energy.

For this reason, the natural choice is the space $\dot{B}^0_{2,1}(\R^2)$, which captures both finite energy and the critical integrability properties. To show that initial data in $\dot{B}^0_{2,1}(\R^2)$ generate Lipschitz solutions, several techniques can be used. Among them is the method of \emph{dynamic interpolation}.

We briefly sketch this idea. Let $u_0 \in \dot{B}^0_{2,1}(\R^2)$. Then, for every $\eta \in (0,1)$, there exists a sequence $(u_{0,j})_{j \in \mathbb{Z}}$ such that
\begin{align}\label{eq:atomic-decomp}
u_0 = \sum_{j \in \mathbb{Z}} u_{0,j}, \qquad 
c_j :=
2^{-j/2}\|u_{0,j}\|_{\dot{H}^{\eta}(\R^2)}
+
2^{j/2}\|u_{0,j}\|_{\dot{H}^{-\eta}(\R^2)},
\qquad \sum_{j \in \Z} c_j \leq 2\|u_0\|_{\dot{B}^0_{2,1}(\R^2)},
\end{align}
and each $u_{0,j} \in \dot{H}^{\eta}(\R^2) \cap \dot{H}^{-\eta}(\R^2)$.

For each $u_{0,j}$, we denote by $u_j$ the solution to the heat equation with initial data $u_{0,j}$. By the regularization properties of the heat semigroup, one has
\begin{align}\label{eq:assumeestimates}
    \|\nabla u_j(t)\|_{L^\infty(\R^2)} \leq C \min\Bigl\{
t^{\frac{\eta}{2}-1}\,\|u_{0,j}\|_{\dot H^\eta(\R^2)},
\;
t^{-\frac{\eta}{2}-1}\,\|u_{0,j}\|_{\dot H^{-\eta}(\R^2)}
\Bigr\}.
\end{align}
The idea is now to exploit both decay regimes in \eqref{eq:assumeestimates}: for small times we use the estimate involving the positive Sobolev norm $\|u_{0,j}\|_{\dot H^\eta(\R^2)}$, while for large times we use the one involving the negative Sobolev norm $\|u_{0,j}\|_{\dot H^{-\eta}(\R^2)}$. This is implemented by splitting the time integral:
\begin{align}
\int_0^\infty \|\nabla u_j(t)\|_{L^\infty(\R^2)} \, dt 
&= \int_0^{2^{-2j/\eta}} \|\nabla u_j(t)\|_{L^\infty(\R^2)} \, dt 
+ \int_{2^{-2j/\eta}}^\infty \|\nabla u_j(t)\|_{L^\infty(\R^2)} \, dt \\
&\lesssim \|u_{0,j}\|_{\dot H^\eta(\R^2)} \int_0^{2^{-2j/\eta}} t^{\frac{\eta}{2}-1} dt
+ \|u_{0,j}\|_{\dot H^{-\eta}(\R^2)} \int_{2^{-2j/\eta}}^\infty t^{-\frac{\eta}{2}-1} dt \\
&\lesssim 2^{-j/2}\|u_{0,j}\|_{\dot H^\eta(\R^2)}
+ 2^{j/2}\|u_{0,j}\|_{\dot H^{-\eta}(\R^2)}
= c_j.
\end{align}
Since $(c_j) \in \ell^1(\Z)$, summing over $j$ and using the linearity of the heat equation yields
\begin{align}
\int_0^\infty \|\nabla u(t)\|_{L^\infty(\R^2)} \, dt
\le \sum_{j \in \Z} \int_0^\infty \|\nabla u_j(t)\|_{L^\infty(\R^2)} \, dt
\lesssim \sum_{j \in \Z} c_j
\le 2\|u_0\|_{\dot{B}^0_{2,1}(\R^2)}.
\end{align}

Compared to other techniques for obtaining Lipschitz regularity, \emph{dynamic interpolation} can also be adapted to PDEs with nonlinearities, as in the case of the Navier--Stokes system. This was observed by Danchin in \cite{Danchin2025}. We take inspiration from his work and apply a similar idea in the vacuum case. Since $u_0 \in \dot{B}^0_{2,1}(\R^2) \subseteq L^2(\R^2)$, we can use \cref{thm:existL2} to construct a solution $(\rho,u)$ from patch data $(\rho_0=\mathbf{1}_D,u_0)$. We then apply the atomic decomposition \eqref{eq:atomic-decomp} and, for each atom, consider $u_j$ as the solution to the linear system \eqref{eq:linearized-system} with initial data $u_{0,j}$.

We then prove the estimates in \eqref{eq:assumeestimates} with a constant $C$ independent of $j \in \Z$, and by repeating the time-splitting argument described above, we obtain $\nabla u \in L^1((0,\infty);L^\infty(\R^2))$, using that the series $\sum_{j \in \Z} u_j$ converges to $u$.

Since the existence of $(\rho,u)$ is already established in \cref{thm:existL2}, the main difficulty is to prove \eqref{eq:assumeestimates}. In our case, the estimates corresponding to positive Sobolev regularity in $\dot{H}^\eta(\R^2)$ follow from the bounds on the functionals $A_i^\eta(t,u_j,u)$ established in Section~2. These estimates hold for any nonnegative initial density and do not rely on the patch structure.

It remains to study the decay of $\|\nabla u_j(t)\|_{L^\infty(\R^2)}$ in terms of negative Sobolev initial data, and here the patch structure plays a crucial role. To better explain the difficulty, let us first consider the homogeneous case $\rho_0(x)=1$. In that setting, negative Sobolev regularity is propagated, and by interpolating
\begin{align}
     \|\nabla u_j(2t)\|_{L^\infty(\R^2)}  
     \leq C\, t^{-\frac{1}{2}} \|u_j(t)\|_{\dot{H}^1(\R^2)},
     \qquad
     \|\nabla u_j(2t)\|_{L^\infty(\R^2)}  
     \leq C\, t^{-1} \|u_j(t)\|_{L^2(\R^2)},
\end{align}
one obtains
\begin{align}\label{eq:grad-Linf-interpolation}
     \|\nabla u_j(2t)\|_{L^\infty(\R^2)}  
     \leq C\, t^{-1+\frac{\eta}{2}} \|u_j(t)\|_{\dot{H}^\eta(\R^2)}.
\end{align}
Moreover, by the smoothing effect of the heat semigroup, one has
\begin{align}\label{eq:heat-negative-to-positive}
    \|u_j(t)\|_{\dot{H}^\eta(\R^2)}
    \lesssim t^{-\eta}\|u_{0,j}\|_{\dot{H}^{-\eta}(\R^2)}.
\end{align}
Inserting \eqref{eq:heat-negative-to-positive} into \eqref{eq:grad-Linf-interpolation}, one recovers the desired decay estimate. The previous interpolation argument can also be suitably adapted to the case of initial densities bounded away from zero, namely when $\rho_0(x)\ge c>0$. Unfortunately, this is no longer possible in the patch case $\rho_0=\mathbf{1}_D$. Indeed, for $t>0$ one generally has $u_j(t)\notin L^2(\R^2)$, and therefore one cannot rely on $\dot{H}^\eta(\R^2)$ estimates as in the homogeneous case. For this reason, we cannot proceed as above and must develop a different approach.

As a first step, we perform a careful time shift and prove that
\begin{align}\label{eq:andso}
     \|\nabla u_j(2t)\|_{L^\infty(\R^2)}  
     \leq C\, t^{-\frac{3}{2}}\left( \int_0^t \|\sqrt{\rho}\, u_j(s)\|_{L^2(\R^2)}^2 \dd s \right)^{\frac{1}{2}}.
\end{align}
where the presence of $\sqrt{\rho}$ is crucial. We now study the decay of the $L^2$-norm of $\sqrt{\rho}\,u_j$ in terms of negative Sobolev initial data. We proceed by duality, introducing the backward system on $[0,t]\times\R^2$:
\begin{align}
\begin{cases}\label{eq:BLSjw}
\partial_s (\rho w) + \dive (\rho u \otimes w) + \nabla Q
= - \Delta w + \rho u_j,\\
\dive w = 0,\\
w(t) = 0.
\end{cases}
\end{align}
Then we have
\begin{align}\label{eq:neg11}
\int_0^t \|\sqrt{\rho}\, u_j(s)\|_{L^2(\R^2)}^2 \dd s
= 
- \int_{\R^2} \rho_0(x)\, w(0,x)\cdot u_{0,j}(x)\,\dd x
\leq 
\|\sqrt{\rho_0}\, w(0)\|_{\dot{H}^\eta(\R^2)} 
\|\sqrt{\rho_0}\, u_{0,j}\|_{\dot{H}^{-\eta}(\R^2)}.
\end{align}
Since $D$ is a bounded Lipschitz domain, one has
\begin{align}
\mathbf{1}_D \in \mathcal{M}\bigl(\dot{H}^\eta(\R^2)\bigr)
\qquad \text{for every } \eta \in \Bigl(-\tfrac12,\tfrac12\Bigr),
\end{align}
and there exists a bounded extension operator $E:H^1(D)\to H^1(\R^2)$. Therefore, up to a multiplicative constant depending only on $D$, we have
\begin{align}\label{eq:ExD}
\begin{aligned}
\|\sqrt{\rho_0}\, w(0)\|_{\dot{H}^\eta(\R^2)}
&= \|\mathbf{1}_D E(w(0))\|_{\dot{H}^\eta(\R^2)} \lesssim \|E(w(0))\|_{\dot{H}^\eta(\R^2)} \lesssim \|E(w(0))\|_{L^2(\R^2)}^{1-\eta}\,
\|E(w(0))\|_{\dot{H}^1(\R^2)}^\eta \\
&\lesssim \|w(0)\|_{L^2(D)}^{1-\eta}\,
\|w(0)\|_{H^1(D)}^\eta \lesssim \|w(0)\|_{L^2(D)}
+ \|w(0)\|_{L^2(D)}^{1-\eta}\|\nabla w(0)\|_{L^2(D)}^\eta \\
&\lesssim \|\sqrt{\rho_0}\, w(0)\|_{L^2(\R^2)}
+ \|\sqrt{\rho_0}\, w(0)\|_{L^2(\R^2)}^{1-\eta}\|\nabla w(0)\|_{L^2(\R^2)}^\eta.
\end{aligned}
\end{align}
Arguing as for the linearized system \eqref{eq:linearized-system}, one obtains the parabolic estimate
\begin{align}
 t^{-1} \|\sqrt{\rho_0}\, w(0)\|_{L^2(\R^2)}^2
+
\|\nabla w(0)\|_{L^2(\R^2)}^2
\leq
C\int_0^t \|\sqrt{\rho}\, u_j(s)\|_{L^2(\R^2)}^2 \dd s.
\end{align}
Inserting these bounds into \eqref{eq:ExD} we obtain the following 
\begin{align}
\|\sqrt{\rho_0}\, w(0)\|_{\dot{H}^\eta(\R^2)}
\lesssim (t + t^{1-\eta})^{\frac{1}{2}}
\left( \int_0^t \|\sqrt{\rho(s)}\, u_j(s)\|_{L^2(\R^2)}^2 \dd s \right)^{\frac{1}{2}}.
\end{align}
Plugging this estimate into \eqref{eq:neg11} yields
\begin{align}\label{eq:thisone}
\left( \int_0^t \|\sqrt{\rho}\, u_j(s)\|_{L^2(\R^2)}^2 \dd s \right)^{\frac{1}{2}}  
\lesssim (t + t^{1-\eta})^{\frac{1}{2}} \| u_{0,j}\|_{\dot{H}^{-\eta}(\R^2)}.
\end{align}
We have finally obtained (see \eqref{eq:andso}), after replacing $2t$ by $t$, that
\begin{align}\label{eq:andso2}
     \|\nabla u_j(t)\|_{L^\infty(\R^2)}  
     \lesssim \left(t^{-1}  + t^{-1-\frac{\eta}{2}}\right)\| u_{0,j}\|_{\dot{H}^{-\eta}(\R^2)}.
\end{align}
Finally, for $t \leq 1$ this provides the desired decay in \eqref{eq:assumeestimates}. However, for large times one has the decay $\|\nabla u_j(t)\|_{L^\infty(\R^2)} \sim t^{-1}$, which is not integrable. As a consequence, the time-splitting argument can only be applied on bounded time intervals, and we obtain
\begin{align}
    \nabla u \in L^1_{\mathrm{loc}}([0,\infty);L^\infty(\R^2)).
\end{align}
To overcome this obstruction, inspired by the Galilean transformation, we introduce the momentum per unit mass associated with $u_j$, which is conserved in time:
\begin{align}\label{eq:Mj}
    M_j = \frac{1}{|D|}\int_{\R^2} \rho_0(x)\, u_{0,j}(x) \,\dd x =  \frac{1}{|D_t|}\int_{\R^2} \rho(t,x)\, u_{j}(t,x) \,\dd x.
\end{align}
We then consider the following Galilean change of variables:
\begin{align}
    \rho_{M_j}(t,x) := \rho(t,x + M_j t), \qquad
    u_{j,M_j}(t,x) := u_j(t,x + M_j t) - M_j, \qquad 
    \bar w_0(x) := w(0,x) + t M_j.
\end{align}
Integrating in space the momentum equation in \eqref{eq:BLSjw}, we obtain
\begin{align}\label{eq:meanwjj}
- \int_{\R^2} \rho_0(x)\, w(0,x)\,\dd x
=
\int_0^t \int_{\R^2} \rho(s,x)\, u_j(s,x)\,\dd x \,\dd s
= t\, M_j\, |D|.
\end{align}
As a consequence, $\bar w_0$ has zero average on $D$. Moreover, subtracting $t |D| |M_j|^2$ from both sides of \eqref{eq:neg11}, we obtain
\begin{align}
    \int_0^t \|\sqrt{\rho_{M_j}}\,u_{j,M_j}(s)\|_{L^2(\R^2)}^2 \,\dd s
    = - \int_{\R^2} \rho_0(x)\, u_{0,j}(x)\cdot  \bar{w}_0(x)\,\dd x.
\end{align}
Now, since $\bar{w}_0$ is mean free on $D$, we have $\|\bar{w}_0\|_{H^1(D)} \lesssim \|\nabla \bar{w}_0\|_{L^2(D)}$ by the Poincaré inequality. For this reason, \eqref{eq:ExD} can be simplified to
\begin{align}\label{eq:ExD2}
\|\sqrt{\rho_0}\,  \bar{w}_0\|_{\dot{H}^\eta(\R^2)}\lesssim \|\sqrt{\rho_0}\,  \bar{w}_0\|_{L^2(\R^2)}^{1-\eta}\|\nabla  \bar{w}_0\|_{L^2(\R^2)}^\eta.
\end{align}
In the same spirit as before one can improve \eqref{eq:thisone} to 
\begin{align}\label{eq:thisone2}
\left(  \int_0^t \|\sqrt{\rho_{M_j}}\,u_{j,M_j}(s)\|_{L^2(\R^2)}^2 \right)^{\frac{1}{2}}  
\lesssim t^{\frac{1}{2}-\frac{\eta}{2}} \| u_{0,j}\|_{\dot{H}^{-\eta}(\R^2)}.
\end{align}
This shows that, in the moving frame associated with $M_j$, the decay is compatible with the desired one and is faster than in the original variables. To conclude, it is enough to observe that the gradient is invariant under translations, and therefore the left-hand side of \eqref{eq:andso} is unchanged. Hence for al $t>0$,
\begin{align}
     \|\nabla u_j(2t)\|_{L^\infty(\R^2)}   \lesssim t^{-1-\frac{\eta}{2}} \| u_{0,j}\|_{\dot{H}^{-\eta}(\R^2)}.
\end{align}
This is precisely the desired decay estimate.

Therefore, the lack of decay in the original variables is entirely due to a rigid motion component, and does not affect the Lipschitz continuity of the velocity field.

\subsection{Organization of the paper}

In Section 2, we establish \emph{a priori} bounds for the functionals $A_i^\eta$ under $L^2$ initial data and non-negative bounded density.

In Section 3, we investigate the properties of immediately strong solutions. In particular, we prove the $\sqrt{t}$-rate of propagation of the support (point (iv) of \cref{thm:existL2}) and the conservation of momentum (point (iii) of \cref{thm:existL2}). Moreover, we introduce and discuss a dynamic Gagliardo--Nirenberg inequality.

In Section 4, we first prove existence for \eqref{eq:INS} via a lifting procedure, and then establish existence by gluing, thereby proving \cref{thm:existL2}. We also study the linearized system \eqref{eq:linearized-system} and the backward system \eqref{eq:BLSjw}, establishing both existence and suitable bounds.

In Section 5, we introduce the Galilean transform and show how to obtain global-in-time Lipschitz solutions. We further prove uniqueness and stability within this class, thereby completing the proof of \cref{thm:existB}.

In Section 6, we address more regular patches and analyze the long-time dynamics of Lipschitz patches, proving \cref{thm:gammareg} and \cref{thm:asymptotics}.

Finally, in the appendix, we discuss the derivation of relative energy identities and characterize the admissible class of test functions for the weak formulation of \eqref{eq:INS}.

\section{A priori estimates}\label{sec:energy}

In this section, we study linear systems of the form
\begin{align}
    \begin{cases}\label{eq:LSs}
        \partial_t \rho + \dive(\rho u) = 0,\\
        \rho \partial_t v + \rho (u \cdot \nabla) v + \nabla P = \Delta v,\\
        \dive v = 0,\\
        v|_{t=0} = v_0,
    \end{cases}
\end{align}
where $v_0 \in L^2_\sigma(\R^2)$ and $(\rho,u)$ if fixed. \cref{def:immstrongsol-INS} extends to the system \eqref{eq:LSs}. 

\begin{definition}\label{def:immstrongsol-LS}
Let $(\rho,u)$ be a weak solution of \eqref{eq:INS} with initial data $(\rho_0,u_0)$ satisfying \eqref{ass:weakData}. We say that $v$ is an \emph{immediately strong solution} of \eqref{eq:LSs}, advected by $(\rho,u)$, if the following properties hold:
\begin{itemize}
    \item $v$ is a distributional solution of \eqref{eq:LSs}, in the sense of
    \cref{def:Distsol}.
    \item The energy inequality holds. 
    \item There exists a constant $C_{(\rho,u)}(v)$ such that
    \begin{align}\label{eq:BarC-LS}
       \sup_{i \in \{0,1,2,3\}} \sup_{t>0} A_i^0(t,v,u) \le C_{(\rho,u)}(v).
    \end{align}
\end{itemize}
\end{definition}

In the following we assume that $(\rho,u)$ is an immediately strong solution of \eqref{eq:INS} with respect to initial data $(\rho_0,u_0)$ satisfying
\begin{align}\label{ass:datasec2}
    0 < \eps \leq \rho_0 \leq 1, \quad u_0 \in L^2_\sigma(\R^2).
\end{align}
Our goal is to obtain \emph{a priori} estimates on $v$ that are independent of $\varepsilon>0$. More precisely, we aim to derive bounds of the form
\begin{align}\label{eq:estimateslinear}
    \sup_{i \in \{0,1,2,3\}} \sup_{t>0} A_i^0(t,v,u)
    & \le C(\|\sqrt{\rho_0} u_0\|_{L^2(\R^2)}) \, \|\sqrt{\rho_0}\, v_0\|_{L^2(\R^2)}.
\end{align}
Here, $C \colon [0,\infty) \to [0,\infty)$ denotes a continuous and monotonically increasing function. Consequently, by setting $v=u$, we obtain that
\begin{align}\label{eq:estimatesnonlinear}
    \sup_{i \in \{0,1,2,3\}} \sup_{t>0} A_i^0(t,u)
    & \le C(\|\sqrt{\rho_0} u_0\|_{L^2(\R^2)}) \, \|\sqrt{\rho_0}\, u_0\|_{L^2(\R^2)}.
\end{align}
Without loss of generality, we can assume that $(\rho_0, u_0)$ and $v_0$ are smooth and therefore $(\rho,u)$ and $v$ are smooth solution of \eqref{eq:INS} and \eqref{eq:LSs}, respectively.

\begin{theorem}\label{thm:sect2}
Let $(\rho,u)$ be a solution of \eqref{eq:INS} with initial data $(\rho_0,u_0)$ satisfying \eqref{ass:datasec2}, and let $v$ be a solution of \eqref{eq:LSs}. Then there exists a constant
\begin{align}\label{eq:constofsect2}
    C_u = C(\|\sqrt{\rho_0} u_0\|_{L^2(\R^2)}),
\end{align}
which is non decreasing and independent of $\eps > 0$, such that
\begin{align}
    \sup_{i \in \{0,1,2,3\}} \sup_{t>0} A_i^0(t,v,u)
    & \le C_u \, \|\sqrt{\rho_0}\, v_0\|_{L^2(\R^2)}, \label{eq:A0}\\
    \sup_{i \in \{1,2,3\}} \sup_{t>0} A_i^s(t,v,u)
    & \le C_u \, \|v_0\|_{\dot H^s(\R^2)}\label{eq:As}.
\end{align}
In particular, if $v=u$, then 
\begin{align}\label{eq:A0u}
    \sup_{i \in \{0,1,2,3\}} \sup_{t>0} A_i^0(t,u)
    \le C_u \, \|\sqrt{\rho_0}\, u_0\|_{L^2(\R^2)}.
\end{align}
\end{theorem}
\begin{remark}
    Notice that, in order to prove \eqref{eq:As}, it is sufficient to establish \eqref{eq:A0} and \eqref{eq:As} in the case $s=1$. Indeed, interpolating these two estimates yields the result for any $s \in (0,1)$. We therefore present only the proof of \eqref{eq:A0}, since the estimate \eqref{eq:As} with $s=1$ follows in the same way, up to a straightforward modification of the time weights. We refer, for instance, to \cite{Danchin2025}.
\end{remark}

The proof of \eqref{eq:A0} and \eqref{eq:A0u} will be carried out in the next subsections simultaneously for each $A_i$.

\subsection{\texorpdfstring{Proof of \cref{thm:sect2}}{Proof of Theorem}}

\subsubsection{\texorpdfstring{Estimates for $A_0$}{Estimates for A0}}

It is well-known that smooth solutions of \eqref{eq:INS} and \eqref{eq:LSs} satisfy the energy equality and this immediately gives
\begin{align}
    \sup_{t>0} A^0_0(t,v,u) \leq \|\sqrt{\rho_0} v_0\|_{2}^2.
\end{align}

\subsubsection{\texorpdfstring{Estimates for $A_1$}{Estimates for A1}}

First, notice that by standard Stokes-type estimates, since $\rho$ is bounded, we have
\begin{align} \label{eq:StokesEst}
    \|\nabla^2 v\|_2^2 + \|\nabla P\|_2^2 \leq C_S \|\sqrt{\rho} \dot{v}\|_2^2.
\end{align}
 We test the momentum equation in \eqref{eq:LSs} with \( \dot{v} \) and obtain
\begin{align}\label{eq:momA1}
    \|\sqrt{\rho} \dot{v}\|_2^2 = \ip{\Delta v - \nabla P, \dot{v}} = - \frac{1}{2}\frac{d}{dt}\norm{\nabla v}_2^2 - \ip{\nabla v, \nabla(u \cdot \nabla v)} + \ip{P, \dive \dot{v}}.
\end{align}
Integration by parts and incompressibility yield
\begin{align}
     -  \ip{\nabla v, \nabla(u \cdot \nabla)v} 
     &= - \frac{1}{2} \ip{u, \nabla |\nabla v|^2} - \int_{\R^2} \tr(\nabla v \nabla u (\nabla v)^T) \dd x \eqqcolon 0+ L(t)
\end{align}
Using Gagliardo--Nirenberg, the Stokes estimates \eqref{eq:StokesEst}, and Young's inequality, we obtain
\begin{align}
\label{eq:Lfirst}
\begin{aligned}
    |L(t)| &\leq C_{GN}\|\nabla v\|_2 \|\nabla^2 v\|_2 \|\nabla u\|_2
    \leq C_{GN} C_S\|\nabla v\|_2 \|\sqrt{\rho}\dot v\|_2 \|\nabla u\|_2
    \\ &\leq \frac{1}{8} \norm{\sqrt{\rho} \dot{v}}_2^2 + C\, \norm{\nabla v}_2^2 \norm{\nabla u}_2^2,
    \end{aligned}
\end{align}
where $C$ depends only on $C_{GN}$ and $C_S$.

For the pressure term, we first use \cref{prop:divH} and then proceed exactly as for $L$:
\begin{align}\label{eq:boundphi}
   \ip{P, \dive \dot{v}}   
   &\leq \tilde{C} \norm{\nabla P}_2 \norm{\nabla v}_2 \norm{\nabla u}_2
   \leq \frac{1}{8} \norm{\sqrt{\rho}\dot{v}}_2^2 + C \norm{\nabla v}_2^2 \norm{\nabla u}_2^2.
\end{align}
Combining the above estimates with \eqref{eq:momA1} and absorbing $\norm{\sqrt{\rho}\dot{v}}_2$ into the left-hand side yields
\begin{align} \label{eq:conl1}
    \frac{d}{dt} \|\nabla v\|_2^2 + \|\sqrt{\rho} \dot{v}\|_2^2 
    \lesssim \|\nabla v\|_2^2 \|\nabla u\|_2^2.
\end{align}
To prove \eqref{eq:A0} for $i=1$, we multiply \eqref{eq:conl1} by \( t \):
\begin{align}\label{eq:conl2}
    \frac{d}{dt} \left( t \|\nabla v\|_2^2 \right) + t\|\sqrt{ \rho} \dot{v}\|_2^2 
    \lesssim \|\nabla v\|_2^2 + t \|\nabla v\|_2^2 \|\nabla u\|_2^2.
\end{align}
Applying Grönwall's inequality and using the energy inequality for $(\rho,u)$ and $v$ leads to
\begin{align}
    t \| \nabla v(t)\|_2^2  \lesssim \int_0^t \|\nabla v(s)\|_2^2 \, \dd s 
    \exp\left( \int_0^t \|\nabla u(s)\|_2^2 \, \dd s \right) \lesssim \exp\left(2 \norm{\sqrt{\rho_0} u_0}_2^2 \right) \norm{\sqrt{\rho_0} v_0}_2^2.
\end{align}
Integrating in time \eqref{eq:conl2} and using the previous estimate, we see that
\begin{align}
   \int_0^t s \|\sqrt{ \rho} \dot{v}\|_2^2  \dd s 
    & \lesssim \int_0^t \|\nabla v\|_2^2 \dd s + \int_0^t s \|\nabla v\|_2^2 \|\nabla u\|_2^2 \dd s  \\  & \lesssim \| \sqrt{\rho_0} v_0 \|_2^2
     + \exp\left(2 \norm{\sqrt{\rho_0} u_0}_2^2 \right) \norm{\sqrt{\rho_0} v_0}_2^2 \norm{\sqrt{\rho_0} u_0}_2^2.
\end{align}
Setting
\begin{align}\label{eq:C0def}
    C^0_u \coloneqq \exp\left(4 \norm{\sqrt{\rho_0} u_0}_2^2 \right),
\end{align}
and using the Stokes estimates \eqref{eq:StokesEst} we conclude
\begin{align}\label{eq:A1H2}
        t \|\nabla v (t)\|_{2}^2 + \int_0^t s \norm{\nabla P(s), \nabla^2 v (s), \sqrt{\rho} \dot{v}(s)}_{2}^2 \dd s 
        \lesssim C^0_u \|\sqrt{\rho_0}v_0\|_{2}^2. 
\end{align}
Since $C^0_u$ is non-decreasing in $\|\sqrt{\rho_0}\,u_0\|_{2}$, 
\eqref{eq:A1H2} yields \eqref{eq:A0} in the case $i=1$. Moreover, taking $v=u$ we obtain \eqref{eq:A0u} in the case $i=1$.

\subsubsection{\texorpdfstring{Estimates for $A_2$}{Estimates for A2}}
To estimate $A_2$, it is necessary to differentiate \eqref{eq:INS} and \eqref{eq:LSs} with respect to the material derivative $D$. Here, for every smooth vector field $w \colon [0,\infty) \times \R^2 \to \R^2$, the material derivative of $w$ is given by
\begin{align}
    D w = D_u w = \partial_t w + (u \cdot \nabla) w.
\end{align}
It is clear that $D$ does not commute with $\nabla$, $\dive$, and $\Delta$. However, we have the following identities:
\begin{align}
    D(\nabla w) &= \nabla \dot w - \nabla w\, \nabla u, \label{eq:grad}\\
    D(\dive w) &= \dive \dot w - \tr(\nabla u\, \nabla w), \label{eq:div}\\
    D(\Delta w) & = \Delta \dot w
    - \partial_k ((\partial_k u \cdot \nabla)w) - (\partial_k u\cdot \nabla) \partial_k w. \label{eq:lap}
\end{align}
Note that, since $\rho$ satisfies the transport equation with transport field $u$, we have $D_u (\rho \dot{v}) = \rho \ddot{v}$.

We now turn to the estimate. By applying \( D = D_u \) to \eqref{eq:LSs}, multiplying the result by $\dot{v}$, and integrating, we deduce that
\begin{align}\label{eq:mom2}
    \langle \rho \ddot{v}, \dot{v} \rangle = \langle D \Delta v - D \nabla P, \dot{v} \rangle.
\end{align}
The left-hand side simplifies as follows:
\begin{align}
    \langle \rho \ddot{v}, \dot{v} \rangle 
    = \langle \rho [\partial_t \dot{v} + (u \cdot \nabla) \dot{v}], \dot{v} \rangle
    = \frac{1}{2} \int \rho \partial_t |\dot{v}|^2 + \rho u \cdot \nabla |\dot{v}|^2 
    = \frac{1}{2} \int \rho \partial_t |\dot{v}|^2 - \dive(\rho u) |\dot{v}|^2.
\end{align}
Using that $- \dive(\rho u) = \partial_t \rho$ and the product rule for the time derivative, we obtain
\begin{align}
    \langle \rho \ddot{v}, \dot{v} \rangle = \frac{1}{2} \int \rho \partial_t |\dot{v}|^2 + \partial_t \rho |\dot{v}|^2 = \frac{1}{2} \frac{d}{dt} \| \sqrt{\rho} \dot{v} \|_2^2.
\end{align}
To estimate the right-hand side of \eqref{eq:mom2}, we treat the Laplacian and pressure terms separately.

For the Laplacian term, we use identity \eqref{eq:lap} to commute the material derivative with the Laplacian and then integrate by parts:
\begin{align}\label{eq:LapA2}
\begin{aligned}
    \langle D \Delta v, \dot{v} \rangle 
    &= \ip{\Delta \dot{v}, \dot{v}} - \ip{\partial_k(\partial_k u_j \partial_j v_i), \dot{v}_i} - \ip{\partial_k u_j \partial_j \partial_k v_i, \dot{v}_i}\\
    &= - \frac{1}{2} \| \nabla \dot{v} \|_2^2 + \ip{\partial_k u_j \partial_j v_i, \partial_k \dot{v}_i} + \ip{\partial_k u_j \partial_k v_i, \partial_j \dot{v}_i}\\
    &= - \frac{1}{2} \| \nabla \dot{v} \|_2^2 + \int_{\R^2} \tr(\nabla v \nabla u (\nabla \dot{v})^T) + \int_{\R^2} \tr(\nabla \dot{v} \nabla u (\nabla v)^T) \eqqcolon - \frac{1}{2} \| \nabla \dot{v} \|_2^2 + L(t).
\end{aligned}
\end{align}
For the pressure term, we apply Leibniz’s rule in time and integrate by parts:
\begin{align}\label{eq:formulaforpressure}
    \langle D \nabla P, \dot{v} \rangle 
    = & \frac{d}{dt} \langle \nabla P, \dot{v} \rangle - \langle \nabla P, \ddot{v} \rangle 
    \eqqcolon -\frac{d}{dt} \langle P, \dive \dot{v} \rangle + Q(t).
\end{align}

Collecting all terms, we obtain
\begin{align}\label{eq:coll}
    \phi'(t) \coloneqq \frac{d}{dt} \left( \| \sqrt{\rho} \dot{v} \|_2^2 - \langle P, \dive \dot{v} \rangle \right) + \frac{1}{2} \| \nabla \dot{v} \|_2^2 
    & = L(t) + Q(t).
\end{align}
We now estimate each term on the right-hand side of \eqref{eq:coll}. To this end, we introduce 
\begin{align}\label{eq:C1def}
    C^1_u \coloneqq  \exp\left(9 \norm{\sqrt{\rho_0} u_0}_2^2 \right).
\end{align}
For the  laplacian term \(L(t)\), with the same combination of \eqref{eq:Lfirst}:
\begin{align}\label{eq:Lsecond}
    \abs{L(t)} 
    \leq \underbrace{C \| \nabla v \|_2 \| \nabla^2 v \|_2  \| \nabla u \|_2 \| \nabla^2 u \|_2 }_{\eqqcolon L_1(t) }+ \frac{1}{8} \| \nabla \dot{v} \|_2^2.
\end{align}
Thanks to \eqref{eq:A1H2}, we obtain:
\begin{align}
    \int_0^t s^2 \abs{L_1(s)} \dd s & \lesssim \sup_{s\in (0,t)} \sqrt{s}\norm{\nabla v}_2  
    \sup_{s\in (0,t)} \sqrt{s}\norm{\nabla u}_2  \left( \int_0^t s \norm{\nabla^2 v}_2^2  \dd s \right)^{\frac{1}{2}} 
    \left( \int_0^t s \norm{\nabla^2 u}_2^2  \dd s \right)^{\frac{1}{2}}\\
    & \lesssim (C^0_u)^2 \norm{\sqrt{\rho_0}u_0}_2^2 \norm{\sqrt{\rho_0}v_0}_2^2 \leq C^1_u\norm{\sqrt{\rho_0}v_0}_2^2
\end{align} 
In order to estimate the pressure term \(Q\), we use \cref{prop:divH} and Young's inequality:
\begin{align}\label{eq:Qsecond}
    Q(t) & \leq \underbrace{C\| \nabla P \|_2  \| \nabla u \|_2 \| \nabla \dot{v} \|_2 }_{ \eqqcolon Q_1(t) } + \frac{1}{8} \|\nabla \dot{v} \|^2_2  + \underbrace{C\| \nabla P \|_2  \| \nabla v \|_2 \| \nabla \dot{u} \|_2}_{ \eqqcolon Q_2(t) }.
\end{align}
$Q_1(t)$ behaves as $L_1(t)$:
\begin{align}
    \int_0^t s^2 Q_1(s) \dd s \lesssim  C^1_u \norm{\sqrt{\rho_0}v_0}_2^2.
\end{align}
Inserting \eqref{eq:Lsecond} and \eqref{eq:Qsecond} in \eqref{eq:coll}, we obtain
\begin{align}\label{eq:coll2}
    \phi^\prime(t) + \| \nabla \dot{v} \|_2^2 \lesssim L_1(t) + Q_1(t) + Q_2(t).
\end{align}
Multiplying \eqref{eq:coll2} by \( s^2 \) and integrating in time, thanks to the estimates on $L_1$ and $Q_1$ we get
\begin{align}\label{eq:coll3}
    t^2 \phi(t) +  \int_0^t s^2 \| \nabla \dot{v} \|_2^2 \, ds 
    &\lesssim  \int_0^t s \phi(s) \, ds +  C^1_u \norm{\sqrt{\rho_0}v_0}_2^2+  \int_0^t s^2   Q_2(s)\dd s.
\end{align}
Notice that, using \eqref{eq:boundphi}, we obtain an upper bound for \( \phi(t) \):
\begin{align}
    \phi(t) = \| \sqrt{\rho} \dot{v} \|_2^2 - \langle P, \dive \dot{v} \rangle \leq \frac{5}{4}\norm{\sqrt{\rho} \dot{v}}_2^2 + C \| \nabla u \|^2_2 \| \nabla v \|^2_2.
\end{align}
Hence, by \eqref{eq:A1H2},
\begin{align}\label{eq:upperboundphi}
    \int_0^t s \phi(s) \, ds & \lesssim  C^0_u  \|\sqrt{\rho_0} v_0\|_{2}^2  +C^0_u    \|\sqrt{\rho_0} u_0\|_{2}^2  \|\sqrt{\rho_0} v_0\|_{2}^2   \leq C^1_u   \|\sqrt{\rho_0} v_0\|_{2}^2.
\end{align}
Again by \eqref{eq:A1H2}, we obtain a lower bound for \( \phi \):
\begin{align}
     \phi(t) 
    & \geq  \| \sqrt{\rho} \dot{v} \|_2^2 - \frac{1}{4} \| \sqrt{\rho} \dot{v} \|_2^2 - C \| \nabla v \|_2^2  \| \nabla u \|_2^2 \gtrsim \frac{3}{4} \| \sqrt{\rho} \dot{v} \|_2^2 -C \| \nabla v \|_2^2  \| \nabla u \|_2^2
\end{align}
Then, by \eqref{eq:A1H2},
\begin{align}\label{eq:revbound}
    t^2 \phi(t) &    \gtrsim t^2 \| \sqrt{\rho} \dot{v} \|_2^2 - (C^0_u)^2 \norm{\sqrt{\rho_0} v_0}_2^2 \norm{\sqrt{\rho_0} u_0}_2^2   \geq  t^2 \| \sqrt{\rho} \dot{v} \|_2^2- C^1_u \norm{\sqrt{\rho_0} v_0}_2^2.
\end{align}
By inserting \eqref{eq:upperboundphi} and \eqref{eq:revbound} into \eqref{eq:coll3}, we obtain
\begin{align}\label{eq:estimateA2bad}
    t^2 \| \sqrt{\rho} \dot{v} \|_2^2 + \int_0^t s^2 \| \nabla \dot{v} \|_2^2 \, ds \lesssim C^1_u \norm{\sqrt{\rho_0} v_0}_2^2 + \int_0^t s^2 \| \nabla P \|_2  \| \nabla v \|_2 \| \nabla \dot{u} \|_2 \dd s.
\end{align}
By setting $u=v$ in \eqref{eq:estimateA2bad} and applying Young's inequality, we get
\begin{align}
    \int_0^t s^2 \| \nabla \dot{u} \|_2^2 \, ds 
    & \leq C C^1_u \norm{\sqrt{\rho_0} u_0}_2^2 + C \int_0^t s^2 \| \nabla P \|_2^2  \| \nabla u \|_2^2\dd s + \frac{1}{2}\int_0^t s^2 \| \nabla \dot{u} \|_2^2\dd s \\
    & \leq 2 C C^1_u \norm{\sqrt{\rho_0} u_0}_2^2 + \frac{1}{2}\int_0^t s^2 \| \nabla \dot{u} \|_2^2\dd s
\end{align}
and therefore
\begin{align}
     \int_0^t s^2 \| \nabla \dot{u} \|_2^2 \, ds \lesssim C^1_u \norm{\sqrt{\rho_0} u_0}_2^2.
\end{align}
By inserting the latter estimate into \eqref{eq:estimateA2bad}, we deduce that
\begin{align}\label{eq:estimateA2conclusion}
\begin{aligned}
    t^2 \| \sqrt{\rho} \dot{v} \|_2^2 &+ \frac{1}{2} \int_0^t s^2 \| \nabla \dot{v} \|_2^2 \, ds \lesssim C^1_u \norm{\sqrt{\rho_0} v_0}_2^2  + \left(  \int_0^t s^2 \| \nabla P \|^2_2  \| \nabla v \|_2^2 \dd s\right)^\frac{1}{2} \left(  \int_0^t s^2  \| \nabla \dot{u} \|^2_2 \dd s \right)^\frac{1}{2} \\ & \lesssim C^1_u \norm{\sqrt{\rho_0} v_0}_2^2  + \left( (C^0_u)^2  \norm{\sqrt{\rho_0} v_0}_2^4   \right)^\frac{1}{2} \left(  C^1_u \norm{\sqrt{\rho_0} u_0}_2^2 \right)^\frac{1}{2} \lesssim C^1_u \norm{\sqrt{\rho_0} v_0}_2^2   
\end{aligned}
\end{align}
which, thanks to Stokes estimates \eqref{eq:StokesEst}, completes the proof of \eqref{eq:A0} for $i=2$.

Recall that by \cref{prop:Linfty}
\begin{align}\label{eq:L2Lip}
    \|\nabla v(t)\|^2_{\infty} \lesssim \|\nabla^2 v(t)\|^2_{2} + \|\nabla v(t)\|_{2} \|\nabla \dot v(t)\|_{2}.
\end{align}
By combining \eqref{eq:L2Lip} with \eqref{eq:estimateA2conclusion}, we conclude by Hölder's inequality
\begin{align}
\begin{aligned}
    \int_0^t s\,\|\nabla v\|_{\infty}^2 \dd s
    \leq & \int_0^t s\,\|\nabla^2 v\|_2^2 \dd s + \left(\int_0^t \|\nabla v\|_2^2 \dd s\right)^{1/2} \left(\int_0^t s^2\,\|\nabla \dot v\|_2^2 \dd s\right)^{1/2}\\
    \lesssim & C_u^0\,\|\sqrt{\rho_0}\,v_0\|_2^2 + (C^1_u)^\frac{1}{2} \|\sqrt{\rho_0}\,v_0\|^2_2 \lesssim C_u^1\,\|\sqrt{\rho_0}\,v_0\|_2^2.
\end{aligned}
\end{align}
Since the right-hand side does not depend on $t$, we conclude
\begin{align}\label{eq:almostlipv}
     \int_0^\infty t\,\|\nabla v\|_{\infty}^2 \dd t
   \lesssim C_u^1\,\|\sqrt{\rho_0}\,v_0\|_2^2.
\end{align}
With the same arguments as above, we get
\begin{align}\label{eq:almostlipu}
\begin{aligned}
    \int_0^\infty t\,\|\nabla u\|_{\infty}^2 \dd t
    \lesssim & C_u^1\,\|\sqrt{\rho_0}\,u_0\|_2^2.
\end{aligned}
\end{align}

\subsubsection{\texorpdfstring{Estimates for $A_3$}{Estimates for A3}}

Similarly to what was done in the estimate for \( A_2 \), we commute the material derivative with the differential operators following \eqref{eq:grad} and \eqref{eq:lap} to obtain:
\begin{align}\label{eq:momA3}
    \rho \ddot{v} + \nabla \dot{P} - \Delta \dot{v} 
    = \nabla P\, \nabla u- \partial_k ((\partial_k u \cdot \nabla)v) - (\partial_k u\cdot \nabla) \partial_k v \eqqcolon F.
\end{align}
Using \cref{lem:Genstokes}, we deduce:
\begin{align}
    \|\nabla^2 \dot{v}, \nabla \dot{P}\|_2^2 
    &\leq \norm{\rho \ddot{v}}_2^2 + \norm{F}_2^2 + \norm{\nabla \dive \dot{v}}_2^2 \\
    &\leq \norm{\rho \ddot{v}}_2^2 + C\left( \norm{\nabla u}_\infty^2 \norm{\nabla P}_2^2 + \norm{\nabla u}_\infty^2 \norm{\nabla^2 v}_2^2 +  \norm{\nabla v}_\infty^2 \norm{\nabla^2 u}_2^2 \right)\\
    &\leq\norm{\sqrt{\rho} \ddot{v}}_2^2 + C \left( \norm{\nabla u}_\infty^2 \norm{\sqrt{\rho} \dot{v}}_2^2 +  \norm{\nabla v}_\infty^2 \norm{\nabla^2 u}_2^2 \right).
\end{align}
Here, we have used \eqref{eq:StokesEst} in the last inequality. Then by \eqref{eq:estimateA2conclusion}, \eqref{eq:almostlipv} and \eqref{eq:almostlipu}, we deduce that
\begin{align}
    \int_0^t s^3 & \Bigl(\norm{\nabla u}_\infty^2 \norm{\sqrt{\rho} \dot{v}}_2^2 + \norm{\nabla v}_\infty^2 \norm{\nabla^2 u}_2^2 \Bigr) \dd s \\
    &\lesssim \sup_{s\in(0,t)} \left(s^2 \norm{\sqrt{\rho} \dot{v}}_2^2 \right) \int_0^t s \norm{\nabla u}_\infty^2 \dd s + \sup_{s\in(0,t)} \left( s^2 \norm{\nabla^2 u }_2^2 \right) \int_0^t s \norm{\nabla v}_\infty^2 \dd s \\
    & \lesssim (C^1_u)^2 \norm{\sqrt{\rho_0} u_0}_2^2  \norm{\sqrt{\rho_0} v_0}_2^2.
\end{align}
Introducing
\begin{align}\label{eq:C2def}
    C^2_u \coloneqq  \exp\left(19 \norm{\sqrt{\rho_0} u_0}_2^2 \right)
\end{align}
we conclude that
\begin{align}\label{eq:stokesA3}
    \int_0^t s^3 \|\nabla^2 \dot{v}, \nabla \dot{P}\|_2^2 \dd s \leq \int_0^t s^3 \norm{\rho \ddot{v}}_2^2 \dd s +C  C^2_u \norm{\sqrt{\rho_0}v_0}_2^2.
\end{align}
Now we estimate $A_3$. By testing \eqref{eq:momA3} with $\ddot{v}$, we get
\begin{align}\label{eq:momdd}
    \norm{\sqrt{\rho}\ddot{v}}_2^2 +  \ip{D \nabla P, \ddot{v}} =  \ip{ D \Delta v, \ddot{v}}.
\end{align}
For the term involving the Laplacian, similarly to \eqref{eq:LapA2}
\begin{align}\label{eq:lapA3}
\begin{aligned}
     \ip{ D \Delta v, \ddot{v}} =  & \ip{  \Delta \dot{v},\partial_t\dot{v}} + \ip{\Delta \dot{v},u \cdot \nabla \dot{v}}
    + \int_{\R^2} \tr(\nabla v \nabla u (\nabla \ddot{v})^T) + \int_{\R^2} \tr(\nabla \ddot{v} \nabla u (\nabla v)^T)\\
    = & -\frac{1}{2} \frac{d}{d t}\|\nabla \dot{v}\|^2_2 + \ip{\Delta \dot{v},u \cdot \nabla \dot{v}}
    + \int_{\R^2} \tr(\nabla v \nabla u (\nabla \ddot{v})^T) + \int_{\R^2} \tr(\nabla \ddot{v} \nabla u (\nabla v)^T)\\
    =: & -\frac{1}{2} \frac{d}{d t}\|\nabla \dot{v}\|^2_2 + L_1(t) + L_2(t) + L_3(t).
     \end{aligned}
\end{align}
For $L_1$, we integrate by parts and use the incompressibility
\begin{align}
    L_1(t) = & - \ip{u_j \partial_j \partial_k \dot{v}_i, \partial_k  \dot{v}_i} - \ip{\partial_k u_j \partial_j \dot{v}_i, \partial_k  \dot{v}_i}=  0 - \int_{\R^2} \tr(\nabla \dot{v} \nabla u (\nabla \dot{v})^T).
\end{align}
Using \eqref{eq:stokesA3} we deduce that 
\begin{align}\label{eq:L1A3}
\begin{aligned}
    \int_0^t s^3 \abs{L_1(s)} \leq & \int_0^t s^3 \|\nabla u\|_2 \|\nabla \dot{v}\|_4^2 \leq C\int_0^t s^3 \|\nabla u\|_2^2 \|\nabla \dot{v}\|_2^2 + \frac{1}{8}\int_0^t s^3 \|\nabla^2 \dot{v}\|_2^2\\
    \leq & CC^0_u \norm{\sqrt{\rho_0} u_0}_2^2  C^1_u\norm{\sqrt{\rho_0} v_0}_2^2 + \frac{1}{8}\int_0^t s^3 \norm{\rho \ddot{v}}_2^2 + C C^2_u \norm{\sqrt{\rho_0}v_0}_2^2\\
    \leq & C  C^2_u \norm{\sqrt{\rho_0}v_0}_2^2 + \frac{1}{8}\int_0^t s^3 \norm{\rho \ddot{v}}_2^2.
\end{aligned}
\end{align}
For $L_2,$ we integrate by parts
\begin{align}
    L_2(t) & = \int\tr( \nabla v \nabla u D (\nabla^T \dot{v})) + \int  \tr( \nabla v \nabla u \nabla \dot{v}^T \nabla u)  \\ 
    & = \frac{d}{dt} \int\tr( \nabla v \nabla u  \nabla^T \dot{v})  \underbrace{-\int \tr(D(\nabla v \nabla u) \nabla^T \dot{v}) + \int  \tr( \nabla v \nabla u \nabla \dot{v}^T \nabla u)}_{L_{21}(t)}.
\end{align}
We use the product rule and \eqref{eq:grad} and obtain
\begin{align}
    \norm{D(\nabla v \nabla u)}_{4\slash 3} \leq & \norm{\nabla uD\nabla v }_{4\slash 3} + \norm{\nabla v D\nabla u}_{4\slash 3} \leq \norm{D\nabla v}_2 \norm{\nabla u}_4 + \norm{D\nabla u}_2 \norm{\nabla v}_4\\
    \leq & \norm{\nabla u}_4 \norm{\nabla \dot{v}}_2 + \norm{\nabla v}_4 \norm{\nabla u}^2_4 + \norm{\nabla v}_4 \norm{\nabla \dot{u}}_2.
\end{align}
Thus, we can estimate
\begin{align}\label{eq:L21A3}
    \begin{aligned}
        L_{21}(t) \leq & \norm{\nabla \dot{v}}_4 \norm{D(\nabla v \nabla u)}_{4\slash 3} + \norm{\nabla \dot{v}}_4\norm{\nabla v}_4 \norm{\nabla u}^2_4 \\
        \lesssim & \norm{\nabla \dot{v}}_4 \left( \norm{\nabla u}_4 \norm{\nabla \dot{v}}_2 + \norm{\nabla v}_4 \norm{\nabla u}^2_4 + \norm{\nabla v}_4 \norm{\nabla \dot{u}}_2 \right).
    \end{aligned}
\end{align}
By exploiting the same argument as in \eqref{eq:L1A3}, we infer that
\begin{align}
\int_0^t s^3 \abs{L_{21}(s)}  \dd s
\leq \frac{1}{8} \int_0^t s^3 \|\sqrt{\rho} \ddot{v}\|_2^2  \dd s
+ C C_u^2 \|\sqrt{\rho_0} v_0\|_2^2.
\end{align}
Similarly, we have:
\begin{align}
   L_3(t) 
   = \frac{d}{dt} \int \tr( \nabla u\, \nabla\dot{v}\, \nabla^T v) + L_{31}(t),
\end{align}
where $L_{31}$ satisfies
\begin{align}
    \int_0^t s^3 \abs{L_{31}(s)} \dd s \leq \frac{1}{8} \int_0^t s^3 \|\sqrt{\rho} \ddot{v}\|^2_2 + C C_u^2 \norm{\sqrt{\rho_0} v_0}^2_2.
\end{align}
To simplify the pressure term we integrate by parts
\begin{align}\label{eq:presA3}
    \ip{D \nabla P, \ddot{v}} = \frac{d}{d t}\ip{\nabla P, \ddot{v}} - \ip{\nabla P, \dddot{v}} = \frac{d}{d t}\ip{\nabla P, \ddot{v}} + \ip{P, \dive \dddot{v}}.
\end{align}
Then, by \cref{prop:divH},
\begin{align}
    \ip{P, \dive \dddot{v}} \lesssim & \frac{d}{d t} \ip{P ,\tr(3\nabla u \nabla \dot{v}+\nabla v \nabla \dot{u})}  + \|\nabla P\|_2 \|\nabla \dot{u}\|_2\|\nabla \dot{v}\|_2 + \|\nabla \dot{P}\|_2 \left(\|\nabla v\|_2 \|\nabla \dot{u}\|_2 + \|\nabla u\|_2 \|\nabla \dot{v}\|_2 \right)\\
    \eqqcolon & \frac{d}{d t} \ip{P ,\tr(3\nabla u \nabla \dot{v}+\nabla v \nabla \dot{u})} + Q_1(t).
\end{align}
For $Q_1,$ similarly to \eqref{eq:L1A3}, we have
\begin{align}
\int_0^t s^3 |Q_1(s)| \, ds
\leq {}& \sup_{s>0} \bigl(s \|\nabla P\|_2\bigr)
\left(\int_0^t s^2 \|\nabla \dot{u}\|_2^2 \, ds\right)^{1/2}
\left(\int_0^t s^2 \|\nabla \dot{v}\|_2^2 \, ds\right)^{1/2} \\
&+ \frac{1}{8} \int_0^t s^3 \|\nabla \dot{P}\|_2^2 \, ds + \sup_{s>0} \bigl(s \|\nabla v\|_2^2\bigr)
\int_0^t s^2 \|\nabla \dot{u}\|_2^2 \, ds + \sup_{s>0} \bigl(s \|\nabla u\|_2^2\bigr)
\int_0^t s^2 \|\nabla \dot{v}\|_2^2 \, ds \\
\leq {}& C C_u^2 \|\sqrt{\rho_0} v_0\|_2^2
+ \frac{1}{8} \int_0^t s^3 \|\sqrt{\rho} \ddot{v}\|_2^2 \, ds.
\end{align}
Now we define
\begin{align}
    \psi(t) \coloneqq \frac{1}{2}\|\nabla \dot{v}\|_2^2  - \int_{\R^2} \tr(\nabla v \nabla u \nabla^T \dot{v})
    - \int_{\R^2} \tr(\nabla u\, \nabla \dot{v}\, \nabla^T v)+ \ip{P, \tr(3\nabla u \nabla \dot{v}+\nabla v \nabla \dot{u})}
    + \ip{\nabla P, \ddot{v}}.
\end{align}
Finally, we may write
\begin{align}
    \frac{1}{2}\|\sqrt{\rho}\ddot{v}\|_2^2
    + \psi^\prime(t)
    \leq L_{21}(t) + L_{31}(t) + Q_1(t).
\end{align}
Multiplying by \( s^3 \), integrating in time and using the previous estimates
\begin{align}
\begin{aligned}
    \frac{1}{2} \int_0^t s^3 \|\sqrt{\rho}\ddot{v}\|_2^2\, ds
    + t^3  \psi(t)
    \leq & C \int_0^t s^2 \psi(s) \, ds + \int_0^t s^3 \left(L_{21}(s) + L_{31}(s) + Q_1(s) \right)\, ds \\
    \leq & C \int_0^t s^2 \psi(s) ds + C C_u^2 \|\sqrt{\rho_0}v_0\|_2^2
    + \frac{1}{4} \int_0^t s^3 \|\sqrt{\rho}\ddot{v}\|_2^2\, ds,
\end{aligned}
\end{align}
absorbing we get
\begin{align}\label{eq:A33}
\begin{aligned}
    \int_0^t s^3 \|\sqrt{\rho}\ddot{v}\|_2^2\, ds
    + t^3  \psi(t)
    \lesssim   \int_0^t s^2 \psi(s) ds +  C_u^2 \|\sqrt{\rho_0}v_0\|_2^2
\end{aligned}
\end{align}
Next, we estimate \( \psi(t) \) from above and below. For the upper bound, we combine Young's inequality, Ladyzhenskaya's inequality, and \eqref{eq:Qsecond} for the pressure term to obtain
\begin{align}
    \psi(t) 
    \lesssim & \|\nabla \dot{v}\|_2^2
    + \| \nabla P \|_2 \| \nabla u \|_2 \| \nabla \dot{v} \|_2  + \| \nabla P \|_2 \| \nabla v \|_2 \| \nabla \dot{u} \|_2 
    + \|\nabla \dot{v}\|_2 \|\nabla u\|_4 \|\nabla v\|_4.
\end{align}
Arguing as in \eqref{eq:estimateA2conclusion}, we infer that
\begin{align}\label{eq:uppsi}
    \int_0^t s^2 \psi(s)\, ds \lesssim C^1_u \|\sqrt{\rho_0} v_0\|_2^2.
\end{align}
For the lower bound, we use Young's inequality:
\begin{align}
\begin{aligned}
    \psi(t) 
    &\geq \frac{1}{2}\norm{\nabla \dot{v}}_2^2 - C\norm{\nabla P}_2\, \norm{\nabla u}_2\, \norm{\nabla \dot{v}}_2-C \| \nabla P \|_2 \| \nabla v \|_2 \| \nabla \dot{u} \|_2  - C \norm{\nabla \dot{v}}_2 \norm{\nabla u}_4 \norm{\nabla v}_4 \\   &\geq \frac{1}{4}\norm{\nabla \dot{v}}_2^2 - C\norm{\nabla P}^2_2\, \norm{\nabla u}^2_2\,-C \| \nabla P \|_2 \| \nabla v \|_2 \| \nabla \dot{u} \|_2  - C \norm{\nabla u}^2_4 \norm{\nabla v}^2_4,
\end{aligned}
\end{align}
and so
\begin{align}\label{eq:lowpsi}
\begin{aligned}
    t^3 \psi(t) 
    &\gtrsim  t^3 \norm{\nabla \dot{v}}_2^2 - t^3  \left( \norm{\nabla P}^2_2\, \norm{\nabla u}^2_2\, + \norm{\nabla u}^2_4 \norm{\nabla v}^2_4 \right) -t^3 \| \nabla P \|_2 \| \nabla v \|_2 \| \nabla \dot{u} \|_2 \\ &\gtrsim  t^3 \norm{\nabla \dot{v}}_2^2 - C^2_u\|\sqrt{\rho_0} v_0\|_2^2  -t^3 \| \nabla P \|_2 \| \nabla v \|_2 \| \nabla \dot{u} \|_2 
\end{aligned}
\end{align}
Inserting \eqref{eq:uppsi} and \eqref{eq:lowpsi} in \eqref{eq:A33} we get 
\begin{align}\label{eq:A33bis}
\begin{aligned}
    \int_0^t s^3 \|\sqrt{\rho}\ddot{v}\|_2^2\, ds
    + t^3  \norm{\nabla \dot{v}}_2^2
    \lesssim t^3  \| \nabla P \|_2 \| \nabla v \|_2 \| \nabla \dot{u} \|_2 +    C_u^2 \|\sqrt{\rho_0}v_0\|_2^2.
\end{aligned}
\end{align}
Assuming \(v=u\), we may apply Young's inequality to \eqref{eq:A33bis} and obtain
\begin{align}
\begin{aligned}
     t^3 \|\nabla \dot{u}\|_2^2
    \leq {}& C\, t^2 \|\nabla P\|_2^2 \, t \|\nabla u\|_2^2
    + \frac{1}{2} t^3 \|\nabla \dot{u}\|_2^2 + C C_u^2 \|\sqrt{\rho_0}u_0\|_2^2.
\end{aligned}
\end{align}
Absorbing the second term on the right-hand side and using the previous estimates for $\nabla u$ and $\nabla P$, we obtain
\begin{align}\label{eq:u3}
\begin{aligned}
     t^3 \|\nabla \dot{u}\|_2^2
    \lesssim C_u^2 \|\sqrt{\rho_0}u_0\|_2^2.
\end{aligned}
\end{align}
Finally, we may use \eqref{eq:u3} to conclude the estimate for \eqref{eq:A33bis}. Indeed,
\begin{align}\label{eq:last}
\begin{aligned}
    t^3 \| \nabla P \|_2 \| \nabla v \|_2 \| \nabla \dot{u} \|_2
    &\lesssim ( t \| \nabla P \|_2)\, (\sqrt{t} \| \nabla v \|_2)\, (t^{3/2} \| \nabla \dot{u} \|_2 ) \\
    &\lesssim \sqrt{C^1_u}\,\|\sqrt{\rho_0}v_0\|_2\; \sqrt{C^0_u}\,\|\sqrt{\rho_0}v_0\|_2 \sqrt{C^2_u}\;\|\sqrt{\rho_0}u_0\|_2 \lesssim C^2_u \|\sqrt{\rho_0}v_0\|_2^2.
\end{aligned}
\end{align}
Inserting \eqref{eq:last} into \eqref{eq:A33bis} and using \eqref{eq:stokesA3}, we obtain
\begin{align}
    \int_0^t s^3 \|\sqrt{\rho}\ddot{v}, \nabla^2 \dot{v}, \nabla \dot{P}\|_2^2\, ds
    + t^3 \|\nabla \dot{v}\|_2^2
    \lesssim C^2_u \|\sqrt{\rho_0} v_0\|_2^2.
\end{align}

This concludes the proof of \eqref{eq:A0} and \eqref{eq:A0u} for $i=3$.

\subsection{Auxiliary Results }

\subsubsection{Some useful identities}

\begin{lemma}\label{lem:divstr}
Assume that \( v,u \colon [0,\infty) \times \mathbb{R}^2 \to \mathbb{R}^2 \) are smooth divergence-free vector fields. Then the following identities hold:
\begin{align}
    \dive \dot{v} &= \tr(\nabla u \nabla v) \label{eq:H1dot} \\
    \dive \ddot{v} &= 2 \tr(\nabla u \nabla \dot{v}) +  \tr(\nabla \dot{u} \nabla v)\label{eq:H2dot} \\
    \dive \dddot{v} & = D(\tr( 3\nabla u \nabla \dot{v}+\nabla v \nabla \dot{u})) - \tr(\nabla \dot{u} \nabla \dot{v}) + \tr(\nabla \dot{u}) \tr(\nabla \dot{v}) \label{eq:H3dot_alt}
\end{align}
where we have used the notation $\dot{v} = D_u v$.
\end{lemma}

\begin{proof}
By incompressibility we compute:
\begin{align}
    \dive \dot{v} = \dive [ (u \cdot \nabla) v] = \partial_i v_k \partial_k v_i = \tr ( \nabla u \nabla v),
\end{align}
which proves \eqref{eq:H1dot}. 

For \eqref{eq:H2dot}, we apply \eqref{eq:div} to interchange \( D \) and \( \dive \):
\begin{align}
    \dive \ddot{v} = D(\dive \dot{v}) + \tr(\nabla u \nabla \dot{v}).
\end{align}
Now, using the product rule and \eqref{eq:grad}, we compute:
\begin{align}
    D(\dive \dot{v}) & = D \tr( \nabla u \nabla v) \\
    & =  \tr( D(\nabla u) \nabla v)+ \tr(\nabla u  D(\nabla v) ) \\
    & =  \tr( \nabla \dot{u} \nabla v) - 2\tr( \nabla u \nabla u \nabla v) + \tr(\nabla u  \nabla \dot{v} )
\end{align}
Since both \( \nabla u \) and \( \nabla v \) are trace-free \( 2 \times 2 \) matrices, we recall:
\begin{equation} \label{eq:A^2}
\operatorname{adj}(\nabla u) = -\nabla u 
\quad \Rightarrow \quad 
\nabla u \nabla u = - \det(\nabla u) \operatorname{Id} 
\quad \Rightarrow \quad  
\tr( \nabla u \nabla u \nabla v) = 0
\end{equation}
and thus we obtain \eqref{eq:H2dot}.

For the third derivative, apply again \eqref{eq:div} and \eqref{eq:grad}:
\begin{align}\label{eq:dddot1}
    \dive \dddot{v} &= D(\dive \ddot{v}) + \tr(\nabla u \nabla \ddot{v}).
\end{align}
For the second term, apply the product rule:
\begin{align}
    \tr(\nabla u \nabla \ddot{v}) & = \tr( \nabla u D(\nabla \dot{v})) + \tr(\nabla u \nabla u \nabla \dot{v}) \\
    & = D \tr( \nabla u \nabla \dot{v}) - \tr(D (\nabla u ) \nabla \dot{v}) + \tr(\nabla u \nabla u \nabla \dot{v}) \\
    & = D \tr( \nabla u \nabla \dot{v}) - \tr(\nabla \dot{u}  \nabla \dot{v}) + 2\tr(\nabla u \nabla u \nabla \dot{v}).
\end{align}
To handle the last term, recall that, due to \eqref{eq:A^2},
\begin{align}
\tr(\nabla \dot{u}) = \dive \dot{u} = \tr(\nabla u \nabla u)
\end{align}
(from \eqref{eq:H1dot}). Thus, we have:
\begin{align}
    2 \tr(\nabla \dot{v} \nabla u \nabla u)
    = -2 \det(\nabla \dot{u}) \, \tr (\nabla \dot{v})= - 2 \tr \big(\nabla \dot{v} \det(\nabla \dot{u}) \operatorname{Id}\big)
    =\tr (\nabla \dot{v}) \tr (\nabla \dot{u})
\end{align}
which completes the proof.
\end{proof}

\subsubsection{Hardy and BMO}\label{subsub: Hardy}

\paragraph{Functions with bounded mean oscillations.}Let \(\mathcal{B}\) be the family of all balls in \(\mathbb{R}^d\).  A locally integrable function \(f\) belongs to \(\mathrm{BMO}(\mathbb{R}^d)\) if
\begin{align}
\|f\|_{\mathrm{BMO}}
\;=\;
\sup_{B\in\mathcal{B}}
\fint_{B}\bigl|f(x)-\fint_{B}f(y)\,\mathrm{d}y\bigr|\,
\mathrm{d}x
\;<\;\infty.
\end{align}
Identifying functions that differ by constants, \(\mathrm{BMO}(\mathbb{R}^d)\) becomes a Banach space with the norm above.

In \(\mathbb{R}^2\), Poincaré’s inequality on balls yields, for every \(B=B(x_{0},R)\) and all \(f\in H^{1}(B)\),
\begin{align}
\bigl\|\,f-\fint_{B}f\bigr\|_{L^{2}(B)}
\;\le\;
C\,R\;\|\nabla f\|_{L^{2}(B)},
\end{align}
whence, by Hölder’s inequality,
\begin{align}
\fint_{B}\bigl|f(x)-\fint_{B}f(y)\,\mathrm{d}y\bigr|\,
\mathrm{d}x
\;\le\;
\frac{C}{R}
\bigl\|\,f-\fint_{B}f\bigr\|_{L^{2}(B)}.
\end{align}
Together, these estimates imply the following:

\begin{lemma}\label{lem:embbmo}
There exists \(C>0\) such that for every \(f\in H^{1}_{\mathrm{loc}}(\mathbb{R}^2)\) with \(\nabla f\in L^{2}(\mathbb{R}^2)\),
\begin{align}
\|f\|_{\mathrm{BMO}(\mathbb{R}^2)}
\;\le\;
C\,\|\nabla f\|_{L^{2}(\mathbb{R}^2)}.
\end{align}
\end{lemma}

\paragraph{Hardy Space.}
The \textit{Hardy space} $\mathcal{H}^1(\mathbb{R}^d)$ is a subspace of $L^1(\mathbb{R}^d)$ consisting of all functions $f \in L^1$ whose Riesz transforms also belong to $L^1$:
\begin{align}
    \mathcal{H}^1(\R^d) := \left\{ f \in L^1(\R^d) : R_j f \in L^1(\R^d), \ 1 \leq j \leq d \right\},
\end{align}
where $R_j = \partial_j (-\Delta)^{1/2}$.  
For the norm, alternative definitions, and equivalence between them, we refer to~\cite{FeffermanStein1972}.

We recall the following famous result.

\begin{theorem}[Fefferman–Stein]\label{thm:FeffStein}
    The dual of the Hardy space \( \mathcal{H}^1(\mathbb{R}^d) \) is isomorphic to \( \mathrm{BMO}(\mathbb{R}^d) \), with equivalent norms.
\end{theorem}

We also need the following important result, see \cite{CoifmanLionsMeyerSemmes1993}.

\begin{theorem}\label{thm:divcurl}
    Let \( v, z \in L^2(\mathbb{R}^2; \mathbb{R}^2) \) be two vector fields such that \(\dive v = 0\) and \(\operatorname{curl} z = 0\). Then, we have that $v \cdot z \in \mathcal{H}^1(\mathbb{R}^2)$ and the estimate
    \begin{align}
         \|u \cdot v\|_{\mathcal{H}^1(\mathbb{R}^2)} \leq C \|u\|_{L^2(\mathbb{R}^2)}\|v\|_{L^2(\mathbb{R}^2)},
    \end{align}
    where $C$ is a constant independent of $u$ and $v.$
\end{theorem}

\cref{thm:FeffStein} and \cref{thm:divcurl} allow us to show more accurate estimates of certain products which cannot be deduced from Hölder's inequality. We obtain the following.

\begin{corollary}
    Let \( v, z \in L^2(\mathbb{R}^2; \mathbb{R}^2) \) be two vector fields such that \(\dive v = 0\) and \(\operatorname{curl} z = 0\) and let $\varphi \in \dot{H}^1(\R^2).$ Then we have the estimate
    \begin{align}\label{eq:h1bmo}
        \int_{\mathbb{R}^2} \varphi(x) \, v(x) \cdot z(x) \, \dd x
        \lesssim \norm{ \nabla \varphi}_{L^2(\mathbb{R}^2)} \, \norm{ v}_{L^2(\mathbb{R}^2)} \, \norm{ z}_{L^2(\mathbb{R}^2)}.
    \end{align}
    Moreover, if $v \in \dot{H}^1(\mathbb{R}^2; \mathbb{R}^2)$, then we have $\det(\nabla v) \in \mathcal{H}^1(\R^2)$ and
    \begin{align}\label{eq:h2bmo}
        \int_{\mathbb{R}^2} \varphi(x) \det(\nabla v(x)) \, \dd x 
        \lesssim \norm{\nabla \varphi}_{L^2(\mathbb{R}^2)} \, \norm{\nabla v}_{L^2(\mathbb{R}^2)}^2.
    \end{align}
\end{corollary}

\begin{proof}
    Since $v \cdot z \in \mathcal{H}^1(\R^2)$ by \cref{thm:divcurl} and since $\varphi \in \operatorname{BMO}(\R^2)$ by \cref{lem:embbmo}, the estimate in \eqref{eq:h1bmo} follows from \cref{thm:FeffStein}. The estimate in \eqref{eq:h2bmo} follows from the identity
    \begin{align}
        \det( \nabla \varphi) = \nabla^\perp \varphi_1 \cdot \nabla \varphi_2
    \end{align}
    and from \eqref{eq:h1bmo}.
\end{proof}

Combining \eqref{eq:h1bmo}, \eqref{eq:h2bmo} and \cref{lem:divstr}, we prove the following proposition.

\begin{proposition}\label{prop:divH}
Assume \( v,u : \R_+ \times \mathbb{R}^2 \to \mathbb{R}^2 \) are sufficiently regular divergence-free vector fields. Then, for every \(\varphi \in C^\infty_c([0,\infty) \times \mathbb{R}^2)\), we have
\begin{align}
    \int_{\mathbb{R}^2} \varphi \, \dive \dot{v} \lesssim & \norm{\nabla \varphi}_{2} \, \norm{\nabla v}_{2} \, \norm{\nabla u}_{2}, \\
    \int_{\mathbb{R}^2} \varphi \, \dive \ddot{v} \lesssim & \norm{\nabla \varphi}_{2} \left( \norm{\nabla \dot{v}}_{2} \, \norm{\nabla u}_{2} + \norm{\nabla v}_{2} \, \norm{\nabla \dot{u}}_{2} \right),\\
    \int_{\mathbb{R}^2} \varphi \dive \dddot{v} \lesssim & \frac{\dd}{\dd t} \ip{\varphi ,\tr(3\nabla u \nabla \dot{v}+\nabla v \nabla \dot{u})} + \|\nabla \varphi\|_2 \|\nabla \dot{u}\|_2\|\nabla \dot{v}\|_2\\
    & + \|\nabla \dot{\varphi}\|_2 \left(\|\nabla v\|_2 \|\nabla \dot{u}\|_2 + \|\nabla u\|_2 \|\nabla \dot{v}\|_2 \right)
\end{align}
\end{proposition}

\begin{proof}
The first two estimates follow by combining \eqref{eq:H1dot} and \eqref{eq:H2dot} with \eqref{eq:h1bmo}, and observing that
\begin{align}\label{eq:tracedivcurl}
    \tr(\nabla v \nabla \dot{u}) = \partial_i v_j \, \partial_j \dot{u}_i = \partial_i v \cdot \nabla \dot{u}_i,
\end{align}
where \(\dive (\partial_i v) = 0\) (due to divergence-free condition) and \(\curl (\nabla u_i) = 0\), with similar reasoning applying to \(\tr(\nabla v \nabla u)\) and \(\tr(\nabla \dot{v} \nabla u)\).

With \eqref{eq:H3dot_alt}, we can write
\begin{align}
    \int_{\mathbb{R}^2} \varphi \dive \dddot{v} = & \int_{\mathbb{R}^2} \varphi D(\tr( 3\nabla u \nabla \dot{v}+\nabla v \nabla \dot{u}))\\
    & + \int_{\mathbb{R}^2} \varphi \det (\nabla (\dot{v} +\dot{u})) - \det (\nabla \dot{v})  - \det (\nabla \dot{u})
    \eqqcolon I(t) + II(t).
\end{align}
For $I$, observe that
\begin{align}
    I(t) = & \frac{\dd}{\dd t} \ip{\varphi ,\tr(3\nabla u \nabla \dot{v}+\nabla v \nabla \dot{u})} - \ip{\dot{\varphi} ,\tr(3\nabla u \nabla \dot{v}+\nabla v \nabla \dot{u})}\\
    \lesssim & \frac{\dd}{\dd t} \ip{\varphi ,\tr(3\nabla u \nabla \dot{v}+\nabla v \nabla \dot{u})} + \|\nabla \dot{\varphi}\|_2 \left(\|\nabla v\|_2 \|\nabla \dot{u}\|_2 + \|\nabla u\|_2 \|\nabla \dot{v}\|_2 \right).
\end{align}
Note that we used \eqref{eq:tracedivcurl} to estimate the second term in the last step. In order to use estimate $II$, we use the following identity
\begin{align}\label{eq:noquadraticterms}
    \det\big(\nabla (\dot{v} + \dot{u})\big) - \det(\nabla \dot{v}) - \det(\nabla \dot{u}) = - \nabla  \dot{v}_1 \cdot \nabla^\perp \dot{u}_2 -  \nabla \dot{u}_1 \cdot \nabla^\perp  \dot{v}_2.
\end{align}
Using \eqref{eq:h1bmo} and \eqref{eq:noquadraticterms}, we can estimate $II$ as
\begin{align}
    II(t) \lesssim & \|\nabla \varphi\|_2 \|\nabla \dot{u}\|_2\|\nabla \dot{v}\|_2,
\end{align}
which completes the proof.
\end{proof}

\subsubsection{Linear Stokes estimates}

We start with the following estimate for the Stokes operator, see \cite{Danchin2025}.

\begin{lemma}\label{lem:Genstokes}
    Let $p \in (1,\infty)$. Let $(u,P)$ be a solution of the following Stokes equation
    \begin{align}
        \begin{cases}
            - \Delta u + \nabla P = f,\\
            \dive(u) = g,\\
            u \to 0,\quad \abs{x} \to \infty,
        \end{cases}
    \end{align}
    where $f \in L^p(\R^2)$ and $g \in \dot{W}^{1,p}(\R^2).$ Then we have that
    \begin{align}
        \| \nabla^2 u, \nabla P\|_{L^2(\R^2)} \lesssim  \|f\|_{L^2(\R^2)} + \| \nabla g \|_{L^p(\R^2)}.
    \end{align}
\end{lemma}

We finish this section with an important $L^\infty$-estimates for the gradient of the velocity of solutions \eqref{eq:LSs}.
The case $v=u$ was shown in \cite{HaoShaoWeiZhang2026}, the extension to the linear case is straightforward.

\begin{proposition}\label{prop:Linfty}
    Let $(\rho,u)$ be an immediately strong solution of \eqref{eq:INS} and assume that $v$ is an immediately strong solution of \eqref{eq:LSs}. Then, for every $t>0,$
    \begin{align}
        \norm{\nabla v(t)}_{L^\infty(\R^2)} \lesssim \|\sqrt{\rho}\dot{v}(t)\|_{L^2(\R^2)} + \|\nabla v(t)\|_{L^2(\R^2)}^{1/2} \|\nabla\dot{v}(t)\|_{L^2(\R^2)}^{1/2}.
    \end{align}
\end{proposition}

\section{Properties of immediately strong solutions}

\subsection{The continuity equation}

In this subsection, we recall some basic properties of bounded solutions to the continuity equation. 
Let $T>0$ be fixed and let $D\subseteq\mathbb{R}^2$ be a bounded Lipschitz domain. Let
\begin{equation}\label{eq:trExDP}
u \in L^1\big((0,T);L^1_{\loc}(\mathbb{R}^2)\big),
\qquad \dive u(t,\cdot)=0 \quad \text{for a.e. } t\in(0,T).
\end{equation}
We consider the continuity equation on $(0,T)\times\mathbb{R}^2$
\begin{equation}\label{eq:conteq}
\begin{cases}
\partial_t \rho + \dive(\rho u)=0,\\
\rho(0,\cdot)=\rho_0=\mathbf{1}_D.
\end{cases}
\end{equation}

\begin{definition}[Weak and renormalized solutions]\label{def:weakrensol}
We say that $\rho \in L^\infty((0,T)\times\mathbb{R}^2)$ is a weak solution of \eqref{eq:conteq} (under the assumption \eqref{eq:trExDP}) if for every $\varphi \in C_c^\infty([0,T)\times\mathbb{R}^2)$,
\begin{align}
\int_0^T \!\!\int_{\mathbb{R}^2} \rho\,\partial_t \varphi \,\dd x\,\dd t
+ \int_0^T \!\!\int_{\mathbb{R}^2} \rho\, u \cdot \nabla \varphi \,\dd x\,\dd t
= - \int_{\mathbb{R}^2} \rho_0\, \varphi(0,\cdot)\,\dd x.
\end{align}
We say that $\rho$ is a renormalized solution of \eqref{eq:conteq} if, for every $\beta \in C^1(\mathbb{R})$, the function $\beta(\rho)$ is a weak solution of \eqref{eq:conteq} with initial datum $\beta(\rho_0)$.
\end{definition}

Under assumption \eqref{eq:trExDP}, the continuity equation \eqref{eq:conteq} admits weak solutions, which are renormalized under additional regularity assumptions on the velocity field.

\begin{proposition}\label{prop:wellposedtrans}
Assume that $u$ satisfies \eqref{eq:trExDP}. Then there exists a weak solution $\rho$ of \eqref{eq:conteq} with initial datum $\rho_0$.  If, in addition, for some $q\in[1,\infty)$,
\begin{equation}\label{ass:velocity}
u \in L^1\big((0,T);W^{1,q}_{\loc}(\mathbb{R}^2)\big),
\qquad
\frac{u}{1+|x|} \in L^1\big((0,T);L^1(\mathbb{R}^2)+L^\infty(\mathbb{R}^2)\big),
\end{equation}
then there exists a unique renormalized solution $\rho \in L^\infty((0,T)\times\mathbb{R}^2)$. Moreover, up to a modification on a negligible set of times, one has
\begin{align}
\rho \in C_{w^\ast}\big([0,T];L^\infty(\mathbb{R}^2)\big)
\quad\text{and}\quad
\rho \in C\big([0,T];L^p(\mathbb{R}^2)\big)
\ \text{for every } p\in[1,\infty).
\end{align}

\end{proposition}

\begin{proof}
The existence of weak solutions is proved in \cite[Proposition~II.4]{DipLi89}. The time continuity properties are pointed out in \cite{Skondric2025}. The uniqueness and renormalization results follow from \cite{DipLi89}.
\end{proof}

\subsection{\texorpdfstring{$A$-property and $L^p$ strip estimate}{A-property and Lp strip estimate}}
In this subsection we investigate additional properties of weak solutions to the continuity equation \eqref{eq:conteq} under the assumption \eqref{eq:trExDP}, and assuming further regularity on the solution--velocity pair $(\rho,u)$:
\begin{equation}\label{eq:trExDP2}
 \norm{\sqrt{\rho}\,u }_{L^\infty\big((0,T);L^2(\mathbb{R}^2)\big)} \eqqcolon M < \infty.
\end{equation}
Throughout, the regularity of $D$ is crucial. Since $D$ is a bounded Lipschitz domain, it enjoys the so-called $(A)$--property.

\begin{definition}
    A set $D \subseteq \R^2 $ has the $(A)$--property if there exists a constant $A>0$ such that for every $r>0$ and every
$x\in\mathbb{R}^2$ satisfying
\begin{align}
d := \operatorname{dist}(x,D) < r,
\end{align}
one has
\begin{equation}\label{eq:Aproperty}
|B_r(x_0)\cap D| \;\ge\; \min \{A\,(r-d)^2, \abs{D}\}.
\end{equation}
\end{definition}

The $(A)$--property is used to track the size of the mixing zone between mass and vacuum. For this reason for $t>0$ and $R>0$ a key role is played by the following time-dependent neighbourhood of $D$:
\begin{align}
D_{t,R} := D + B_{R\sqrt{t}}(0).
\end{align}

\begin{lemma}\label{lem:nondegdensity}
Let $u$ satisfy \eqref{eq:trExDP} and let $\rho$ be a weak solution of  \eqref{eq:conteq} on $(0,T)$. Assume that the pair $(\rho,u)$ satisfies  \eqref{eq:trExDP2}, and fix $R>0$.  Then there exists a positive constant
\begin{align}
L = L(M,R,T,D),
\end{align}
which is increasing with respect to both $M$ and $R$, such that for every $t \in (0,T)$ and every $x \in D_{t,R}$ one has
\begin{equation}\label{eq:chAprop2}
\int_{B_{\sqrt{t}\,L}(x)} \rho(t,x)\,dx 
\;\gtrsim_{T,D}\; t.
\end{equation}
\end{lemma}
\begin{proof} For simplicity we assume $\abs{D}=1$. Since $x \in D + B_{R\sqrt t}(0)$, we have $d := \dist(x,D) < R\sqrt t$.
By the $(A)$--property, there exists a constant $A>0$, such that for every $r> d$
\begin{align}
\int_{B_r(x)} \rho_0(x)\,dx
= |B_r(x)\cap D|
\ge \min \{A\,(r-d)^2, 1\}.
\end{align}
We then set $r := (R+1)\sqrt t$. 
Using $d\le R\sqrt t$, we obtain
\begin{align}
\int_{B_r(x)} \rho_0(x)\,dx
\ge  \min \{A t, 1\}=  C \; t, \qquad \text{with }C := \min \{ A, T^{-1}\}.
\end{align}
Following
the argument of \cite[Lemma~2.2]{HaoShaoWeiZhang2026} with
$c_0 := C_T \; t$ and $R_0 = r$, we infer that if
\begin{align} \label{eq:condQ}
Q \ge (R+1) \sqrt{t} + \frac{2t}{\sqrt{ C \; t}\,}\,
M  =: \sqrt{t} \; L 
\end{align}
then
\begin{align}
\int_{B_Q(x)} \rho(t,x)\,dx \ge \frac{C \; t}{4}.
\end{align}
Since $L$ is increasing in both $R$ and $M$ we conclude the thesis.
\end{proof}

Now we show that inside $D_{t,R}$ one can control the $L^p$ norm of a general measurable function $v$ in a scale-critical way.

\begin{proposition}\label{prop:localLp}
Let $u$ satisfy \eqref{eq:trExDP} and let $\rho$ be a weak solution of
\eqref{eq:conteq} in $(0,T)$. Assume that $(\rho,u)$ satisfies \eqref{eq:trExDP2}, and fix $R>0$.
 Then there exist two constants $C_1$ and $C_2$ such that, for any measurable function $v$,
any $t \in (0,T)$ and any $x \in D_{t,R}$, the following holds.

For any $p \in [2,\infty)$
\begin{equation}\label{eq:localLp}
t^{\frac12-\frac1p}\,
\|v\|_{L^p(B_{\sqrt t L}(x))}
\lesssim_{T,D,p}
C_1\left(
t^{\frac{1}{2}}\,\|\nabla v\|_{L^2(B_{\sqrt t L}(x))}
+
\|\rho(t)\,v\|_{L^2(B_{\sqrt t L}(x))}
\right)
\end{equation}
Moreover, if $p=\infty$, then for any $q>2$ we have
\begin{align}\label{eq:globLinfty}
t^{\frac12}\,
\|v\|_{L^\infty(D_{R,t})}
\lesssim_{T,D,q}
C_2\left(
t^{1-\frac{1}{q}}\|\nabla v\|_{L^q(\mathbb{R}^2)}
+
\|\rho(t)\,v\|_{L^2(\mathbb{R}^2)}
\right).
\end{align}
Both constants depend in a non decreasing way on $R$ and $M$.
\end{proposition}

\begin{proof} We pick $ L$ as in \cref{lem:nondegdensity} and consider $t \in (0,  T)$ and $x \in D_{R,t}$. For simplicity we use the shortcut $B = B_{\sqrt{t }L}(x)$. We introduce the weighted average
\begin{align}
v_\rho = 
\frac{\int_{B} \rho(t,y)\,v(y)\,dy}
{\int_{B} \rho(t,y)\,dy}.
\end{align}
To estimate $v_\rho$, we use Cauchy--Schwarz inequality for the numerator
\begin{align}
\int_{B} \rho(t,y)\,|v(y)|\,dy
& \le \;
|B|^{1/2}\,
\|\rho(t)\,v\|_{L^2(B)} \lesssim \;  L t^{1\slash 2}
\|\rho(t)\,v\|_{L^2(B)}.
\end{align}
and \cref{lem:nondegdensity} for the denominator
\begin{align}
\int_{B} \rho(t,y)\,dy
\gtrsim_{T,D}
\,t.
\end{align}
Therefore we conclude
\begin{align}\label{eq:pointwise}
\abs{v_\rho}
\lesssim_{T,D}
L 
\,t^{-1/2}\,
\|\rho(t)\,v\|_{L^2(B)}.
\end{align}
Let $\Bar{v}$ be the average of $v$ in $B$. By \eqref{eq:pointwise} we can estimate the difference between the weighted average and the average. Indeed, since $\rho \leq 1$ 
\begin{align}
\begin{aligned}\label{eq:pointwise2}
    \abs{v_\rho - \Bar{v}} & = \abs{ ( v -\Bar{v})_\rho} \lesssim  L 
\,t^{-1/2}\, \norm{  v -\Bar{v}}_{L^2(B)}.
\end{aligned}
\end{align}
The last ingredient we need is the Sobolev--Poincar\'e inequality in dimension two. For $p \in [2,\infty)$
\begin{align}\label{eq:Poinc}
\|v-\Bar{v}\|_{L^p(B)}
\lesssim_p\,(\sqrt t L)^{2/p}\,\|\nabla v\|_{L^2(B)}.
\end{align}
Finally we can estimate $v$ in $L^p$
\begin{align}
\|v\|_{L^p(B)}
\le
\|v- \Bar{v}\|_{L^p(B)} 
+ 
|v_\rho - \Bar{v}| |B|^{1/p}+ |v_\rho||B|^{1/p}.
\end{align}
Using \eqref{eq:pointwise}, \eqref{eq:pointwise2}, \eqref{eq:Poinc} and $|B| \leq 4 t L^2 $,
\begin{align}
t^{-\frac{1}{p}}\|v\|_{L^p(B)}
&\lesssim_{T,,pD}
 \; L^\frac2p\,\|\nabla v\|_{L^2(B)}  +  L^{\frac2p+2}\,\|\nabla v\|_{L^2(B)} + L^{\frac2p+1} t^{- \frac12} 
\|\rho(t)\,v\|_{L^2(B)}
\end{align}
and since $L$ depends in a non decreasing way on $M$ and $R$ we conclude \eqref{eq:localLp}. For the case $p = \infty$, we fix $\eps>0$ and choose $x'\in B$ such that
\begin{align}
|v(x')|
\le
\operatorname*{ess\,inf}_{y\in B} |v(y)| + \varepsilon.
\end{align}
Then, by \eqref{eq:pointwise}, we obtain
\begin{align}
|v(x')|
\le
\frac{\int_{B} \rho(t,y)\,|v(y)|\,dy}
{\int_{B} \rho(t,y)\,dy}
+\varepsilon \lesssim_{T,D}
L 
\,t^{-1/2}\,
\|\rho(t)\,v\|_{L^2(B)}.
\end{align}
Then, by Morrey's inequality for any $q>2$:
\begin{align}
|v(x)|
&\le |v(x) - v(x')| + |v(x')| \\
&\lesssim_{T,D,p} \|\nabla v\|_{L^q(\mathbb{R}^2)} |x - x'|^{1-2/q}
+ L\,t^{-\frac12}\,\|\rho(t)\,v\|_{L^2(\mathbb{R}^2)}
+\varepsilon \\
&\lesssim \|\nabla v\|_{L^q(\mathbb{R}^2)} (L t^\frac12)^{1-2/q}
+ L\,t^{-\frac12}\,\|\rho(t)\,v\|_{L^2(\mathbb{R}^2)}
+\varepsilon.
\end{align}
Letting $\varepsilon \to 0$ and using the arbitrariness of $x \in D_{t,R}$, we conclude \eqref{eq:globLinfty}.
\end{proof}

The following corollary will be a key point in the subsequent analysis.

\begin{corollary}\label{cor:Lpstrip}
Let $(\rho,u)$ be a Leray--Hopf weak solution arising from \eqref{ass:weakData2}. Then, for any $R>0$ and $T>0$, there are two constants $C_1$, $C_2$ such that, for any measurable
function $v$, and every $t \in (0,   T)$ the following holds:

For any $p\in[2,\infty)$,
\begin{align}
t^{\frac12-\frac1p}\,
\|v\|_{L^p(D_{R,t})}
\lesssim_{T,D,p}
C_1\Big(
t^{1\slash 2}\,\|\nabla v\|_{L^2(\mathbb{R}^2)}
+
\|\rho(t)\,v\|_{L^2(\mathbb{R}^2)}
\Big).
\end{align}
If $p=\infty$, then for any $q>2$,
\begin{equation}
t^{\frac12}\,
\|v\|_{L^\infty(D_{R,t})}
\lesssim_{T,D,q}
C_2\left(
t^{1-\frac{1}{q}}\|\nabla v\|_{L^q(\mathbb{R}^2)}
+
\|\rho(t)\,v\|_{L^2(\mathbb{R}^2)}
\right)
\end{equation}
Both constants depend  monotonically on $R$ and $\|\rho_0 u_0\|_{L^2(\mathbb{R}^2)}$. 
\end{corollary}

\begin{proof}
Set $M=\|\rho_0 u_0\|_{L^2(\mathbb{R}^2)}$. By the energy inequality this is an admissible choice in \eqref{eq:trExDP2}.
The claim follows by combining \cref{prop:localLp} with the Besicovitch covering theorem.
\end{proof}

\subsection{Finite propagation of the density}
We now apply the theory from the previous subsection to immediately strong solutions arising from initial data $(\rho_0,u_0)$ satisfying \eqref{ass:weakData2}, namely 
\begin{align}
    u_0 \in L^2_\sigma(\R^2), \qquad
    0 \le \rho_0 = \mathbf{1}_D,
\end{align}
where $D \subseteq \R^2$ is a bounded Lipschitz domain. 
Let $(\rho,u)$ be a weak solution of \eqref{eq:INS} with initial data $(\rho_0,u_0)$, and let $v$ be an immediately strong solution to \eqref{eq:LSs} advected by $(\rho,u)$. 
We recall that, by \cref{def:immstrongsol-LS}, there exists a constant $C_{(\rho,u)}(v)$ such that
\begin{align}
    \sup_{i \in \{0,1,2,3\}} \sup_{t>0} A_i^0(t,v,u) \le C_{(\rho,u)}(v).
\end{align}
All the properties studied in the next two subsections become particularly transparent in the self-advected case, namely when $v = u$; however, for future analysis we will make essential use of the properties of solutions to the linearized system, which is why we work in this more general setting.

We start by proving estimates on the $L^\infty$ norm of the solution.

\begin{lemma}\label{lem:Linfty}
    Let $(\rho,u)$ a weak solution of \eqref{eq:INS} arising from \eqref{ass:weakData2} and $v$ an immediately strong solution to \eqref{eq:LSs} advected by $(\rho,u)$. Then, for any $R>0$ there exist two constants 
    \begin{align}
        C= C(R, \norm{\rho_0 u_0}_{L^2(\R^2)}, C_{(\rho,u)}(v)), \qquad C^\prime= C^\prime( \norm{\rho_0 u_0}_{L^2(\R^2)}, C_{(\rho,u)}(v))
    \end{align}
    such that for every $T>0$
    \begin{align}\label{eq:Linftytubdomainglobal}
        \sup_{t\in (0, T)} t^\frac{1}{2} \norm{v(t)}_{L^\infty(D_{t,R})} \lesssim_{T,D} C, \quad \int_0^{T} t^2 \norm{\dot{v}(t)}^2_{L^\infty(D_{t,R})} \dd t \lesssim_{T,D} C 
    \end{align}
    and every $x \in \R^2$,
    \begin{align}\label{eq:growuptw}
       \sup_{t \in (0,T)}t^{3\slash 4} \frac{\abs{v(t,x)}}{(1 + \abs{x})} \lesssim_{T,D} C^\prime, \quad \int_0^{T} t^{ 5\slash 2} \frac{ \abs{\dot{v}(t,x)}^2 }{(1+\abs{x})^2} \dd t \lesssim_{T,D} C^\prime.
    \end{align}
    Both $C$ and $C^\prime$ are non-decreasing in their arguments.
\end{lemma}
\begin{proof}
   Choosing $v = v(t)$, $q=4$ in \cref{cor:Lpstrip} we have 
\begin{align}
    t^{\frac12}\,
\|v(t)\|_{L^\infty(D_{R,t})}
& \le
C_2\left(
t^{\frac{3}{4}}\|\nabla v(t)\|_{L^4(\mathbb{R}^2)}
+
\|\rho(t)\,v(t)\|_{L^2(\mathbb{R}^2)}.
\right) \\ & \le
C_2\left(
t^{\frac{3}{4}}\|\nabla v(t)\|^{1\slash 2}_{L^2(\mathbb{R}^2)} \|\nabla^2 v(t)\|^{1\slash 2}_{L^2(\mathbb{R}^2)}
+
C_{(\rho,u)}(v)^{1\slash 2}
\right) \\ & \le
C_2\left(
t^{\frac{3}{4}} C_{(\rho,u)}(v)^{1\slash 2} t^{-1\slash 4} t^{-1\slash 2}
+
C_{(\rho,u)}(v)^{1\slash 2}
\right)
\end{align}
and this  proves the first estimate in \eqref{eq:Linftytubdomainglobal}. For the estimate on $\dot{u}$ we use $ v = \dot{u}(t)$ and as before we get
\begin{align}
    t^{\frac12}\,
\|\dot{v}(t)\|_{L^\infty(D_{R,t})}
 & \le
C_2( 
t^{\frac{3}{4}}\|\nabla \dot{v}(t)\|_{L^4(\mathbb{R}^2)} 
+
\|\rho(t)\,\dot{v}(t)\|_{L^2(\mathbb{R}^2)}).
\end{align}
Multiplying both sided by $t^{1\slash 2}$, squaring and integerating in $(0,T)$ we get
\begin{align}
    \int_0^{T} t^2 
\|\dot{v}(t)\|_{L^\infty(D_{R,t})}^2 \dd t
 & \lesssim_{C_2}  \int_0^{T} t^{\frac{5}{2}} 
\|\nabla \dot{v}(t)\|_{L^4(\mathbb{R}^2)}^2  \dd t 
+ \int_0^{T} t 
\|\rho \dot{v}(t)\|_{L^2(\mathbb{R}^2)}^2 \dd t . 
\end{align}
The last term is bounded by $C_{(\rho,u)}(v)$ and 
\begin{equation}\label{eq:grad_udot_L4}
\begin{aligned}
\int_0^{T} t^{\frac{5}{2}}
\|\nabla \dot{v}(t)\|_{L^4(\mathbb{R}^2)}^2 \, dt
&\;\lesssim_{C_2}\;
\int_0^{T} t^{\frac{5}{2}}
\|\nabla \dot{v}(t)\|_{L^2(\mathbb{R}^2)}
\|\nabla^2 \dot{v}(t)\|_{L^2(\mathbb{R}^2)} \, dt \\
&\;\lesssim\;
\int_0^{T} t^{2}
\|\nabla \dot{v}(t)\|_{L^2(\mathbb{R}^2)}^2 \, dt
+
\int_0^{T} t^{3}
\|\nabla^2 \dot{v}(t)\|_{L^2(\mathbb{R}^2)}^2 \, dt .
\end{aligned}
\end{equation}
which is also bounded by the same constant and concludes  \eqref{eq:Linftytubdomainglobal}.

To prove \eqref{eq:growuptw} we consider $x\in \R^2$ and $y \in D$:
\begin{align}
    \abs{v(t,x)} & \leq \abs{v(t,x) - v(t,y)} + \abs{v(t,y)} \\ & \leq C \norm{\nabla v(t)}_{L^4(\R^2)}\abs{ x -y }^{1\slash 2} + C t^{-1\slash 2} \\ & \leq C C_{(\rho,u)}(v)^{1\slash 2} t^{-\frac{3}{4}} \abs{ x -y }^{1\slash 2} + C T^{1\slash2}t^{-3\slash 4}
\end{align}
where, in the second inequality, we have used Morrey's inequality and \eqref{eq:Linftytubdomainglobal} for $R = 0$. Now, since $y \in D,$ we simply estimate
\begin{align}
    \abs{ x -y }^{1\slash 2} \leq C ( 1 + \abs{x})
\end{align}
where $C$ depends on $\operatorname{diam}(D)$ and so we get the estimate for $v$ in \eqref{eq:growuptw}. For $\dot{v}$ in the same way we get
\begin{align}
    \abs{\dot{v}(t,x)} & \leq C  \norm{\nabla \dot{v}(t)}_{L^4(\R^2)}( 1 + \abs{x})+ \abs{\dot{v}(t,y)}   
\end{align}
Then, dividing by $1+\abs{x}$, multiplying by $t^{5\slash 4}$, squaring and integrating in $(0,T)$ we get 
\begin{align}
     \int_0^{T} t^{5\slash 2} \frac{ \abs{\dot{v}(t,x)}^2 }{(1+\abs{x})^2} \dd t \leq C  \int_0^{T} t^{5\slash 2} \norm{\nabla \dot{v}(t)}^2_{L^4(\R^2)} \dd t + \int_0^{T} t^{5\slash 2} \norm{\dot{v}(t)}^2_{L^\infty(D)} \dd t.
\end{align}
The right hand side is bounded by \eqref{eq:grad_udot_L4} and \eqref{eq:Linftytubdomainglobal}. This concludes \eqref{eq:growuptw}.
\end{proof}
We immediately obtain the following corollary.

\begin{corollary}\label{cor:lingrowthu}
   Let $(\rho,u)$ and immediately strong solution of \eqref{eq:INS} arising from \eqref{ass:weakData2}. Then  
    \begin{align}\label{eq:lineargrowthu}
        \frac{u}{1+\abs{x}} \in L_\loc^1 ([0, \infty);L^\infty(\R^2)).
    \end{align}
    In particular, $\rho$ is the unique renormalized solution \eqref{eq:conteq} with vector field $u$ and, up to a modification on a negligible set of times, one has
\begin{align}
\rho \in C_{w^\ast}\big([0,\infty);L^\infty(\mathbb{R}^2)\big)
\quad\text{and}\quad
\rho \in C\big([0,\infty);L^p(\mathbb{R}^2)\big)
\ \text{for every } p\in[1,\infty).
\end{align}
\end{corollary}

\begin{proof}
    \eqref{eq:lineargrowthu} follows immediately from \eqref{eq:growuptw} in the case $u=v$. Now, the rest follows from \cref{prop:wellposedtrans}. 
\end{proof}
We now show that the support of the density moves with a scale-critical speed. More precisely, the transported patch remains confined in a $\sqrt{t}$–neighbourhood of the initial domain.

\begin{proposition}\label{prop:ImprovedDecay}
Let $(\rho,u)$ be an immediately strong solution of \eqref{eq:INS} arising from
\eqref{ass:weakData2}, and fix $T>0$. Then there exists a constant 
\begin{align}
C = C\big( \|\rho_0 u_0\|_{L^2(\mathbb{R}^2)}, C_{(\rho,u)}(u), T, D \big)
\end{align}
such that, for almost every $t \in [0,T)$ and for almost every $x \in D$,
\begin{align}
    |X(t,x) - x|
    \leq C \, t^{1/2}.
\end{align}
In particular,
\begin{align}
\operatorname{supp}\rho(t) \subseteq D + B_{C\sqrt{t}}(0).
\end{align}
\end{proposition}
\begin{proof}
  Since $u$ satisfies \eqref{eq:lineargrowthu} and enjoys Sobolev regularity provided by the energy inequality, there exists a unique regular Lagrangian flow
\begin{align}
X : [0,T) \times \mathbb{R}^2 \to \mathbb{R}^2
\end{align}
associated with $u$ (see \cite{CripDel08}).  
Then, for every $t \in (0,T)$ and every $x \in D$, we have
\begin{align}
    |X(t,x) - x|
    &\le \int_0^t |u(s,X(s,x))|\,\dd s \\
    &\le \int_0^t |u(s,x)|\,\dd s
    + \int_0^t |u(s,x)-u(s,X(s,x))|\,\dd s \\
    &\le \int_0^t |u(s,x)|\,\dd s
    + C_M \int_0^t \|\nabla u(s)\|_{L^4(\mathbb{R}^2)}
    |x - X(s,x)|^{1/2}\,\dd s,
\end{align}
where we have used Morrey's inequality.

For the first term, there exists a constant
\begin{align}
C = C\bigl(\|\rho_0 u_0\|_{L^2(\mathbb{R}^2)}, C_{(\rho,u)}(u), D\bigr)
\end{align}
(see \eqref{eq:Linftytubdomainglobal} with $v = u$) such that
\begin{align}
    \int_0^t |u(s,x)|\,\dd s
    &\le \int_0^t \|u(s)\|_{L^\infty(D)}\,\dd s
    \lesssim_{T,D} C \int_0^t s^{-1/2}\,\dd s
    \leq C\, t^{1/2}.
\end{align}
For the second term, we use the interpolation inequality
\begin{align}
    \|\nabla u(s)\|_{L^4(\mathbb{R}^2)}
    \lesssim
    \|\nabla u(s)\|_{L^2(\mathbb{R}^2)}^{1/2}
    \|\nabla^2 u(s)\|_{L^2(\mathbb{R}^2)}^{1/2}
    \le C_{(\rho,u)}(u)^{1/2}\, s^{-3/4}.
\end{align}
Up to redefining the constant $C$, we conclude that
\begin{align}
    |X(t,x) - x|
    \lesssim_{ C}
    t^{1/2}
    + \int_0^t s^{-3/4} |x - X(s,x)|^{1/2}\,\dd s.
\end{align}
Using fractional Grönwall's we get
    \begin{align}
        \abs{X(t,x) - x} \lesssim_{ C} t^{1/2}
    \end{align}
    which concludes the proof.
\end{proof}

We conclude this subsection by establishing a regularity property for the velocity field away from the initial time.

\begin{lemma}\label{lem:B1}
  Let $(\rho,u)$ an immediately strong solution of \eqref{eq:INS} arising from \eqref{ass:weakData2} and $v$ an immediately strong solution to \eqref{eq:LSs} advected by $(\rho,u)$. Then 
\begin{align}
 v  \in H^1_{\loc}\big((0,\infty); H^1_{\loc}(\mathbb{R}^2)\big).
\end{align}
\end{lemma}

\begin{proof}
Fix $\varepsilon >0$, then by  \eqref{eq:BarC-LS} we have
\begin{align}\label{eq:gradientB}
v,\,\dot v \in L^\infty\!\big((\varepsilon,\infty);\dot H^1(\mathbb{R}^2)\big).
\end{align}
Fix also $T>\eps$ and let $K \subseteq \mathbb{R}^2$ be compact. 
We also claim that
\begin{align}\label{eq:Lloc}
v \in L^\infty\!\big((\varepsilon,T); L^\infty(K)\big),
\qquad
\dot{v} \in L^2\!\big((\varepsilon,T); L^\infty(K)\big).
\end{align}
The first inclusion in \eqref{eq:Lloc} follows directly from \eqref{eq:growuptw}. For the
second one, we note that for every $x\in K$
\begin{align}
\int_{\varepsilon}^{0} |\dot{v}(t,x)|^2 \,dt
&\le_{\varepsilon,K}
\int_{\varepsilon}^{T} t^{5/2}\,
\frac{|\dot{v}(t,x)|^2}{(1+|x|)^2} \,dt
\le C',
\end{align}
where the last inequality follows again from \eqref{eq:growuptw} and proves \eqref{eq:Lloc}.

We now turn to the estimates for $\partial_t v$. Recall that
\begin{align}\label{eq:ptuu}
\partial_t v = \dot v - (u\cdot\nabla)v,
\qquad
\nabla\partial_t v = \nabla\dot v - \nabla\big[(u\cdot\nabla)v\big].
\end{align}
From the first identity in \eqref{eq:ptuu}, we obtain
\begin{align}
\|\partial_t v(t)\|_{L^2(K)}
&\le
\|\dot v(t)\|_{L^2(K)}
+ \|u(t)\|_{L^\infty(K)}\,\|\nabla v(t)\|_{L^2(\mathbb{R}^2)},
\end{align}
which belongs to $L^2(\varepsilon,T)$ by \eqref{eq:Lloc} and
\eqref{eq:gradientB} applied to both $v$ and $u$. In a similar way,
\begin{align}
\|\nabla \partial_t v(t)\|_{L^2(K)}
&\le
\|\nabla \dot{v}(t)\|_{L^2(\mathbb{R}^2)}
+ \|\nabla[(u\cdot\nabla)v](t)\|_{L^2(K)}.
\end{align}
The first term belongs to $L^\infty (\varepsilon,T)$ by \eqref{eq:gradientB}.
For the second one, we estimate
\begin{align}
\|\nabla[(u\cdot\nabla)v]\|_{L^2(K)}
&\le
\|\nabla u\|_{L^4(\mathbb{R}^2)} \|\nabla v\|_{L^4(\mathbb{R}^2)}
+ \|u\|_{L^\infty(K)}\,\|\nabla^2 v\|_{L^2(\mathbb{R}^2)}
\end{align}
which belongs to $L^\infty (\varepsilon,T)$ by \eqref{eq:BarC-INS}, \eqref{eq:BarC-LS} and
\eqref{eq:Lloc}. Since the divergence-free condition is propagated in time, the claim
follows.
\end{proof}

\subsection{Conservation of momentum and tail estimates}
Fix $T>0$, thanks to \cref{prop:ImprovedDecay}, the density remains compactly supported in
space--time. More precisely there exists a compact set $K \subseteq \R^2$ (possibly depending on $T$), such that for all
$t \in [0,T)$,
\begin{equation}\label{eq:Kdef}
\operatorname{supp}\rho(t) \subseteq K .
\end{equation}
This property is sufficient to prove conservation of the total momentum for
immediately strong solutions. We emphasize that, for general initial densities,
momentum conservation in the whole space $\mathbb{R}^2$ is still unknown.

\begin{lemma}\label{lem:MomentCons}
  Let $(\rho,u)$ an immediately strong solution of \eqref{eq:INS} arising from \eqref{ass:weakData2} and $v$ an immediately strong solution to \eqref{eq:LSs} advected by $(\rho,u)$. Then, for almost every $t >0 $,
\begin{equation}\label{eq:momentum0}
\int_{\mathbb{R}^2} \rho(t,x)\,v(t,x)\,dx
=
\int_{\mathbb{R}^2} \rho_0(x)\,v_0(x)\,dx .
\end{equation}
Moreover,
\begin{equation}\label{eq:momentum}
\int_{\mathbb{R}^2} \rho(t,x)\,\dot v(t,x)\,dx = 0,
\qquad
\int_{\mathbb{R}^2} |x|\,|\rho(t,x)\,\dot u(t,x)|\,dx
\lesssim_K \|\rho \dot v(t)\|_{L^2(\mathbb{R}^2)} ,
\end{equation}
where $K$ is defined in \eqref{eq:Kdef}.
\end{lemma}

\begin{proof}
Fix $T>0.$ Since $v$ is an immediately strong solution of \eqref{eq:LSs}, it satisfies the weak
formulation of the momentum equation in the sense of \cref{def:Distsol}. By a
standard time--cutoff argument, for every
$\varphi \in C_{c,\sigma}^\infty(\mathbb{R}^2;\mathbb{R}^2)$ and for a.e.\ $t>0$,
\begin{equation}\label{eq:momentum_spatial_test}
\int_{\mathbb{R}^2} \rho(t)\,v(t)\cdot\varphi
-
\int_{\mathbb{R}^2} \rho_0\,v_0\cdot\varphi
=
-\int_0^t\!\!\int_{\mathbb{R}^2}
\rho\,u\otimes v : \nabla \varphi
+
\int_0^t\!\!\int_{\mathbb{R}^2}
\nabla v : \nabla \varphi .
\end{equation}

Let $N>0$ and let $\chi\in C_c^\infty(\mathbb{R}^2)$ be such that $\chi\equiv1$ on
$B(0,1)$ and $\chi\equiv0$ outside $B(0,2)$. Define
\begin{align}
\chi_N(x):=\chi\!\left(\frac{x}{N}\right),
\end{align}
so that $\chi_N\equiv1$ on $B(0,N)$, $\supp\chi_N\subset B(0,2N)$, and
$\|\nabla\chi_N\|_{L^\infty}\lesssim N^{-1}$. Fix $a\in\mathbb{R}^2$ and set
\begin{align}
\varphi_N(x):=\nabla^\perp\!\big((x\cdot a^\perp)\,\chi_N(x)\big).
\end{align}
Then $\dive\varphi_N=0$, $\supp\varphi_N\subset B(0,2N)$, and
$\varphi_N(x)\to a$ pointwise as $N\to\infty$. Moreover, defining
\begin{align}
A_N:=\{x\in\mathbb{R}^2:\ N\le |x|\le 2N\},
\end{align}
we have $\supp\nabla\varphi_N\subset A_N$ and
\begin{align}
\|\nabla\varphi_N\|_{L^2(\mathbb{R}^2)} \le C|a|,
\end{align}
with $C$ independent of $N$.

Using $\varphi_N$ as a test function in \eqref{eq:momentum_spatial_test}, and
observing that for $N$ large enough $\supp\rho(t)\cap A_N=\emptyset$, the
convective term vanishes and we obtain
\begin{equation}\label{eq:momentum_spatial_test2}
\int_{\mathbb{R}^2} \rho(t)\,v(t)\cdot\varphi_N
-
\int_{\mathbb{R}^2} \rho_0\,v_0\cdot\varphi_N
=
\int_0^t\!\!\int_{\mathbb{R}^2}
\nabla v : \nabla \varphi_N .
\end{equation}
By Cauchy--Schwarz,
\begin{align}
\left|\int_0^t\!\!\int_{\mathbb{R}^2} \nabla v : \nabla \varphi_N\right|
\le C|a|\,\|\nabla v\|_{L^2((0,t)\times A_N)},
\end{align}
which converges to zero as $N\to\infty$. Passing to the limit in \eqref{eq:momentum_spatial_test2} yields
\eqref{eq:momentum0}. Indeed, the limit is well defined since
$\rho v \in L^\infty((0,T);L^1(\mathbb{R}^2))$, as a consequence of the uniform compact support $\supp\rho(t)\subset K$.

Fix now $\varepsilon>0$ and define $v^n := v\,\chi_n$. By \cref{lem:B1},
\begin{align}
v^n \in H^1\bigl((\varepsilon,T); H^1(\mathbb{R}^2)\bigr),
\end{align}
and $v^n$ is an admissible test function for the weak formulation of the
transport equation. Therefore, for a.e.\ $s,t \in (\varepsilon,T)$,
\begin{align}\label{eq:Transportw}
\frac12 \int_{\mathbb{R}^2} \rho(t)\,v^n(t)
- \frac12 \int_{\mathbb{R}^2} \rho(s)\,v^n(s)
= \int_s^t\!\!\int_{\mathbb{R}^2}
\rho\,\partial_t v^n
+ \rho\,u \cdot \nabla v^n .
\end{align}
The regularity of $v$ given by \eqref{lem:B1} together with \eqref{eq:Kdef} allows us to
pass to the limit $n\to\infty$ in \eqref{eq:Transportw}. By
\eqref{eq:momentum0}, the left-hand side vanishes, and we obtain
\begin{align}
\int_s^t\!\!\int_{\mathbb{R}^2} \rho\,\dot v = 0
\qquad \text{for a.e.\ } s,t \in (\varepsilon,T).
\end{align}
By the freedom in the choice of $\eps>0$ and $T>0.$
we get the first identity in \eqref{eq:momentum}. The second estimate follows
from the compact support of $\rho(t)$ and the Cauchy--Schwarz inequality.
\end{proof}

The main consequence of \cref{lem:MomentCons} is the derivation of tail estimates for $(u,P)$.
These follow from the properties of the stationary Stokes system in $\R^2$
\begin{equation}\label{eq:Stokes}
    \begin{cases}
        -\nabla Q + \Delta w = f,\\
        \dive w = 0.
    \end{cases}
\end{equation}
with mean free source.

\begin{lemma}\label{lem:stokes}
Let $f\in L^{2}(\mathbb{R}^{2};\mathbb{R}^{2})$ satisfy
\begin{equation}\label{eq:fprop}
\int_{\mathbb{R}^{2}} f(y)\,dy = 0,
\qquad
\int_{\mathbb{R}^{2}} |y|\,|f(y)|\,dy \le \widehat C .
\end{equation}
Then there exists a distributional solution $(\tilde w,\tilde Q)$ of \eqref{eq:Stokes}
such that $\nabla \tilde w \in L^2(\mathbb R^2)$ and
\begin{align}
\abs{\tilde Q(x)} = \widehat C\,O(|x|^{-2}),
\qquad
\abs{\tilde w(x)} = \widehat C\,O(|x|^{-1}),
\qquad
\text{as } |x|\to\infty.
\end{align}
Moreover, if $(w,Q)$ is another distributional solution of \eqref{eq:Stokes}
such that $\nabla w \in L^2(\mathbb R^2)$, then there exists a constant vector
$c \in \mathbb R^2$ such that $w = \tilde w + c$ and $ \nabla Q = \nabla \tilde{Q}$
\end{lemma}

\begin{proof}
The fundamental solution of the two-dimensional stationary Stokes system is given by
(see \cite[IV.2.4]{Galdi2011IntroNS})
\begin{equation}\label{eq:Stokeslet2D}
q(x)= \frac{1}{2\pi}\,\frac{x}{|x|^{2}}, \qquad 
W(x) =
-\frac{1}{4\pi}
\left(
\log\frac{\mathrm{Id}}{|x|}
+
\frac{x \otimes x }{|x|^{2}}
\right).
\end{equation}
Here $\mathrm{Id}$ denotes the identity matrix in $\mathbb{R}^{2\times 2}$.
Convolving with the source term $f$, we define
\begin{equation}\label{eq:Stokeslet2D2}
\tilde{Q}(x) = \int_{\mathbb{R}^2} q(x-y)\cdot f(y)\,\dd y,
\quad
\tilde{w}(x) = \int_{\mathbb{R}^2} W(x-y)\,f(y)\,\dd y.
\end{equation}

Using Taylor expansions, we obtain, as $\abs{y} \ll |x|$ and $\abs{x} \to \infty$,
\begin{align}
\begin{aligned}
 q(x-y)& =O(|x|^{-1}) + y\,O(|x|^{-2}), \\ 
 W(x-y)& =O(\log|x|) + y\,O(|x|^{-1}), \\ 
\nabla W(x-y)& =O(|x|^{-1}) + y\,O(|x|^{-2}).
\end{aligned}
\end{align}
Therefore,
\begin{align}
\tilde{Q}(x)
=
O(|x|^{-1}) \int_{\mathbb{R}^2} f(y)\,\dd y
+
O(|x|^{-2}) \int_{\mathbb{R}^2} y \cdot f(y)\,\dd y,
\end{align}
which, by \eqref{eq:fprop}, yields the first estimate in \eqref{eq:Stokes}.
The estimate $\tilde{w}(x)=O(|x|^{-1})$ follows in the same way.

Again by the previous expansions, we have
$\nabla \tilde{w}(x)=O(|x|^{-2})$ as $|x|\to\infty$, hence
$\nabla \tilde{w}\in L^2(\mathbb{R}^2\setminus B_R(0))$ for $R$ big enough.
Moreover, by Stokes estimates,
\begin{align}
\|\nabla \tilde{Q}, \nabla^2 \tilde{w}\|_{L^2(\mathbb{R}^2)} \lesssim \|f\|_{L^2(\mathbb{R}^2)},
\end{align}
which implies $\nabla \tilde{w}\in L^2_{\mathrm{loc}}(\mathbb{R}^2)$ and therefore
$\nabla \tilde{w}\in L^2(\mathbb{R}^2)$.

Finally, let $(w,Q)$ be another distributional solution of \eqref{eq:Stokes} such that $\nabla w \in L^2(\mathbb{R}^2)$. Then $Q-\tilde Q$ is harmonic and has gradient in $L^2(\mathbb{R}^2)$ (by the Stokes estimates), hence it is constant.
Since $\nabla Q=\nabla \tilde Q$, it follows that $w-\tilde w$ is harmonic with gradient in $L^2(\mathbb{R}^2)$, and therefore it is constant.
\end{proof}

Collecting the two previous lemma \cref{lem:MomentCons} and \cref{lem:stokes} we get the following proposition.

\begin{proposition}\label{prop:tailest}
 Let $(\rho,u)$ an immediately strong solution of \eqref{eq:INS} arising from \eqref{ass:weakData2} and $v$ an immediately strong solution to \eqref{eq:LSs} advected by $(\rho,u)$. Then, up to redefining the pressure (of \eqref{eq:LSs}) by a constant, for almost
every $t>0 $ we have, as $|x|\to\infty$,
\begin{align}
\label{eq:pressure-decay}
|Q(t,x)|
&=
\|\rho\,\dot v(t)\|_{L^2(\mathbb{R}^2)}\,O(|x|^{-2}),\\
\label{eq:velocity-decay}
|v(t,x)|
&=
\|\rho\,\dot v(t)\|_{L^2(\mathbb{R}^2)}\,O(|x|^{-1}) + c(t).
\end{align}

\end{proposition}

\begin{proof}
Set $f(t,x):=(\rho\dot v)(t,x)$. By \cref{lem:MomentCons},
in particular by \eqref{eq:momentum}, the conditions \eqref{eq:fprop} are satisfied
for almost every $t>0$. Hence, there exists a distributional solution
$(\tilde w(t),\tilde Q(t))$ of
\begin{equation}
\begin{cases}
-\nabla \tilde Q(t) + \Delta \tilde w(t) = \rho \dot v(t),\\
\dive \tilde w(t) = 0,
\end{cases}
\end{equation}
enjoying the decay properties stated in \cref{lem:stokes}.

Since $\nabla v(t)\in L^2(\mathbb{R}^2)$, the uniqueness result in
\cref{lem:stokes} implies that
\begin{align}
v(t)=\tilde w(t)+c(t)
\quad\text{and}\quad
\nabla Q(t)=\nabla \tilde Q(t).
\end{align}
Therefore, up to redefining the pressure by an additive constant, we may assume $Q(t)=\tilde Q(t)$, which concludes the proof.
\end{proof}

\begin{corollary}\label{cor:tailEst}
Let $(\rho,u)$ an immediately strong solution of \eqref{eq:INS} arising from \eqref{ass:weakData2} and $v$ an immediately strong solution to \eqref{eq:LSs} advected by $(\rho,u)$. Let $h\in \dot H^1(\mathbb{R}^2)$ be divergence free and such that
$h(x)=O(1)$ as $|x|\to\infty$. Then, for almost every $t>0)$,
\begin{align}
\int_{\mathbb{R}^2} \rho \dot v(t,x)\cdot h(x)\,\dd x
=
- \int_{\mathbb{R}^2} \nabla v(t,x):\nabla h(x)\,\dd x .
\end{align}
\end{corollary}

\begin{proof}
Let $\chi_N$ be the cutoff function introduced in the proof of
\cref{lem:MomentCons}, and set $h_N:=h\,\chi_N$, so that $h_N\in L^2(\mathbb{R}^2)$.
Since for every $\varepsilon>0$ we have
\begin{align}\label{eq:moma.e.}
\rho \dot v + \nabla P = \Delta v
\quad\text{in } L^2((\varepsilon,\infty);L^2(\mathbb{R}^2)),
\end{align}
we may multiply \eqref{eq:moma.e.} against $h_N$. Integrating by parts and using that
$\dive h=0$, we obtain for almost every $t\in(0,\infty)$
\begin{align}
\int_{\mathbb{R}^2} \rho \dot v(t)\cdot h
+ \int_{\mathbb{R}^2} \nabla v(t):\nabla h
=
- \lim_{N\to\infty}
\Bigg(
\int_{\mathbb{R}^2} Q(t)\, h\cdot\nabla\chi_N
+
\int_{\mathbb{R}^2} \nabla v(t)\, (h\otimes\nabla\chi_N)
\Bigg).
\end{align}

By \cref{prop:tailest}, for almost every $t>0$ and for $|x|\ge N$ with big $N$,
\begin{align}
|Q(t,x)| \lesssim \frac{1}{N^2},
\qquad
|v(t,x)| \lesssim 1 .
\end{align}
Therefore,
\begin{align}
\int_{\mathbb{R}^2} Q(t)\, h\cdot\nabla\chi_N
&\le
\|Q(t)\,h\|_{L^1(A_N)}\,\frac{1}{N}
\lesssim
\frac{1}{N^2}\,|A_N|\,\frac{1}{N}
\;\longrightarrow\;0 ,
\end{align}
and
\begin{align}
\int_{\mathbb{R}^2} \nabla v(t)\,(h\otimes\nabla\chi_N)\,\dd x
&\le
\|\nabla v(t)\|_{L^2(A_N)}\,\|h\|_{L^2(A_N)}\,\frac{1}{N}
\lesssim
\|\nabla v(t)\|_{L^2(A_N)}
\;\longrightarrow\;0 .
\end{align}
This concludes the proof.
\end{proof}

\section{\texorpdfstring{Existence of $L^2(\mathbb{R}^2)$ solutions}{Existence of L2(R2) solutions}}
In this section we prove an existence result for \eqref{eq:INS}.
Throughout, we assume
\begin{align}\label{eq:initial-data-L2}
\rho_0 = \mathbf{1}_D,
\qquad 
D \subset \R^2 \text{ bounded Lipschitz domain},
\qquad 
u_0 \in L_\sigma^2(\R^2).
\end{align}

\subsection{Existence by lifting}
As described in the introduction, our strategy is to approximate the vacuum density by strictly positive
densities. For every $n \in \N$, we set
\begin{align}\label{eq:rho0-approx}
\rho_0^n \coloneqq \mathbf{1}_D + \frac{1}{n},
\end{align}
and denote by $(\rho_n,u_n)$ the unique solution of \eqref{eq:INS}
with initial data $(\rho_0^n,u_0)$.

By \cref{thm:sect2}, and in particular by \eqref{eq:A0u}, there exists a constant $C_u = C ( \|\sqrt{\rho^n_0}u_0\|_{L^2(\R^2)})$ such that
\begin{align}
\sup_{i\in\{0,1,2,3\}} \sup_{t>0} A_i^0(t,u_n)
\le 
C_u \; \|\sqrt{\rho_0^n}\,u_0\|_{L^2(\R^2)}.
\end{align}
Since
\begin{align}
\|\sqrt{\rho_0^n}\,u_0\|_{L^2(\R^2)}
\le 
2 \|u_0\|_{L^2(\R^2)},
\end{align}
and since $C_u$ is non-decreasing in its argument, there exists a constant $C>0$, depending only on
$\|u_0\|_{L^2(\R^2)}$, such that
\begin{align}\label{eq:uniform-Ai-un}
\sup_{i\in\{0,1,2,3\}} \sup_{t>0} A_i^0(t,u_n)
\le C.
\end{align}

Thanks to the uniform estimates \eqref{eq:uniform-Ai-un} we can show compactness properties of the sequence $(\rho_n,u_n)$.

\begin{proposition}[Existence with $L^2$ velocity]\label{prop:existenceL2}
Let $(\rho_0,u_0)$ satisfy \eqref{eq:initial-data-L2}, and let $(\rho_n,u_n)$ be the unique solution to \eqref{eq:INS} associated with the lifted initial data $(\rho_0^n,u_0)$.
Then, up to a subsequence, $(\rho_n,u_n)$ converges to a pair $(\rho,u)$ which is a Leray-Hopf solution to \eqref{eq:INS} with initial data $(\rho_0,u_0)$.

Moreover, the following convergences hold in $\mathcal D'\bigl((0,\infty)\times\R^2\bigr)$:
\begin{align}\label{eq:distributional-convergences}
 u_n \to  u,
\qquad\sqrt{\rho_n} u_n \to \sqrt{\rho} u,
\qquad
\dot u_n \to \dot u,
\qquad
\sqrt{\rho_n} \dot u_n \to \sqrt{\rho} \dot u.
\end{align}
\end{proposition}

\begin{proof}
For the reader's convenience we recall that, thanks to the uniform estimates \eqref{eq:uniform-Ai-un}, the following bounds hold uniformly in \(n\) and \(t>0\):
\begin{align}\label{eq:nreadocnv}
\begin{aligned}
&\|\sqrt{\rho_n}\,u_n(t)\|_{L^2(\R^2)}^2
+ t\|\nabla u_n(t)\|_{L^2(\R^2)}^2
+ t^2\|\sqrt{\rho_n}\dot u_n(t), \nabla^2 u_n(t)\|_{L^2(\R^2)}^2
+ t^3\|\nabla \dot u_n(t)\|_{L^2(\R^2)}^2 
\le C,
\\
& \int_0^t
\Big(
\|\nabla u_n\|_{L^2(\R^2)}^2
+ s\|\sqrt{\rho_n}\dot u_n\|_{L^2(\R^2)}^2
+ s^2\|\nabla \dot u_n\|_{L^2(\R^2)}^2\Big)\,ds
\le C.
\end{aligned}
\end{align}
Fix \(T>0\) and \(0<\delta<T\).

\underline{Step 1.} Additional estimates for \((\rho_n,u_n)\).

Choose \(R>0\) sufficiently large. We claim that, for every \(w\) and almost every $t \in (0,T)$
\begin{align}\label{eq:step1}
\|w\|_{L^2(B_R)}
\lesssim
\|\nabla w\|_{L^2(\R^2)}
+
\|\sqrt{\rho_n(t)}w\|_{L^2(\R^2)},
\end{align}
where the implicit constant may depend on \(T\), \(R\), \(D\), and \(\|u_0\|_{L^2(\R^2)}\), but is independent of \(n\).
We postpone the proof of \eqref{eq:step1} to the final step.

Choosing \(w=u_n(t)\) in \eqref{eq:step1} and multiplying by \(\sqrt{t}\), we obtain
\begin{align}\label{eq:L2loc_un}
\sqrt{t}\,\|u_n(t)\|_{L^2(B_R)}
\lesssim
\sqrt{t}\,\|\nabla u_n(t)\|_{L^2(\R^2)}
+
\sqrt{t}\,\|\sqrt{\rho_n}u_n(t)\|_{L^2(\R^2)}.
\end{align}
Combining \eqref{eq:L2loc_un} with \eqref{eq:nreadocnv}, we deduce, uniformly in \(n\in\mathbb N\),
\begin{align}\label{eq:local-apriori-goal1}
\sqrt{t}\,u_n \in L^\infty\bigl((0,T);H^1(B_R)\bigr),
\qquad
u_n \in L^\infty\bigl((\delta,T);H^2(B_R)\bigr).
\end{align}
Arguing in the same way with \(w=\dot u_n(t)\), we infer
\begin{align}\label{eq:local-apriori-goal2}
t\,\dot u_n \in L^2\bigl((0,T);H^1(B_R)\bigr),
\qquad
\dot u_n \in L^\infty\bigl((\delta,T);H^1(B_R)\bigr).
\end{align}
We now estimate the time derivative \(\partial_t u_n\), uniformly in \(n\in\mathbb N\). Since
\begin{align}
\partial_t u_n=\dot u_n-(u_n\cdot\nabla)u_n,
\end{align}
it remains to control the nonlinear term.

First, by Hölder's inequality,
\begin{align}
s\,\|(u_n(s)\cdot\nabla u_n(s))\|_{L^1(B_R)}
&\lesssim
s\,\|u_n(s)\|_{L^2(B_R)}\|\nabla u_n(s)\|_{L^2(B_R)}
\lesssim
\bigl(\sqrt{s}\,\|u_n(s)\|_{H^1(B_R)}\bigr)^2,
\end{align}
and the right-hand side belongs to \(L^\infty(0,T)\) thanks to \eqref{eq:local-apriori-goal1}.

Moreover, for every \(s\in(\delta,T)\),
\begin{align}
\|\nabla[(u_n(s)\cdot\nabla)u_n(s)]\|_{L^1(B_R)}
&\lesssim
\|\nabla u_n(s)\|_{L^2(B_R)}^2
+
\|u_n(s)\|_{L^2(B_R)}\|\nabla^2 u_n(s)\|_{L^2(B_R)},
\end{align}
which is bounded in \(L^\infty(\delta,T)\) thanks to \eqref{eq:local-apriori-goal1}.

Combining these bounds with \eqref{eq:local-apriori-goal2}, we conclude that
\begin{align}\label{eq:local-apriori-goal3}
t\,\partial_t u_n \in L^2\bigl((0,T);L^1(B_R)\bigr),
\qquad
\partial_t u_n \in L^\infty\bigl((\delta,T);W^{1,1}(B_R)\bigr).
\end{align}
\underline{Step 2.} Local strong convergence of $(u_n)_{n\in\N}$ and $(\rho_n)_{n\in\N}$.

Let $R>0$ and choose $\varepsilon\in\bigl(0,\frac12\bigr)$.
For every $n\in\N$, define
\begin{align}
v_n(t)\coloneqq t^{\frac12+\varepsilon}u_n(t).
\end{align}
By \eqref{eq:local-apriori-goal1} and \eqref{eq:local-apriori-goal3},
\begin{align}
v_n \in L^\infty\bigl((0,T);H^1(B_R)\bigr),
\qquad
\partial_t v_n \in L^q\bigl((0,T);L^1(B_R)\bigr)
\end{align}
for some $q>1$.
Hence, by the Aubin--Lions lemma, up to a subsequence,
\begin{align}
v_n\to v
\quad\text{in } C\bigl([0,T];L^2(B_R)\bigr).
\end{align}
Setting $u(t)\coloneqq t^{-(\frac12+\varepsilon)}v(t)$, we infer that
\begin{align}
u_n\to u
\quad\text{in } L^1\bigl((0,T);L^2(B_R)\bigr).
\end{align}
By a diagonal argument, we conclude that
\begin{align}\label{eq:u-nconv21}
u_n\to u
\quad\text{in }
L^1_{\loc}\bigl([0,\infty);L^2_{\loc}(\R^2)\bigr).
\end{align}
Next, by \cref{lem:Linfty} (in particular \eqref{eq:growuptw}) with $v = u_n$, together with the uniform estimates in \eqref{eq:uniform-Ai-un}, there exists a constant $C'>0$, depending only on $T$, $D$, and $\|u_0\|_{L^2(\R^2)}$ (uniformly in $n$), such that for every $x\in\R^2$,
\begin{align}
\sup_{t\in(0,T)} t^{3/4}\frac{|u_n(t,x)|}{1+|x|}
\le C'.
\end{align}
Passing to the limit and using \eqref{eq:u-nconv21}, we obtain
\begin{align}
\frac{|u(t,x)|}{1+|x|}
\le \liminf_{n\to\infty}\frac{|u_n(t,x)|}{1+|x|}
\le C' t^{-3/4}.
\end{align}
Since $t^{-3/4}\in L^1_{\loc}[0,\infty)$, it follows that
\begin{align}
\frac{|u(t,x)|}{1+|x|}
\in L^1_{\loc}\bigl([0,\infty);L^\infty(\R^2)\bigr).
\end{align}
Then, by the DiPerna--Lions theory \cite[Theorem II.4.1]{DipLi89}, there exists $\rho\in C_{w^*}\bigl([0,\infty);L^\infty(\R^2)\bigr)$
such that
\begin{align}\label{eq:rho-nconv}
\rho_n\to \rho
\qquad \text{in } C\bigl([0,\infty);L^p_{\loc}(\R^2)\bigr),
\qquad \forall p\in[1,\infty).
\end{align}
By combining \eqref{eq:u-nconv21} and \eqref{eq:rho-nconv}, and using that
$\rho_n$ is uniformly bounded, we obtain\footnote{The same argument also yields \eqref{eq:sqrt-rhou-nconv} with $\rho_n$ in place of $\sqrt{\rho_n}$.}
\begin{align}\label{eq:sqrt-rhou-nconv}
\sqrt{\rho_n}\, u_n \to \sqrt{\rho}\, u
\qquad\text{in }
L^1_{\loc}\bigl([0,\infty);L^2_{\loc}(\R^2)\bigr).
\end{align}
Indeed, writing
\begin{align}
\sqrt{\rho_n}\,u_n-\sqrt{\rho}\,u
=\sqrt{\rho_n}(u_n-u)+(\sqrt{\rho_n}-\sqrt{\rho})u,
\end{align}
the first term converges by \eqref{eq:u-nconv21}, while the second one is handled by truncating $u$
and using the strong convergence of $\rho_n$ in $C([0,\infty);L^p_{\loc})$.
 Thanks to \eqref{eq:u-nconv21} and \eqref{eq:sqrt-rhou-nconv} we obtain the first two convergences in \eqref{eq:distributional-convergences}.

\underline{Step 3.} $(\rho,u)$ is a Leray--Hopf weak solution.

Thanks to the uniform estimates \eqref{eq:nreadocnv} and the identification of the limit \eqref{eq:sqrt-rhou-nconv}, we obtain
\begin{align}\label{eq:weakconv}
\begin{aligned}
\nabla u_n \rightharpoonup \nabla u 
&\qquad \text{in } L^2\bigl((0,\infty);L^2(\R^2)\bigr),\\
\sqrt{\rho_n}\, u_n \rightharpoonup^* \sqrt{\rho}\, u 
&\qquad \text{in } L^\infty\bigl((0,\infty);L^2(\R^2)\bigr).
\end{aligned}
\end{align}
Together with \eqref{eq:rho-nconv} and \eqref{eq:sqrt-rhou-nconv}, these convergences allow us to pass to the limit in the weak formulation of $(\rho_n,u_n)$.
Hence $(\rho,u)$ is a weak solution.

To conclude that $(\rho,u)$ is a Leray--Hopf solution, it remains to prove the energy inequality. The additional regularity given by \eqref{eq:local-apriori-goal1} and \eqref{eq:local-apriori-goal3} yields
\begin{align}
u_n \in L^\infty\bigl((\delta,T);H^2(B_R)\bigr),
\qquad
\partial_t u_n \in L^\infty\bigl((\delta,T);W^{1,1}(B_R)\bigr),
\end{align}
for every $R>0$. 
By the Aubin--Lions lemma, up to extraction and a diagonal argument, we obtain
\begin{align}\label{eq:u-nconv22}
u_n \to u
\quad\text{in } 
C\bigl((0,\infty);H^1_{\loc}(\R^2)\bigr),
\end{align}
which in particular implies
\begin{align}\label{eq:u-nconv23}
u_n \to u
\quad\text{in } 
C\bigl((0,\infty);L^p_{\loc}(\R^2)\bigr),
\qquad \forall p\in[1,\infty).
\end{align}
Combining \eqref{eq:u-nconv23} with \eqref{eq:rho-nconv}, we deduce
\begin{align}\label{eq:sqrt-rhou-conv2}
\sqrt{\rho_n}\,u_n \to \sqrt{\rho}\,u
\qquad \text{in } C\bigl((0,\infty);L^p_{\loc}(\R^2)\bigr) \qquad \forall p\in[1,\infty),
\end{align}
and thanks to \eqref{eq:weakconv} we also get
\begin{align}\label{eq:sqrt-rhou-conv3}
\sqrt{\rho_n}\,u_n(t) \rightharpoonup \sqrt{\rho}\,u(t)
\qquad \text{weakly in } L^2(\R^2),
\qquad \forall t>0.
\end{align}
Using \eqref{eq:sqrt-rhou-conv3}, \eqref{eq:weakconv}, and the energy inequality satisfied by $(\rho_n,u_n)$, we infer that
\begin{align}
    \frac12 \|\sqrt{\rho(t)}\,u(t)\|_{L^2(\R^2)}^2
    &+ \int_0^t \|\nabla u(s)\|_{L^2(\R^2)}^2 \, \dd s \\
    &\le \liminf_{n\to\infty}
    \Bigg(
        \frac12 \|\sqrt{\rho_n(t)}\,u_n(t)\|_{L^2(\R^2)}^2
        + \int_0^t \|\nabla u_n(s)\|_{L^2(\R^2)}^2 \, \dd s
    \Bigg) \\
    &\le \liminf_{n\to\infty}
    \frac12 \|\sqrt{\rho_0^n}\,u_0\|_{L^2(\R^2)}^2
    = \frac12 \|\sqrt{\rho_0}\,u_0\|_{L^2(\R^2)}^2 .
\end{align}

\underline{Step 4.}  Weak convergences.

By \eqref{eq:u-nconv22}, we have
\begin{align}
(u_n\cdot\nabla)u_n \to (u\cdot\nabla)u
\qquad \text{in } C\bigl((0,\infty); L_\loc^1( \R^2)\bigr).
\end{align}
Consequently, we obtain the following identification:
\begin{align}
\dot{u}_n
\to
\partial_t u+(u\cdot\nabla)u
\qquad
\text{in } \mathcal D'((0,\infty)\times \R^2).
\end{align}
This yields the third convergence in \eqref{eq:distributional-convergences}. 
Moreover, by Banach--Alaoglu and the second estimate in \eqref{eq:local-apriori-goal2}, up to extraction of a subsequence and by a diagonal argument,
\begin{align}
\dot u_n \rightharpoonup^* f
\qquad \text{in } L^\infty_{\loc}\bigl((0,\infty);H^1_{\loc}(\R^2)\bigr).
\end{align}
By uniqueness of the distributional limit, we conclude that \(f=\dot u\). 
In particular,
\begin{align}\label{eq:dotun_weakstar_local}
\dot u_n \rightharpoonup^* \dot{u}
\qquad 
\text{in } L^\infty_{\loc}\bigl((0,\infty);L^p_{\loc}(\R^2)\bigr),
\qquad \forall p \in [1,\infty).
\end{align}
Combining \eqref{eq:dotun_weakstar_local} with the strong convergence of the density \eqref{eq:rho-nconv}, we conclude
\begin{align}
\sqrt{\rho_n}\dot{u}_n
\to \sqrt{\rho}\,\dot{u}
\qquad
\text{in } \mathcal D'((0,\infty)\times \R^2).
\end{align}
This concludes the convergences in \eqref{eq:distributional-convergences}.

\underline{Step 5.} Proof of \eqref{eq:step1}.

Set
\begin{align}
\widetilde{\rho}_n \coloneqq \rho_n - \frac{1}{n}.
\end{align}
For every function $w$,
\begin{align}\label{eq:rhotilde1}
\|\sqrt{\widetilde{\rho}_n}\,w\|_{L^2(\R^2)}
\le
\|\sqrt{\rho_n}\,w\|_{L^2(\R^2)}.
\end{align}
Taking $w=u_n(t)$ and using the energy inequality for $(\rho_n,u_n)$, we get
\begin{align}\label{eq:rhotilde}
\|\sqrt{\widetilde{\rho}_n}\,u_n(t)\|_{L^2(\R^2)}
\le
2\|u_0\|_{L^2(\R^2)}.
\end{align}

The estimate \eqref{eq:rhotilde} allows us to apply \cref{prop:localLp}.
Since $\widetilde{\rho}_n$ solves the transport equation with velocity field $u_n$
and initial datum $\rho_0$, \cref{prop:localLp} with $R=0$ yields, for every $w$
and every $x\in D$,
\begin{align}\label{eq:localLp_balls}
\|w\|_{L^2(B_{L\sqrt t}(x))}
\lesssim
\sqrt{t}\,\|\nabla w\|_{L^2(B_{L\sqrt t}(x))}
+
\|\widetilde{\rho}_n(t)w\|_{L^2(B_{L\sqrt t}(x))},
\end{align}
where $L>0$ depends only on $\|u_0\|_{L^2(\R^2)}$.
By a covering argument and \eqref{eq:rhotilde1}, we infer
\begin{align}\label{eq:uL2DD}
\|w\|_{L^2(D)}
\lesssim
\sqrt{t}\,\|\nabla w\|_{L^2(\R^2)}
+
\|\rho_n(t)w\|_{L^2(\R^2)}.
\end{align}

Fix $R_0>0$ such that $D\subset B_{R_0}(0)$ and let $R>R_0$.
Denote
\begin{align}
w^{B_R}\coloneqq \fint_{B_R} w(x)\,\dd x,
\qquad
w^D\coloneqq \fint_D w(x)\,\dd x.
\end{align}
By Poincar\'e's inequality,
\begin{align}
\|w\|_{L^2(B_R)}
&\le
\|w-w^{B_R}\|_{L^2(B_R)}
+
|B_R|^{1/2}|w^{B_R}| \lesssim
\|\nabla w\|_{L^2(B_R)}
+
|w^{B_R}|.
\end{align}
Moreover,
\begin{align}
|w^{B_R}|
\le
|w^{B_R}-w^D|
+
|w^D|
\lesssim
\|\nabla w\|_{L^2(B_R)}
+
\|w\|_{L^2(D)}.
\end{align}
Combining the above estimates with \eqref{eq:uL2DD}, and using $t \in (0,T)$,
we conclude that
\begin{align}
\|w\|_{L^2(B_R)}
\lesssim
\|\nabla w\|_{L^2(\R^2)}
+
\|\rho_n(t)w\|_{L^2(\R^2)}.
\end{align}

\end{proof}
\begin{remark}\label{rem:nconvergence}
We recall that by \eqref{eq:sqrt-rhou-conv2} and \eqref{eq:weakconv}, we obtain that for every \(t>0\),
\begin{align}
\begin{aligned}
\sqrt{\rho_n}\,u_n(t)
&\to
\sqrt{\rho}\,u(t)
&&\text{in } L^2_{\loc}(\R^2),
\\
\nabla u_n
&\rightharpoonup \nabla u
&&\text{in } L^2\bigl((0,\infty)\times\R^2\bigr).
\end{aligned}
\end{align}
However, the combination of these two convergences is not sufficient to pass to the limit in the energy equality satisfied by \((\rho_n,u_n)\). Indeed, in order to justify such a passage to the limit, one would need either a global version of the first convergence or a strong version of the second one, namely
\begin{align}
\begin{aligned}
\sqrt{\rho_n}\,u_n(t)
&\to
\sqrt{\rho}\,u(t)
&&\text{in } L^2(\R^2),\\
\nabla u_n
&\to \nabla u
&&\text{in } L^2\bigl((0,\infty)\times\R^2\bigr).
\end{aligned}
\end{align}
\end{remark}

\begin{remark}\label{rem:Immstrong}
The solution constructed in \cref{prop:existenceL2} is indeed immediately strong. 
To see this, recall the following general fact. Let $\alpha\ge0$, $p,q\in[1,\infty]$, and
\begin{align}
X=L^p\big((0,\infty),t^\alpha dt;L^q(\R^2)\big).
\end{align}
Assume that for some $m\in\N$
\begin{align}
f_n \to f \quad \text{in } \mathcal D'((0,\infty)\times\R^2),
\qquad
\sup_n \|\nabla^m f_n\|_X < \infty .
\end{align}
Then $\nabla^m f\in X$ and, up to a subsequence,
\begin{align}
\nabla^m f_n \rightharpoonup \nabla^m f 
\quad \text{weakly (or weakly-$*$) in } X,
\end{align}
and
\begin{align}
\|\nabla^m f\|_X 
\le 
\liminf_{n\to\infty}\|\nabla^m f_n\|_X .
\end{align}

Applying this with $f_n=\sqrt{\rho_n} u_n$, $u_n$, $\dot u_n$, and $\sqrt{\rho_n}\dot u_n$, 
and using the distributional convergences in \eqref{eq:distributional-convergences},
we obtain 
\begin{align}
   \sup_{i \in \{0,1,2,3\}} \sup_{t>0}  A_i^0(t,u)
\le 
   \sup_{i \in \{0,1,2,3\}} \sup_{t>0}  \liminf_{n\to\infty} A_i^0(t,u_n)
\le C,
\end{align}
where $C$ is independent of $n$ by \eqref{eq:uniform-Ai-un}.

\end{remark}

\begin{remark}
   In Section~2 we were able to show that smooth solutions also satisfy the following estimates uniformly in $t>0$:
   \begin{align}\label{eq:nosmooth}
    \int_0^t s^3 \| \nabla \dot{P}, \sqrt{\rho}\,\ddot{u} \|_2 \, ds \leq C.
   \end{align}
   Due to the lack of even higher-order estimates (such as $A_4$), it is unclear how to obtain sufficient compactness to show that non-smooth solutions also satisfy \eqref{eq:nosmooth}.
\end{remark}

\subsection{Refined existence via atomic decomposition}

We now aim to prove that the solution $(\rho,u)$ constructed in \cref{prop:existenceL2}
satisfies the energy equality and that
\begin{align}
\rho u \in C([0,\infty);L^2(\R^2)).
\end{align}
The argument is based on an atomic decomposition of the initial velocity.

Recall that if $u_0 \in L^2(\mathbb{R}^2)$, then for every $\eta \in (0,1)$ there exists a sequence $(u_{0,j})_{j \in \mathbb{Z}}$ in
$\dot{H}^{\eta}(\mathbb{R}^2) \cap \dot{H}^{-\eta}(\mathbb{R}^2)$ such that
\begin{align}\label{eq:decompositionL2}
u_0 = \sum_{j \in \mathbb{Z}} u_{0,j}, \qquad
\sum_{j \in \mathbb{Z}}
\left(
2^{-j}\norm{u_{0,j}}_{\dot{H}^{\eta}(\mathbb{R}^2)}^2
+
2^{j}\norm{u_{0,j}}_{\dot{H}^{-\eta}(\mathbb{R}^2)}^2
\right) \leq C \| u_0\|_{L^2(\R^2)}.
\end{align}
For each atom $u_{0,j}$, we then consider the linearized system transported by $(\rho,u)$ with initial datum $u_{0,j}$:
\begin{align}
\begin{cases}\label{eq:LSj}
\partial_t (\rho u_j) + \dive (\rho u \otimes u_j) + \nabla P_j = \Delta u_j,\\
\dive u_j = 0,\\
u_j(0) = u_{0,j}.
\end{cases}
\end{align}

In \cref{lem:exLS} we prove that, for every $j \in \Z$, there exists a solution $u_j$ to \eqref{eq:LSj} such that
\begin{align}
\rho u_j \in C([0,\infty);L^2(\R^2)),
\end{align}
and the corresponding energy equality holds. This is a consequence of the additional regularity of the initial data, namely
$u_{0,j}\in \dot{H}^{\eta}(\mathbb{R}^2)$.

Finally, in \cref{prop:gluinglevel0} we show that the partial sums
\begin{align}\label{eq:psum}
\sum_{|j|\le J} u_j 
\end{align}
converge to $u$ in a stronger sense with respect to the convergence of $u_n$ to $u$. As a consequence, $u$ inherits the same continuity and energy properties, which improves the conclusion of \cref{prop:existenceL2}.

\begin{lemma}\label{lem:exLS}
Let $(\rho_0,u_0)$ be as in \eqref{eq:initial-data-L2}, and let $(\rho,u)$ be the solution to \eqref{eq:INS} with initial data $(\rho_0,u_0)$ constructed in \cref{prop:existenceL2}. 
For every $j\in\mathbb{Z}$, let $u_{0,j}$ be as in \eqref{eq:decompositionL2}. 

Then, for every $j\in\mathbb{Z}$, there exists an immediately strong solution $u_j$ to \eqref{eq:LSj}. Moreover, there exists a constant $C>0$, independent of $j$, such that
\begin{align}
\sup_{i\in\{0,1,2,3\}}\sup_{t>0} A_i^0(t,u_j,u)
&\le
C\,\|\sqrt{\rho_0}\,u_{0,j}\|_{L^2(\R^2)},
\label{eq:boundsujL2}\\
\sup_{i\in\{0,1,2,3\}}\sup_{t>0} A_i^\eta(t,u_j,u)
&\le
C\,\|u_{0,j}\|_{\dot H^\eta(\R^2)}.
\label{eq:boundsujHeta}
\end{align}
In addition, $\sqrt{\rho}\,u_j\in C([0,\infty);L^2(\R^2))$ and $u_j$ satisfies the energy equality.

Finally, there exists a constant $C>0$, independent of $j$, such that for every $0<s<T<\infty$,
\begin{align}
\begin{aligned}\label{eq:decayAoiuj}
\sup_{t\in (s,T)} (t-s)\|\nabla u_j(t)\|_{L^2(\R^2)}^2
&\le C \int_s^T \|\nabla u_j(\tau)\|_{L^2(\R^2)}^2\,\dd\tau, \\
\sup_{t\in (s,T)} (t-s)^2\|\sqrt{\rho}\,\dot u_j(t)\|_{L^2(\R^2)}^2
&\le C \int_s^T \|\nabla u_j(\tau)\|_{L^2(\R^2)}^2\,\dd\tau, \\
\sup_{t\in (s,T)} (t-s)^3\|\nabla \dot u_j(t)\|_{L^2(\R^2)}^2
&\le C \int_s^T \|\nabla u_j(\tau)\|_{L^2(\R^2)}^2\,\dd\tau.
\end{aligned}
\end{align}
\end{lemma}

\begin{proof}
We lift the initial density as in \eqref{eq:rho0-approx}, obtaining a sequence of initial data $(\rho_0^n,u_0)$. 
For each $n\in\mathbb{N}$ we consider the corresponding solution $(\rho_n,u_n)$ arising from $(\rho_0^n,u_0)$.

\underline{Step 1.} Decay estimates.

For each $n\in\mathbb{N}$ and $j\in\mathbb{Z}$, we consider \eqref{eq:LSj} with $\rho$ and $u$ replaced by $\rho_n$ and $u_n$, and with initial datum $u_{0,j}$. 
This yields a unique solution $u_{n,j}$.

By \cref{thm:sect2}, for every $n\in\mathbb{N}$, every $j\in\mathbb{Z}$, every $i\in\{0,1,2,3\}$, and every $t\in[0,\infty)$ we have
\begin{align}
\begin{aligned}\label{eq:unifEstj}
A_i^0(t,u_{n,j},u_n)
&\le C_{u_n}\,\|\sqrt{\rho_0^n}\,u_{0,j}\|_{L^2(\R^2)}^{2}
\le C\,\|\sqrt{\rho_0^n}\,u_{0,j}\|_{L^2(\R^2)}^{2}, \\
A_i^1(t,u_{n,j},u_n)
&\le C_{u_n}\,\|u_{0,j}\|_{\dot{H}^1(\R^2)}^{2}
\le C\,\|u_{0,j}\|_{\dot{H}^1(\R^2)}^{2}.
\end{aligned}
\end{align}

Here we used that $C_{u_n}$ depends in a nondecreasing way on  $\| \sqrt{\rho_0^n}u_{0,j}\|_{L^2(\R^2)}$ which can be controlled by $ \| u_{0,j}\|_{L^2(\R^2)}$. 
Therefore $C_{u_n}$ can be bounded by a constant $C$ independent of $n$.

Since the a priori bounds for $u_{n,j}$ are uniform in $n$ and coincide with those available for $u_n$ in \eqref{eq:uniform-Ai-un}, we can reproduce the proof of \cref{prop:existenceL2}. 
Hence the sequence $(\rho_n, u_{n,j})$ enjoys the same compactness properties as $(\rho_n,u_n)$ in \cref{prop:existenceL2}. 
In particular, up to a subsequence, it converges to a pair $(\rho,u_j)$ which is an immediately strong solution to \eqref{eq:INS} with initial data $(\rho_0,u_{0,j})$.

Moreover, the following convergences hold in $\mathcal D'\bigl((0,\infty)\times\R^2\bigr)$:
\begin{align}
 u_{n,j} \to  u_j,
\qquad
\sqrt{\rho_n} u_{n,j} \to \sqrt{\rho} u_j,
\qquad
\dot u_{n,j} \to \dot u_j,
\qquad
\sqrt{\rho_n} \dot u_{n,j} \to \sqrt{\rho} \dot u_j.
\end{align}

Arguing as in \cref{rem:Immstrong}, we obtain
\begin{align}
A_i^0(t,u_j,u)
\le
\liminf_{n\to\infty} A_i^0(t,u_{n,j},u_n)
\le
C \liminf_{n\to\infty}
\|\sqrt{\rho_0^n}\,u_{0,j}\|_{L^2(\R^2)}^2
=
C\,\|\sqrt{\rho_0}\,u_{0,j}\|_{L^2(\R^2)}^2,
\end{align}
where we used \eqref{eq:unifEstj} and the convergence
\begin{align}\label{eq:intimeconv}
\lim_{n \to \infty} \| \sqrt{\rho_0^n}\,u_{0,j} - \sqrt{\rho_0}\,u_{0,j}\|_{L^2(\R^2)} = 0.
\end{align}
This yields \eqref{eq:boundsujL2}.

Similarly, passing to the limit in the estimate involving $\|u_{0,j}\|_{\dot H^1(\R^2)}$, we obtain the corresponding $A_i^1$ bound. 
Since the constant $C$ does not depend on $j$, interpolating between the $A_i^0$ and $A_i^1$ bounds yields \eqref{eq:boundsujHeta}.

\underline{Step 2.} Additional regularity.

Since $u_{0,j} \in \dot{H}^\eta(\R^2)$ we have the additional regularity of $A_1^\eta(t,u_j,u)$. For instance we have
\begin{align}
\int_0^\infty t^{\eta-1}\|\nabla P_j(t)\|_{L^2(\R^2)}^2\,dt <\infty,
\end{align}
which implies
\begin{align}\label{eq:pjest}
\nabla P_j \in L^1_{\loc}([0,\infty);L^2(\R^2)).
\end{align}
Similarly, by the bounds on $A_1^\eta(\rho,u,u_j)$ and $A_2^\eta(\rho,u,u_j)$ we have
\begin{align}\label{eq:scen2Hyp2}
\sqrt{t}\|\nabla u_j(t)\|_{L^2(\R^2)} \to 0, 
\qquad
\rho \dot{u}_j,\; \sqrt{t}\,\nabla \dot{u}_j \in L^1_{\loc}((0,\infty);L^2(\R^2)).
\end{align}
Moreover, by \eqref{eq:L2Lip} and the bound on $A_3^\eta(\rho,u,u_j)$ we obtain
\begin{align}
\|\nabla u_j(t)\|_{L^\infty(\R^2)}
&\lesssim
\|\nabla^2 u_j(t)\|_{L^2(\R^2)}
+
\|\nabla u_j(t)\|_{L^2(\R^2)}^{\frac12}
\|\nabla \dot u_j(t)\|_{L^2(\R^2)}^{\frac12} \lesssim
t^{-\frac{2-\eta}{2}}
+
t^{-\frac{1-\eta}{4}}
t^{-\frac{3-\eta}{4}} .
\end{align}
Since $\eta>0$, the right-hand side belongs to $L^1_{\loc}[0,\infty)$, and therefore
\begin{align}\label{eq:scen2Hyp3}
\nabla u_j \in L^1_{\loc}([0,\infty);L^\infty(\R^2)).
\end{align}

\underline{Step 3.} Continuity in time and energy equality.

The regularity \eqref{eq:pjest} allows us to test the momentum equation \eqref{eq:LSj} against a non-divergence-free test function 
$\varphi\in C_c^\infty(\R^2)$ and obtain
\begin{align}
\int_{\R^2}\rho(t)u_j(t)\cdot\varphi\,dx
-
\int_{\R^2}\rho_0u_{0,j}\cdot\varphi\,dx
=
\int_0^t\int_{\R^2}
\bigl[\rho u\otimes u_j+\nabla u_j\bigr]:\nabla\varphi
+
\nabla P_j\cdot\varphi
\,dx\,d\tau .
\end{align}

Since 
\begin{align}
\rho u\otimes u_j+\nabla u_j,\ \nabla P_j
\in L^1_{\loc}([0,\infty);L^2(\R^2)),
\end{align}
the right-hand side converges to $0$ as $t\to0^+$. Hence
\begin{align}
\sqrt{\rho}u_j(t)\rightharpoonup \sqrt{\rho_0}u_{0,j}
\qquad
\text{in }L^2(\R^2).
\end{align}
By the energy inequality for $u_j$ the continuity is indeed strong
\begin{align}\label{eq:scen2Hyp1}
\sqrt{\rho}u_j(t)\to \sqrt{\rho_0}u_{0,j}
\qquad
\text{in }L^2(\R^2)\quad\text{as }t\to0^+,
\end{align}
which proves the continuity at the initial time. 

Collecting \eqref{eq:scen2Hyp1}, \eqref{eq:scen2Hyp2}, and \eqref{eq:scen2Hyp3}, we see that all the assumptions of Scenario~2 in \cref{sce:scenario2} are satisfied with
\begin{align}
(\rho_1,u_1)=(\rho_2,u_2)=(\rho,u), \qquad v_1 =v_2 = u_j
\end{align}
Hence, by \cref{prop:Scen2}, we obtain
\begin{align}
\int_{\R^2}\rho(t)\,|u_j(t)|^2\,dx
=
\int_{\R^2}\rho_0\,|u_{0,j}|^2\,dx
-2\int_0^t\int_{\R^2} |\nabla u_j|^2\,dx\,ds,
\end{align}
which yields the energy equality. This extends the continuity  \eqref{eq:scen2Hyp1} to every time $t\geq 0$.

\underline{Step 4.} Strong convergence of the approximating sequence.

As shown in Step 3 of \cref{prop:existenceL2}, see \eqref{eq:weakconv} and \eqref{eq:sqrt-rhou-conv3}, for every $t>0$ one has
\begin{align}\label{eq:weakconv2}
\begin{aligned}
   \|\sqrt{\rho(t)}\,u_j(t)\|_{L^2(\R^2)}^2
    &\le \liminf_{n\to\infty}
    \|\sqrt{\rho_n(t)}\,u_{n,j}(t)\|_{L^2(\R^2)}^2, \\
   \int_0^t \|\nabla u_j(s)\|_{L^2(\R^2)}^2 \,\dd s
    &\le \liminf_{n\to\infty}
    \int_0^t \|\nabla u_{n,j}(s)\|_{L^2(\R^2)}^2 \,\dd s .
\end{aligned}
\end{align}
On the other hand, by the energy equalities for both $(\rho_n,u_{n,j})$ and $(\rho,u_j)$, together with \eqref{eq:intimeconv}, we obtain for every $t>0$
\begin{align}\label{eq:totsum}
\begin{aligned}
\lim_{n\to\infty} & 
\Bigg(
\frac12 \|\sqrt{\rho_n(t)}\,u_{n,j}(t)\|_{L^2(\R^2)}^2
+
\int_0^t \|\nabla u_{n,j}(s)\|_{L^2(\R^2)}^2 \,\dd s
\Bigg)
\\ & =
\frac12 \|\sqrt{\rho(t)}\,u_j(t)\|_{L^2(\R^2)}^2
+
\int_0^t \|\nabla u_j(s)\|_{L^2(\R^2)}^2 \,\dd s .
\end{aligned}
\end{align}
Combining \eqref{eq:weakconv2} and \eqref{eq:totsum}, we infer that for every $t>0$
\begin{align}\label{eq:strongConvnj}
\begin{aligned}
\lim_{n\to\infty}
\|\sqrt{\rho_n(t)}\,u_{n,j}(t)\|_{L^2(\R^2)}^2
&=
\|\sqrt{\rho(t)}\,u_j(t)\|_{L^2(\R^2)}^2, \\
\lim_{n\to\infty}
\int_0^t \|\nabla u_{n,j}(s)\|_{L^2(\R^2)}^2 \,\dd s
&=
\int_0^t \|\nabla u_j(s)\|_{L^2(\R^2)}^2 \,\dd s .
\end{aligned}
\end{align}
Let
\begin{align}
F_n(t):=\int_0^t \|\nabla u_{n,j}(s)\|_{L^2(\R^2)}^2 \,\dd s,
\qquad
F(t):=\int_0^t \|\nabla u_j(s)\|_{L^2(\R^2)}^2 \,\dd s.
\end{align}
Since $F_n$ and $F$ are continuous and nondecreasing, the pointwise convergence in \eqref{eq:strongConvnj} implies $F_n \to F $ in $C([0,\infty))$. Using again the energy equalities, we then deduce that
\begin{align}\label{eq:unifconvenergy}
\|\sqrt{\rho_n(\cdot)}\,u_{n,j}(\cdot)\|_{L^2(\R^2)}^2
\to
\|\sqrt{\rho(\cdot)}\,u_j(\cdot)\|_{L^2(\R^2)}^2
\qquad\text{in } C([0,\infty)).
\end{align}

\underline{Step 5.} Estimates \eqref{eq:decayAoiuj}.
 
To prove the first estimate, we recall that by \eqref{eq:conl1} we have
\begin{align}
\frac{d}{dt} \|\nabla u_{n,j}(t)\|_{L^2(\R^2)}^2 
\lesssim 
\|\nabla u_{n,j}(t)\|_{L^2(\R^2)}^2 
\, \|\nabla u_n(t)\|_{L^2(\R^2)}^2 .
\end{align}
Hence, for every $t>s>0$,
\begin{align}
\frac{d}{dt}\Bigl((t-s)\|\nabla u_{n,j}(t)\|_{L^2(\R^2)}^2\Bigr)
\lesssim 
\|\nabla u_{n,j}(t)\|_{L^2(\R^2)}^2
+
(t-s)\|\nabla u_{n,j}(t)\|_{L^2(\R^2)}^2
\|\nabla u_n(t)\|_{L^2(\R^2)}^2 .
\end{align}
Applying Gronwall's inequality, we obtain
\begin{align}
(t-s)\|\nabla u_{n,j}(t)\|_{L^2(\R^2)}^2
\lesssim 
\int_s^t \|\nabla u_{n,j}(\tau)\|_{L^2(\R^2)}^2\,\dd\tau
\exp\!\left(\int_0^t \|\nabla u_n(\tau)\|_{L^2(\R^2)}^2\,\dd\tau\right).
\end{align}
Since the exponential factor is uniformly bounded with respect to $n$ and $j$, we deduce that for every $t>s\geq0$,
\begin{align}
(t-s)\|\nabla u_{n,j}(t)\|_{L^2(\R^2)}^2
\le
C_u \int_s^t \|\nabla u_{n,j}(\tau)\|_{L^2(\R^2)}^2\,\dd\tau .
\end{align}
Then, similarly to \cref{rem:Immstrong}, we get for every  $0<s<T<\infty$
\begin{align}
 \sup_{t\in (s,T)} (t-s)\|\nabla u_j(t)\|_{L^2(\R^2)}^2
& \le \liminf_{n \to \infty }   \sup_{t\in (s,T)} (t-s)\|\nabla u_{n,j}(t)\|_{L^2(\R^2)}^2 \\ & \le C_u \liminf_{n \to \infty }   \sup_{t\in (s,T)}  \int_s^t \|\nabla u_{n,j}(\tau)\|_{L^2(\R^2)}^2\,\dd\tau  \\ & = C_u \liminf_{n \to \infty }   \int_s^T \|\nabla u_{n,j}(\tau)\|_{L^2(\R^2)}^2\,\dd\tau 
\end{align}
and by \eqref{eq:strongConvnj} we conclude
\begin{align}
    \sup_{t\in (s,T)} (t-s)\|\nabla u_j(t)\|_{L^2(\R^2)}^2 \le C_u \int_s^T \|\nabla u_{j}(\tau)\|_{L^2(\R^2)}^2\,\dd\tau 
\end{align}
which gives the first estimate in \eqref{eq:decayAoiuj}.

The other two estimates are obtained in the same way. For instance, it is enough to observe that
\begin{align}
(t-s)^2 \|\sqrt{\rho_n(t)}\, \dot{u}_{n,j}(t)\|_{L^2(\R^2)}^2
&\le
C_u (t-s)^2 \bigl(t-\tfrac{t+s}{2}\bigr)^{-1}
\|\nabla u_{n,j}\bigl(\tfrac{t+s}{2}\bigr)\|_{L^2(\R^2)}^2 \\
&\le
C_u (t-s)\|\nabla u_{n,j}\bigl(\tfrac{t+s}{2}\bigr)\|_{L^2(\R^2)}^2 \\
&\le
C_u (t-s)\bigl(\tfrac{t+s}{2}-s\bigr)^{-1}
\int_s^t \|\nabla u_{n,j}(\tau)\|_{L^2(\R^2)}^2 \,\dd \tau \\
&\le
C_u \int_s^t \|\nabla u_{n,j}(\tau)\|_{L^2(\R^2)}^2 \,\dd \tau.
\end{align}
One can then argue exactly as above.
\end{proof}

\begin{proposition}\label{prop:gluinglevel0}
Let $(\rho_0,u_0)$ be as in \eqref{eq:initial-data-L2}, and let $(\rho,u)$ be the solution to \eqref{eq:INS} arising from $(\rho_0,u_0)$ constructed in \cref{prop:existenceL2}. 

For every $j\in\mathbb{Z}$, let $u_{0,j}$ be as in \eqref{eq:decompositionL2}, and denote by $u_j$ the solution to \eqref{eq:LSj} arising from $u_{0,j}$ constructed in \cref{lem:exLS}. Then the partial sum $u_J$ satisfies the following convergences:
\begin{align}\label{eq:gluinglevel0}
\rho u_J \to \rho u 
\quad \text{in } C([0,\infty);L^2(\R^2)), 
\qquad 
\nabla u_J \to \nabla u
\quad \text{in } L^2((0,\infty);L^2(\R^2)).
\end{align}
In particular, $u$ satisfies the energy equality. Moreover, the following convergences hold in $\mathcal D'\bigl((0,\infty)\times\R^2\bigr)$:
\begin{align}\label{eq:distributional-convergencesj}
 u_J \to u,
\qquad 
\rho u_J \to \rho u,
\qquad
\dot u_J \to \dot u,
\qquad
\rho \dot u_J \to \rho \dot u.
\end{align}
\end{proposition}

\begin{remark}
\cref{prop:gluinglevel0} shows that the approximating sequence $(u_J)_J$ enjoys better properties than $(u_n)_n$. Indeed, the strong convergences in \eqref{eq:gluinglevel0} are a novelty with respect to the approximation by lifting. According to \cref{rem:nconvergence}, this additional compactness allows us to pass to the limit in the energy equality for $(\rho,u_J)$.
\end{remark}

\begin{proof}
First observe that, by \eqref{eq:decompositionL2}, the partial sums
\begin{align}
u_{0,J}:=\sum_{|j|\le J}u_{0,j}
\end{align}
converge to $u_0$ in $L^2(\R^2)$.

Since $u_J$ is a finite sum of functions of the form $u_j$, it inherits all the properties established in \cref{lem:exLS}. In particular,
\begin{align}\label{eq:UJcont}
\rho u_J \in C([0,\infty);L^2(\R^2)).
\end{align}
Hence, by \cref{thm:relaEnergyL}, for almost every $t>0$,
\begin{align}\label{eq:relenapp}
\|\rho(t)(u_J(t)-u(t))\|_{L^2(\R^2)}^2
+2
\int_0^t \|\nabla(u_J(s)-u(s))\|_{L^2(\R^2)}^2\,\dd s
\le
\|u_{0,J}-u_0\|_{L^2(\R^2)}^2 .
\end{align}
Since $u_{0,J}\to u_0$ in $L^2(\R^2)$, from \eqref{eq:relenapp} we deduce \eqref{eq:gluinglevel0} and  the energy equality for $(\rho,u)$ itself. We are left to show \eqref{eq:distributional-convergencesj}. We fix $T>0$ and $R>0$.

By repeating the proof of \eqref{eq:step1}\footnote{It suffices to take $\rho_n=\tilde{\rho}_n=\rho$.}, we obtain that for every function \(w\),
\begin{align}\label{eq:w_J}
\|w\|_{L^2(B_R)}
\lesssim
\|\nabla w\|_{L^2(\R^2)}
+
\|\rho w\|_{L^2(\R^2)},
\end{align}
where the implicit constant depends only on $T$, $D$, $R$ and $\|u_0\|_{L^2(\R^2)}$.

Applying this estimate with \(w=u_J(t)-u(t)\), we obtain
\begin{align}
\|u_J(t)-u(t)\|_{L^2(B_R)}
\lesssim
\|\nabla(u_J(t)-u(t))\|_{L^2(\R^2)}
+
\|\rho(t)(u_J(t)-u(t))\|_{L^2(\R^2)}.
\end{align}
Using \eqref{eq:relenapp}, the first term on the right-hand side tends to zero in \(L^2(0,T)\), while the second one tends to zero in \(L^\infty(0,T)\). Integrating in time over \((0,T)\), we conclude that
\begin{align}
u_J \to u
\qquad
\text{in } L^2\bigl((0,T);L^2(B_R)\bigr).
\end{align}
Since \(T>0\) and \(R>0\) are arbitrary, we obtain
\begin{align}
u_J \to u
\qquad
\text{in } L^2_{\loc}\bigl([0,\infty);L_\loc^2(\R^2)\bigr).
\end{align}
Together with the strong convergence of the gradient, this yields
\begin{align}
( u_J \cdot \nabla ) u_J
\to
( u \cdot \nabla ) u
\qquad
\text{in } L^1_{\loc}\bigl([0,\infty);L^2_\loc(\R^2)\bigr),
\end{align}
and therefore
\begin{align}
\dot u_J \to \dot u
\qquad
\text{in } \mathcal D'\bigl((0,\infty)\times\R^2\bigr).
\end{align}
To prove the convergence \(\rho \dot{u}_J \to \rho \dot{u}\), we use \eqref{eq:w_J} with
\(w=\dot u_J\). Multiplying by \(t\), squaring and integrating, we obtain
\begin{align}
\int_0^T t^2 \|\dot{u}_J(t)\|_{L^2(B_R)}^2 \, \dd t
&\lesssim
\int_0^T t^2 \|\nabla \dot{u}_J(t)\|_{L^2(\R^2)}^2 \, \dd t
+
\int_0^T t^2 \|\rho \dot{u}_J(t)\|_{L^2(\R^2)}^2 \, \dd t .
\end{align}
Since \(u_J=\sum_{|j|\le J}u_j\), we estimate
\begin{align}
\int_0^T t^2 \|\dot{u}_J(t)\|_{L^2(B_R)}^2 \, \dd t
\lesssim
\sum_{|j|\le J}
\int_0^T t^2 \|\nabla \dot{u}_j(t)\|_{L^2(\R^2)}^2 \, \dd t
+
\sum_{|j|\le J}
\int_0^T t \|\rho \dot{u}_j(t)\|_{L^2(\R^2)}^2 \, \dd t .
\end{align}
By \cref{lem:exLS} (see \eqref{eq:boundsujL2}), there exists a constant
\(C>0\), independent of \(j\), such that
\begin{align}
\int_0^T t^2 \|\dot{u}_J(t)\|_{L^2(B_R)}^2 \, \dd t
\le
C \sum_{|j|\le J}
\|\sqrt{\rho_0}\,u_{0,j}\|_{L^2(\R^2)}^2 .
\end{align}
Since the right-hand side is uniformly bounded in \(J\), and by the
arbitrariness of \(T\) and \(R\), we deduce that
\begin{align}
\dot u_J \to \dot u
\qquad
\text{in } L^2_{\loc}\bigl((0,\infty)\times\R^2\bigr).
\end{align}
Consequently,
\begin{align}
\rho \dot u_J \to \rho \dot u
\qquad
\text{in } \mathcal D'\bigl((0,\infty)\times\R^2\bigr).
\end{align}
\end{proof}
\subsection{Backward parabolic system}

We conclude this section with an existence result for a backward parabolic system.
Let \((\rho,u)\) be the solution given by \cref{prop:existenceL2}, and for every
\(j\in\Z\) let \(u_j\) be the solution of \eqref{eq:LSj} arising from the initial
data \(u_{0,j}\), constructed in \cref{lem:exLS}.

Given \(T>0\), we consider the dual system of \eqref{eq:LSj} on \([0,T]\times\R^2\):
\begin{align}
\begin{cases}\label{eq:BLSj}
\partial_t (\rho w_j^T) + \dive (\rho u \otimes w_j^T) + \nabla Q_j^T
= - \Delta w_j^T + \rho u_j,\\
\dive w_j^T = 0,\\
w_j^T(T) = 0.
\end{cases}
\end{align}
\begin{lemma}\label{lem:exBLSnj}
In the above setting, there exists a distributional solution \(w_j^T\) to
\eqref{eq:BLSj} and a constant \(C>0\), depending only on
\(\|u_0\|_{L^2(\R^2)}\), such that
\begin{align}\label{eq:estibackA1-limit}
\sup_{t\in(0,T)} (T-t)^{-1} \|\sqrt{\rho}\, w_j^T(t)\|_{L^2(\R^2)}^2
+
\sup_{t\in(0,T)}\|\nabla w_j^T(t)\|_{L^2(\R^2)}^2
\leq
C\int_0^T \|\rho u_j(t)\|_{L^2(\R^2)}^2 \dd t.
\end{align}
Moreover,
\begin{align}\label{eq:dualfor}
\int_0^T \|\rho u_j(t)\|_{L^2(\R^2)}^2 \dd t
=
-
\int_{\R^2} \rho_0(x)\, w_j^T(0,x)\cdot u_{0,j}(x)\,\dd x,
\end{align}
and for every \(t\in[0,T]\),
\begin{align}\label{eq:meanwj}
\int_{\R^2} \rho(t)\, w_j^T(t)\,\dd x
=
-(T-t)\int_{\R^2}\rho_0 u_{0,j}\,\dd x.
\end{align}
\end{lemma}

\begin{proof}
For simplicity, we write \(w\) in place of \(w_j^T\).
As discussed in \cref{prop:existenceL2}, there exists a lifting approximation
scheme \((\rho_n,u_n)\) converging to \((\rho,u)\).
Moreover, for each \(n\in\N\) and \(j\in\Z\), there exists a unique solution
\(u_{n,j}\) of \eqref{eq:LSj}, advected by \((\rho_n,u_n)\), with initial data
\(u_{0,j}\).
By \cref{lem:exLS}, \(u_{n,j}\) converges to \(u_j\), which solves \eqref{eq:LSj}
advected by \((\rho,u)\) with initial data \(u_{0,j}\).

We denote by \(w_n\) the solution of
\begin{align}
\begin{cases}\label{eq:BLSnj}
\partial_t(\rho_n w_n) + \dive(\rho_n u_n \otimes w_n) + \nabla Q^T_{n,j}
= -\Delta w_n + \rho_n u_{n,j},\\
\dive w_n = 0,\\
w_n(T)=0.
\end{cases}
\end{align}

\underline{Step 1.} A priori estimates. Testing the momentum equation in \eqref{eq:BLSnj} with \(w_n\) yields
\begin{align}\label{eq:estibackA0}
- \frac12 \frac{\dd}{\dd t}
\|\sqrt{\rho_n}\,w_n(t)\|_{L^2(\R^2)}^2
+ \|\nabla w_n(t)\|_{L^2(\R^2)}^2
&\le
\left|
\int_{\R^2} \rho_n u_{n,j}(t)\cdot w_n(t)\,\dd x
\right| \\
&\le
\|\sqrt{\rho_n}u_{n,j}(t)\|_{L^2(\R^2)}
\|\sqrt{\rho_n}w_n(t)\|_{L^2(\R^2)} .
\nonumber
\end{align}
Integrating in time over \((t,T)\) and using Cauchy--Schwarz, we obtain
\begin{align}\label{eq:unif1back}
\begin{aligned}
\|\sqrt{\rho_n}(t)\,w_n(t)\|_{L^2(\R^2)}
&\le
\int_t^T
\|\sqrt{\rho_n}u_{n,j}(s)\|_{L^2(\R^2)}\,\dd s \\
&\le
(T-t)^{1/2}
\left(
\int_t^T
\|\sqrt{\rho_n}u_{n,j}(s)\|_{L^2(\R^2)}^2\,\dd s
\right)^{1/2}.
\end{aligned}
\end{align}
We next test the momentum equation in \eqref{eq:BLSnj} with \(\dot{w}_n\) and obtain
\begin{align}\label{eq:dotwnjTeq}
\begin{aligned}
\|\sqrt{\rho_n}\dot{w}_n\|_{L^2(\R^2)}^2
-\frac{\dd}{\dd t}\|\nabla w_n\|_{L^2(\R^2)}^2
&= \langle \sqrt{\rho_n}u_{n,j}, \sqrt{\rho_n}\dot{w}_n \rangle \\
&\quad + \langle \nabla w_n, \nabla \big((u_n\cdot\nabla) w_n \big)\rangle
+ \langle Q^T_{n,j}, \dive \dot{w}_n \rangle \\
&\eqqcolon a^T_{n,j}(t).
\end{aligned}
\end{align}
As in \cref{sec:energy} (see \eqref{eq:conl1}), the right-hand side can be estimated as
\begin{align}\label{eq:estiRHSnjT}
|a^T_{n,j}(t)|
\le \frac12 \|\sqrt{\rho_n}\dot{w}_n(t)\|_{L^2(\R^2)}^2
+ C\|\sqrt{\rho_n}u_{n,j}(t)\|_{L^2(\R^2)}^2
+ C\|\nabla u_n(t)\|_{L^2(\R^2)}^2\,\|\nabla w_n(t)\|_{L^2(\R^2)}^2,
\end{align}
with a constant \(C\) independent of \(T\), \(n\in\N\), and \(j\in\Z\).

Inserting \eqref{eq:estiRHSnjT} into \eqref{eq:dotwnjTeq}, absorbing the term
\(\frac12\|\sqrt{\rho_n}\dot{w}_n\|_{L^2(\R^2)}^2\) on the left-hand side, and integrating from \(t\) to \(T\), we obtain, for every \(t\in(0,T]\),
\begin{align}\label{eq:pregronwall_back}
\int_t^T & \|\sqrt{\rho_n}\dot{w}_n(s)\|_{L^2(\R^2)}^2 \dd s
+ \|\nabla w_n(t)\|_{L^2(\R^2)}^2
\\ & \le C\int_t^T \|\sqrt{\rho_n}u_{n,j}(s)\|_{L^2(\R^2)}^2 \dd s
+ C\int_t^T \|\nabla u_n(s)\|_{L^2(\R^2)}^2
\,\|\nabla w_n(s)\|_{L^2(\R^2)}^2 \dd s .
\end{align}
By Grönwall's inequality, we conclude that for every \(t\in[0,T]\),
\begin{align}\label{eq:gronwall_back_concl}
\|\nabla w_n(t)\|_{L^2(\R^2)}^2
\le
C \int_t^T \|\sqrt{\rho_n}u_{n,j}(s)\|_{L^2(\R^2)}^2 \dd s .
\end{align}

\underline{Step 2.} Compactness and convergence.

Combining \eqref{eq:unif1back} and \eqref{eq:gronwall_back_concl}, we obtain
\begin{align}\label{eq:estibackA1}
\sup_{t\in(0,T)}
\Big(
\|\sqrt{\rho_n}\, w_n(t)\|_{L^2(\R^2)}^2
+
\|\nabla w_n(t)\|_{L^2(\R^2)}^2
\Big)
\le C .
\end{align}
Applying \eqref{eq:step1} with \(w=w_n(t)\), we infer that for every sufficiently large \(R>0\),
\begin{align}\label{eq:step1back-app}
\|w_n(t)\|_{L^2(B_R)}
\lesssim
\|\nabla w_n(t)\|_{L^2(\R^2)}
+
\|\sqrt{\rho_n(t)}\,w_n(t)\|_{L^2(\R^2)},
\end{align}
where the implicit constant may depend on \(T\), \(R\), \(D\), and \(\|u_0\|_{L^2(\R^2)}\).
Combining \eqref{eq:step1back-app} with \eqref{eq:estibackA1}, we obtain
\begin{align}
\|w_n\|_{L^\infty((0,T);H^1(B_R))}\le C .
\end{align}
Therefore, up to a subsequence and by a diagonal argument,
\begin{align}\label{eq:wstar}
w_n \rightharpoonup^* w
\qquad
\text{in } L^\infty\bigl((0,T);H^1_{\loc}(\R^2)\bigr).
\end{align}
Moreover, by \eqref{eq:rho-nconv} and \eqref{eq:sqrt-rhou-nconv} (applied to both \(u_n\) and \(u_{n,j}\)) we have
\begin{align}\label{eq:conrecap}
\begin{aligned}
\rho_n &\to \rho
\qquad \text{in } C\bigl([0,T];L^p_{\loc}(\R^2)\bigr),
\qquad \forall p\in[1,\infty), \\
\sqrt{\rho_n}\, u_n &\to \sqrt{\rho}\, u
\qquad
\text{in } L^1\bigl((0,T);L^2_{\loc}(\R^2)\bigr), \\
\rho_n u_{n,j} &\to \rho u_j
\qquad
\text{in } L^1\bigl((0,T);L^2_{\loc}(\R^2)\bigr).
\end{aligned}
\end{align}
In particular, combining \eqref{eq:wstar} with the strong convergence of \(\rho_n\), we infer that
\begin{align}\label{eq:rhow-conv}
\rho_n w_n \rightharpoonup^* \rho w
\qquad
\text{in } L^\infty\bigl((0,T);L^2_{\loc}(\R^2)\bigr).
\end{align}
Since \(w_n\) solves \eqref{eq:BLSnj}, for every divergence-free
\(\varphi\in C_c^\infty(\R^2)\) and every \(0\le s\le t\le T\) we have
\begin{align}\label{eq:weakBLSnj}
\begin{aligned}
&\int_{\R^2}\rho_n(t)\,w_n(t)\cdot\varphi
-
\int_{\R^2}\rho_n(s)\,w_n(s)\cdot\varphi
\\
&\qquad=
-\int_s^t\!\!\int_{\R^2}(\rho_n u_n\otimes w_n):\nabla\varphi
+
\int_s^t\!\!\int_{\R^2}\nabla w_n:\nabla\varphi
+
\int_s^t\!\!\int_{\R^2}\rho_n u_{n,j}\cdot\varphi .
\end{aligned}
\end{align}
Using \eqref{eq:estibackA1} and \eqref{eq:conrecap}, we deduce that there exists a constant \(C>0\), independent of \(n\), \(s\), and \(t\), such that
\begin{align}\label{eq:equicont}
\left|
\int_{\R^2}\rho_n(t)\,w_n(t)\cdot\varphi
-
\int_{\R^2}\rho_n(s)\,w_n(s)\cdot\varphi
\right|
\le
C |t-s|\, \|\varphi\|_{W^{1,\infty}(\R^2)} .
\end{align}
Hence, for every fixed \(\varphi\in C_c^\infty(\R^2)\), the family of functions
\begin{align}
t\mapsto \int_{\R^2}\rho_n(t)\,w_n(t)\cdot\varphi
\end{align}
is equibounded and equicontinuous on \([0,T]\). Therefore, up to a subsequence, it converges uniformly on \([0,T]\). In view of \eqref{eq:rhow-conv}, the limit is necessarily
\begin{align}
t\mapsto \int_{\R^2}\rho(t)\,w(t)\cdot\varphi .
\end{align}
Since \(\rho_n  w_n\) is uniformly bounded in \(L^\infty\bigl((0,T);L^2(\R^2)\bigr)\), by density we conclude that for every \(t\in[0,T]\),
\begin{align}\label{eq:weakwL2}
\rho_n\,w_n(t) \rightharpoonup \rho\,w(t)
\qquad
\text{in } L_\sigma^2(\R^2).
\end{align}

\underline{Step 3.} Conclusion.

Since \(u_{n,j}\) solves
\begin{align}
\begin{cases}\label{eq:LSnj2}
\partial_t(\rho_n u_{n,j}) + \dive(\rho_n u_n \otimes u_{n,j}) + \nabla P_{n,j}
= \Delta u_{n,j},\\
\dive u_{n,j} = 0,\\
u_{n,j}(0)=u_{0,j},
\end{cases}
\end{align}
the system \eqref{eq:BLSnj} is dual to \eqref{eq:LSnj2}, and therefore
\begin{align}
\int_0^T \int_{\R^2} \rho_n u_{n,j} \cdot u_{n,j} \, \dd x \, \dd t
=
-
\int_{\R^2} \rho_n(0,x)\, w_n(0,x) \cdot u_{0,j}(x) \, \dd x .
\end{align}
Thanks to \eqref{eq:weakwL2} and the convergence \eqref{eq:unifconvenergy}, we may pass to the limit as \(n\to\infty\) and obtain \eqref{eq:dualfor}.

We may also pass to the limit in \eqref{eq:weakBLSnj}, using again \eqref{eq:weakwL2} together with the convergences in \eqref{eq:estibackA1} and \eqref{eq:conrecap}. Then, taking \(\varphi=a\in\R^2\) (this can be justified as in \cref{lem:MomentCons}), we obtain
\begin{align}
\int_{\R^2}\rho(t)\,w(t)\cdot a \,\dd x
-
\int_{\R^2}\rho(s)\,w(s)\cdot a \,\dd x
=
\int_s^t\!\!\int_{\R^2}\rho u_j\cdot a \,\dd x\,\dd \tau.
\end{align}
By \cref{lem:MomentCons},
\begin{align}
\int_s^t\!\!\int_{\R^2}\rho u_j\cdot a \,\dd x\,\dd \tau
=
a\cdot \int_s^t\!\!\int_{\R^2}\rho_0 u_{0,j}\,\dd x\,\dd \tau
=
(t-s)\, a\cdot \int_{\R^2}\rho_0 u_{0,j}\,\dd x .
\end{align}
Hence,
\begin{align}
\int_{\R^2}\rho(t)\,w(t)\,\dd x
-
\int_{\R^2}\rho(s)\,w(s)\,\dd x
=
(t-s)\int_{\R^2}\rho_0 u_{0,j}\,\dd x .
\end{align}
Setting \(t=T\) and using \(w(T)=0\), we obtain \eqref{eq:meanwj}.
\end{proof}

\section{Existence and Uniqueness for Besov data}
In this section, we prove a uniqueness result under the following assumptions on the initial data:
\begin{align}\label{ass:databesov2}
    \rho_0 = \mathbf{1}_D, \qquad D \text{ bounded Lipschitz domain,} \qquad u_0 \in \dot{B}^0_{2,1}(\R^2).
\end{align}
We also derive key $L^1$-in-time estimates.

Since $\dot{B}^0_{2,1}(\R^2) \subset L^2(\R^2)$, we can use the decomposition \eqref{eq:decompositionL2}. Namely, for every $\eta \in (0,1)$ there exists a sequence $(u_{0,j})_{j \in \Z}$ in
$\dot{H}^{\eta}(\R^2) \cap \dot{H}^{-\eta}(\R^2)$ such that
\begin{align}
u_0 = \sum_{j \in \Z} u_{0,j}, \qquad
c_j := 
2^{-j/2}\norm{u_{0,j}}_{\dot{H}^{\eta}(\R^2)}
+
2^{j/2}\norm{u_{0,j}}_{\dot{H}^{-\eta}(\R^2)} \in \ell^2(\Z).
\end{align}
For the rest of the section, let $(\rho,u)$ be the solution of \eqref{eq:INS} with initial data $(\rho_0,u_0)$ constructed in \cref{prop:existenceL2}. For every $j \in \Z$, let $u_j$ be the solution of \eqref{eq:LSj} constructed in \cref{lem:exLS} with initial datum $u_{0,j}$.

In Sections 5.1 and 5.2 we exploit the additional information $u_{0,j} \in \dot{H}^{-\eta}(\R^2)$ to obtain improved decay estimates for $u_j$, in contrast with \cref{lem:exLS}, where only positive regularity is used. We emphasize that these arguments only rely on the assumption $u_0 \in L^2(\R^2)$.

Finally, in Section 5.3 we show that the additional structure provided by $u_0 \in \dot{B}^0_{2,1}$ yields further regularity properties for $(\rho,u)$, which are sufficient to establish uniqueness.

We restrict to $\eta \in (0,\tfrac12)$. This is natural since, for Lipschitz domains, the characteristic function $\rho_0=\mathbf{1}_D$ belongs to the multiplier space $\mathcal{M}(\dot{H}^{s}(\R^2))$ for all $s \in (-\tfrac12,\tfrac12)$ (\cite[Proposition 5.3]{zbMATH01908115}). In particular, for every $\eta \in (0,\tfrac12)$ there exists a constant $C>0$ such that, for every $\phi \in \dot{H}^{\pm \eta}(\R^2)$,
\begin{align}
    \norm{\rho_0 \phi}_{\dot{H}^{-\eta}(\R^2)} \leq C \norm{\phi}_{\dot{H}^{-\eta}(\R^2)}, 
    \qquad
    \norm{\rho_0 \phi}_{\dot{H}^{\eta}(\R^2)} \leq C \norm{\phi}_{\dot{H}^{\eta}(\R^2)}.
\end{align}

\subsection{\texorpdfstring{Galilean transform and estimates for $\dot H^{-\eta}$ data}{Galilean transform and estimates for H-dot minus eta data}}

Here and in the following, let
\begin{align}
    M_j = \frac{1}{|D|}\int_{\R^2} \rho_0 u_{0,j} \,\dd x,
    \qquad j \in \Z,
\end{align}
and define, for every $j \in \Z$, every $t \geq 0$, and every $x \in \R^2$,
\begin{align}
    \rho_{M_j}(t,x) &\coloneqq \rho(t,x + M_j t),\\
    u_{M_j}(t,x) &\coloneqq u(t,x + M_j t) - M_j,\\
    u_{j,M_j}(t,x) &\coloneqq u_j(t,x + M_j t) - M_j.
\end{align}
This transform, called the Galilean transform, satisfies the following two basic but crucial properties.

\begin{lemma}
Let $(\rho_{M_j},u_{M_j})$ and $u_{j,M_j}$ be as above. Then, for every $t \geq 0$,
\begin{align}\label{eq:0mom}
    \int_{\R^2} \rho_{M_j}(t,x)\, u_{j,M_j}(t,x)\,\dd x = 0.
\end{align}
Moreover, for every $0 \leq s < t$,
\begin{align}\label{eq:energyequjMj2}
    \frac{1}{2} \| \sqrt{\rho_{M_j}(t)}\, u_{j,M_j}(t)\|^2_{L^2(\R^2)}
    + \int_s^t \| \nabla u_j(\tau)\|^2_{L^2(\R^2)} \,\dd \tau
    = \frac{1}{2} \| \sqrt{\rho_{M_j}(s)}\, u_{j,M_j}(s)\|^2_{L^2(\R^2)}.
\end{align}
\end{lemma}

\begin{proof}
Since $u_j$ is an immediately strong solution by \cref{lem:exLS}, its momentum is conserved in time by \cref{lem:MomentCons}. Therefore,
\begin{align}\label{eq:nullmom}
    \int_{\R^2} \rho(t,x)\,u_j(t,x)\,\dd x
    = \int_{\R^2} \rho_0(x)\,u_{0,j}(x)\,\dd x
    = |D|\,M_j.
\end{align}
Since
\begin{align}\label{eq:consmass}
    \int_{\R^2} \rho(t,x)\,\dd x = |D|,
\end{align}
identity \eqref{eq:0mom} follows directly from the definition of the transform.

This proves \eqref{eq:0mom}. Recall now that $u_j$ satisfies the energy equality, namely, for every $0 \le s < t$,
\begin{align}\label{eq:energyequj-st}
    \frac{1}{2} \int_{\R^2} \rho(t) |u_j(t)|^2 \,\dd x
    + \int_s^t \| \nabla u_j(\tau)\|^2_{L^2(\R^2)} \,\dd \tau
    = \frac{1}{2} \int_{\R^2} \rho(s) |u_j(s)|^2 \,\dd x.
\end{align}
If we show that, for every $t \geq 0$,
\begin{align}\label{eq:compGal}
    \| \sqrt{\rho_{M_j}(t)}\, u_{j,M_j}(t)\|^2_{L^2(\R^2)}
    = \int_{\R^2} \rho(t,x)\, |u_j(t,x)|^2 \,\dd x - |D|\,|M_j|^2,
\end{align}
then inserting \eqref{eq:compGal}, evaluated at times $t$ and $s$, into \eqref{eq:energyequj-st} yields \eqref{eq:energyequjMj2}.

To prove \eqref{eq:compGal}, we expand the square and use \eqref{eq:nullmom} together with \eqref{eq:consmass}:
\begin{align}
\begin{aligned}
    \int_{\R^2} \rho_{M_j}(t,x) |u_{j,M_j}(t,x)|^2 \,\dd x
    &= \int_{\R^2} \rho(t,x+M_j t)\, |u_j(t,x+M_j t)-M_j|^2 \,\dd x \\
    &= \int_{\R^2} \rho(t,x+M_j t)\, |u_j(t,x+M_j t)|^2 \,\dd x \\
    &\quad - 2 M_j \cdot \int_{\R^2} \rho(t,x+M_j t)\, u_j(t,x+M_j t)\,\dd x \\
    &\quad + |M_j|^2 \int_{\R^2} \rho(t,x+M_j t)\,\dd x \\
    &= \int_{\R^2} \rho(t,x)\, |u_j(t,x)|^2 \,\dd x - |D|\,|M_j|^2.
\end{aligned}
\end{align}
This concludes the proof.
\end{proof}

Since $u_{j,M_j}$ has zero $\rho_{M_j}$-momentum, we can eliminate the constant component and recover a Poincaré-type control. This yields improved decay estimates, as shown in the following lemma.
 
\begin{lemma}\label{lem:exLSnegative}
There exists a constant $C>0$, depending only on $D$ and $
\|u_0\|_{L^2(\R^2)}$, such that for every $t > 0$,
\begin{align}\label{eq:neg}
    \int_0^t \|\sqrt{\rho_{M_j}(s)}\, u_{j,M_j}(s)\|_{L^2(\R^2)}^2 \dd s
    \le C \|u_{0,j}\|_{\dot{H}^{-\eta}(\R^2)}^2\, t^{1-\eta},
\end{align}
and
\begin{align}\label{eq:neg2}
    \|\sqrt{\rho_{M_j}(t)}\, u_{j,M_j}(t)\|_{L^2(\R^2)}^2
    \le C \|u_{0,j}\|_{\dot{H}^{-\eta}(\R^2)}^2\, t^{-\eta}.
\end{align}
\end{lemma}

\begin{proof}
Let $w_j^t$ be the solution to \eqref{eq:BLSj} on $[0,t]\times\R^2$. By \cref{lem:exBLSnj} and the definition of $M_j$, we have
\begin{align}\label{eq:meanwj0}
    \int_{\R^2} \rho_0(x)\, w_j^t(0,x)\,\dd x
    = - t \int_{\R^2} \rho_0(x)\, u_{0,j}(x)\,\dd x
    = - t |D|\, M_j .
\end{align}

We introduce
\begin{align}\label{eq:defbarw}
    \bar w_0(x)
    :=
    w_j^t(0,x)+ t M_j,
\end{align}
which has zero mean on $D$ thanks to \eqref{eq:meanwj0} and the fact that $\rho_0=\mathbf{1}_D$. By duality,
\begin{align}\label{eq:dualsplitwj}
    \int_0^t \|\sqrt{\rho}(s)\,u_j(s)\|_{L^2(\R^2)}^2 \,\dd s
    = - \int_{\R^2} \rho_0(x)\, u_{0,j}(x)\cdot w_j^t(0,x)\,\dd x.
\end{align}
Adding $-t |D| |M_j|^2$ to both sides and using \eqref{eq:compGal}, we can rewrite \eqref{eq:dualsplitwj} as
\begin{align}\label{eq:dualsplitwj2}
    \int_0^t \|\sqrt{\rho_{M_j}}(s)\,u_{j,M_j}(s)\|_{L^2(\R^2)}^2 \,\dd s
    = - \int_{\R^2} \rho_0(x)\, u_{0,j}(x)\cdot \bar w_0(x)\,\dd x
    \le \|u_{0,j}\|_{\dot{H}^{-\eta}(\R^2)}\|\rho_0 \bar{w}_0\|_{\dot{H}^\eta(\R^2)}.
\end{align}
To prove \eqref{eq:neg}, it suffices to show
\begin{align}\label{eq:neg1}
    \|\rho_0 \bar{w}_0\|^2_{\dot{H}^\eta(\R^2)}
    \leq C t^{1-\eta}
    \int_0^t \|\sqrt{\rho_{M_j}}(s)\,u_{j,M_j}(s)\|_{L^2(\R^2)}^2 \,\dd s.
\end{align}
Since $D$ is a Lipschitz domain and $\bar{w}_0$ has zero average on $D$, we can apply the Poincaré inequality on $D$ and  arguing as in \eqref{eq:ExD} obtain
\begin{align}\label{eq:hestbarw}
       \|\rho_0 \bar{w}_0\|^2_{\dot{H}^\eta(\R^2)}
       & \lesssim   \|\bar{w}_0\|_{L^2(D)}^{2-2\eta}\,
                     \|\bar{w}_0\|_{H^1(D)}^{2\eta}   \lesssim   \|\bar{w}_0\|_{L^2(D)}^{2-2\eta}
                     \left(
                     \|\bar{w}_0\|_{L^2(D)}^{2}
                     +
                     \|\nabla \bar{w}_0\|_{L^2(D)}^{2}
                     \right)^\eta \\
       & \lesssim   \|\bar{w}_0\|_{L^2(D)}^{2-2\eta}
                     \|\nabla \bar{w}_0\|_{L^2(\R^2)}^{2\eta}.
\end{align}
The implicit constants depend only on $D$.

By \cref{lem:exBLSnj}, there exists $C>0$ such that
\begin{align}\label{eq:estibackA1recall}
        t^{-1} \|\sqrt{\rho_0} w^t_j(0)\|^2_{L^2(\R^2)} + \|\nabla w^t_j(0)\|^2_{L^2(\R^2)}
        \leq \int_0^t \|\sqrt{\rho}(s)\, u_{j}(s)\|_{L^2(\R^2)}^2 \dd s.
\end{align}
Moreover, a direct computation shows that
\begin{align}\label{eq:L2barw}
\|\sqrt{\rho_0}\,\bar w_0\|_{L^2(\R^2)}^2 =
\|\sqrt{\rho_0}\, w_j^t(0)\|_{L^2(\R^2)}^2
- t^2 |D|\, |M_j|^2.
\end{align}
Using \eqref{eq:compGal} and \eqref{eq:L2barw}, we deduce
\begin{align}\label{eq:estibackA1recall2}
        t^{-1} \|\sqrt{\rho_0} \bar{w}_0\|^2_{L^2(\R^2)} + \|\nabla \bar{w}_0\|^2_{L^2(\R^2)}
        \leq \int_0^t \|\sqrt{\rho_{M_j}(s)}\, u_{j,M_j}(s)\|_{L^2(\R^2)}^2 \dd s.
\end{align}
Inserting \eqref{eq:estibackA1recall2} into \eqref{eq:hestbarw} yields \eqref{eq:neg1}.

To prove \eqref{eq:neg2}, we observe that by \eqref{eq:energyequjMj2} the quantity
\begin{align}
    \|\sqrt{\rho_{M_j}(t)}\,u_{j,M_j}(t)\|_{L^2(\R^2)}^2
\end{align}
is non-increasing in time. Therefore,
\begin{align}
    \,t\,\|\sqrt{\rho_{M_j}(t)}\,u_{j,M_j}(t)\|_{L^2(\R^2)}^2
    \le
    \int_0^t \|\sqrt{\rho_{M_j}(s)}\,u_{j,M_j}(s)\|_{L^2(\R^2)}^2 \dd s,
\end{align}
and inserting this into \eqref{eq:neg} gives \eqref{eq:neg2}.
\end{proof}

\subsection{\texorpdfstring{$L^1$-in-time estimates}{L1-in-time estimates}}

Combining the shifted estimates \eqref{eq:decayAoiuj} with the decay of the Galilean transform obtained in \cref{lem:exLSnegative}, we obtain decay estimates for the solution $u_j$.

\begin{lemma}\label{lem:decay-shifted-uj}
There exists a constant $C>0$, depending only on $D$ and $\|u_0\|_{L^2(\R^2)}$, such that for every $s>0$ and for almost every $t>0$,
\begin{align}\label{eq:decayAoiujMj}
\begin{aligned}
    \|\nabla u_j(t)\|_{L^2(\R^2)}
    &\le C\, t^{-1\slash 2-\eta/2}\|u_{0,j}\|_{\dot H^{-\eta}(\R^2)},\\
  \|\sqrt{\rho(t)}\, \dot u_j(t)\|_{L^2(\R^2)}
    &\le C\, t^{-1-\eta/2}\|u_{0,j}\|_{\dot H^{-\eta}(\R^2)},\\
    \|\nabla \dot u_j(t)\|_{L^2(\R^2)}
    &\le C\, t^{-3\slash 2-\eta/2}\|u_{0,j}\|_{\dot H^{-\eta}(\R^2)}.
\end{aligned}
\end{align}
\end{lemma}

\begin{proof}
In the proof, every constant depending only on $D$ and $\|u_0\|_{L^2(\R^2)}$ is denoted by $C$.

We first prove the estimate for $\nabla u_j$. By \eqref{eq:decayAoiuj}, for every $0<s<T<\infty$,
\begin{align}\label{eq:laststop}
   \sup_{t \in (s,T)} (t-s)\|\nabla u_j(t)\|_{L^2(\R^2)}^2
    \le C \int_s^T \|\nabla u_j(\tau)\|_{L^2(\R^2)}^2 \,\dd \tau .
\end{align}
Then, by the energy equality \eqref{eq:energyequjMj2} and the decay estimate \eqref{eq:neg2}, we get
\begin{align}\label{eq:laststop2}
    \int_s^T \|\nabla u_j(\tau)\|_{L^2(\R^2)}^2 \,\dd \tau
    \leq \frac12 \|\sqrt{\rho_{M_j}(s)}\,u_{j,M_j}(s)\|_{L^2(\R^2)}^2
    \le C s^{-\eta}\|u_{0,j}\|_{\dot H^{-\eta}(\R^2)}^2.
\end{align}
Combining \eqref{eq:laststop} and \eqref{eq:laststop2}, we infer that for every $0<s<T<\infty$,
\begin{align}
    \sup_{t \in (s,T)} (t-s)\|\nabla u_j(t)\|_{L^2(\R^2)}^2
    \le C s^{-\eta}\|u_{0,j}\|_{\dot H^{-\eta}(\R^2)}^2.
\end{align}
Since $T>0$ is arbitrary, it follows that for every $s>0$ and for almost every $t>s$,
\begin{align}
  (t-s)\|\nabla u_j(t)\|_{L^2(\R^2)}^2
  \le C s^{-\eta}\|u_{0,j}\|_{\dot H^{-\eta}(\R^2)}^2.
\end{align}
By setting $s = t\slash 2$ we conclude the proof and taking square roots, we obtain the first estimate in \eqref{eq:decayAoiujMj}. The other two estimates are proved in the same way, starting from the corresponding bounds in \eqref{eq:decayAoiuj}. 
\end{proof}

Now we study a time-integrated version of \cref{lem:decay-shifted-uj}.  This combines the decay estimates obtained in \eqref{eq:decayAoiujMj} with the bounds provided by \cref{lem:exLS}, in particular \eqref{eq:boundsujHeta}. 

\begin{lemma}\label{lem:ancientlemma}
There exists a constant $C>0$, depending only on $D$ and $\|u_0\|_{L^2(\R^2)}$, such that the following estimates hold.

\begin{itemize}

\item[(i)] Let $\beta_1>0$ be such that $\beta_1 = 2(\alpha_1 + 1)$. Then
\begin{align}
\int_0^\infty t^{\alpha_1} 
\|\nabla u_j(t)\|_{L^2(\R^2)}^{\beta_1}\, dt
\le
C\Big(
2^{-\frac{j}{2}\beta_1}\|u_{0,j}\|_{\dot H^\eta(\R^2)}^{\beta_1}
+
2^{\frac{j}{2}\beta_1}\|u_{0,j}\|_{\dot H^{-\eta}(\R^2)}^{\beta_1}
\Big).
\end{align}

\item[(ii)] 
Let $\beta_2>0$ be such that $\beta_2 = \alpha_2 + 1$. Then
\begin{align}
\int_0^\infty t^{\alpha_2} 
\|\sqrt{\rho(t)}\,\dot u_j(t)\|_{L^2(\R^2)}^{\beta_2}\, dt
\le
C\Big(
2^{-\frac{j}{2}\beta_2}\|u_{0,j}\|_{\dot H^\eta(\R^2)}^{\beta_2}
+
2^{\frac{j}{2}\beta_2}\|u_{0,j}\|_{\dot H^{-\eta}(\R^2)}^{\beta_2}
\Big).
\end{align}

\item[(iii)] 
Let $\beta_3>0$ be such that $3\beta_3 = 2(\alpha_3 + 1)$. Then
\begin{align}
\int_0^\infty t^{\alpha_3} 
\|\nabla \dot u_j(t)\|_{L^2(\R^2)}^{\beta_3}\, dt
\le
C\Big(
2^{-\frac{j}{2}\beta_3}\|u_{0,j}\|_{\dot H^\eta(\R^2)}^{\beta_3}
+
2^{\frac{j}{2}\beta_3}\|u_{0,j}\|_{\dot H^{-\eta}(\R^2)}^{\beta_3}
\Big).
\end{align}

\end{itemize}
\end{lemma}
\begin{proof}
Throughout the proof, constants depending only on $D$ and $\|u_0\|_{L^2(\R^2)}$ are omitted and denoted by $\lesssim$.

We start with (i). By \cref{lem:decay-shifted-uj} with $s = t/2$, we obtain, for almost every $t>0$,
\begin{align}
\|\nabla u_j(t)\|_{L^2(\R^2)}
\lesssim
\|u_{0,j}\|_{\dot H^{-\eta}(\R^2)}\, t^{-1/2-\eta/2}.
\end{align}
On the other hand, by \cref{lem:exLS} (see \eqref{eq:boundsujHeta}), we have, for almost every $t>0$,
\begin{align}\label{eq:mindec}
\|\nabla u_j(t)\|_{L^2(\R^2)}
\lesssim
\min\Big\{
\|u_{0,j}\|_{\dot H^\eta(\R^2)}\, t^{-1/2+\eta/2},
\,
\|u_{0,j}\|_{\dot H^{-\eta}(\R^2)}\, t^{-1/2-\eta/2}
\Big\}.
\end{align}
Let $A>0$. Splitting the integral at $A$ and using \eqref{eq:mindec}, we obtain
\begin{align}
\int_0^\infty t^{\alpha_1}\|\nabla u_j(t)\|_{L^2(\R^2)}^{\beta_1}\,dt
&\lesssim
\|u_{0,j}\|_{\dot H^\eta}^{\beta_1}
\int_0^A
t^{\alpha_1+(-\frac12+\frac\eta2)\beta_1}\,dt
+
\|u_{0,j}\|_{\dot H^{-\eta}}^{\beta_1}
\int_A^\infty
t^{\alpha_1+(-\frac12-\frac\eta2)\beta_1}\,dt .
\end{align}

Using $\beta_1=2(\alpha_1+1)$, we compute
\begin{align}
\int_0^\infty t^{\alpha_1}\|\nabla u_j(t)\|_{L^2}^{\beta_1}\,dt
\lesssim
\|u_{0,j}\|_{\dot H^\eta}^{\beta_1} A^{\frac{\eta}{2}\beta_1}
+
\|u_{0,j}\|_{\dot H^{-\eta}}^{\beta_1} A^{-\frac{\eta}{2}\beta_1}.
\end{align}
Choosing $A=2^{-j/\eta}$ yields (i).

The proof of (ii) is identical. Indeed, using again \cref{lem:exLS} and \cref{lem:decay-shifted-uj}, we obtain, for almost every $t>0$,
\begin{align}
\|\sqrt{\rho(t)}\,\dot u_j(t)\|_{L^2(\R^2)}
\lesssim
\min\Big\{
\|u_{0,j}\|_{\dot H^\eta}\, t^{-1+\eta/2},
\,
\|u_{0,j}\|_{\dot H^{-\eta}}\, t^{-1-\eta/2}
\Big\}.
\end{align}
Proceeding as above yields (ii). The proof of (iii) is analogous.
\end{proof}
\subsection{Lipschitz regularity of Besov solutions.}

We now exploit the assumption $u_0 \in \dot{B}^0_{2,1}(\R^2)$. 
Since this regularity is more evenly distributed across frequencies, we may assume that the initial decomposition \eqref{eq:initial-data-L2} satisfies the additional summability property
\begin{align}\label{eq:decBes}
  c_j 
  := 
  2^{-j/2}\|u_{0,j}\|_{\dot{H}^{\eta}(\R^2)}
  +
  2^{j/2}\|\rho_0 u_{0,j}\|_{\dot{H}^{-\eta}(\R^2)}
  \in \ell^1(\Z).
\end{align}
The fact that $c_j \in \ell^1(\Z)$ ensures that the solution constructed in \cref{prop:gluinglevel0} is Lipschitz. 
We also note that $u_0$ belongs to the same weighted spaces considered in \cite{Danchin2025}.

\begin{proposition}\label{prop:existencebesov}
Let $(\rho,u)$ be the solution constructed in \cref{prop:existenceL2}, and assume in addition that $u_0 \in \dot{B}^0_{2,1}(\R^2)$. Then there exists a constant $C>0$, depending only on $D$ and $\|u_0\|_{L^2(\R^2)}$, such that
\begin{align}\label{eq:L1L2estimates}
    \int_0^\infty \|\sqrt{\rho} \dot{u}(t),\nabla^2 u(t),\nabla P(t)\|_{L^2(\R^2)} \,\dd t
    +
    \int_0^\infty \sqrt{t}\,\|\nabla \dot{u}(t)\|_{L^2(\R^2)} \,\dd t
    \leq C \|u_0\|_{\dot B^0_{2,1}(\R^2)}.
\end{align}
Moreover,
\begin{align}\label{eq:addiestimatesu}
    \int_0^\infty  t^{-1/2}  \|\nabla u(t)\|_{L^2(\R^2)} + \|\nabla u(t)\|_{L^\infty(\R^2)} \,\dd t
    +
    \int_0^\infty t^{3/4} \|\nabla \dot{u}(t)\|_{L^4(\R^2)} \,\dd t
    \leq C \|u_0\|_{\dot B^0_{2,1}(\R^2)}.
\end{align}
Furthermore,
\begin{align}\label{eq:addiestimatesu2}
   \lim_{t \to 0} \sqrt{t}\,\|\nabla u(t)\|_{L^2(\R^2)} = 0.
\end{align}
\end{proposition}
\begin{proof}
Similarly as before, throughout the proof we omit constants depending only on $D$ and $\|u_0\|_{L^2(\R^2)}$, and denote them by $\lesssim$.

For every $j \in \Z$, let $u_j$ be the solution to the linearized system \eqref{eq:LSj}
with initial data $u_{0,j}$ constructed in \cref{lem:exLS}. For every $J \in \N$, define
\begin{align}
    u_J := \sum_{|j| \leq J} u_j.
\end{align}
We start with \eqref{eq:L1L2estimates}. By \cref{lem:ancientlemma} (ii) with $\alpha_2=0$ and $\beta_2=1$, we obtain
\begin{align}\label{eq:strategy}
    \int_0^\infty \|\rho \dot{u}_{j}(t)\|_{L^2(\R^2)} \,\dd t
    \lesssim c_j.
\end{align}
Summing over $|j|\le J$ and using that $c_j \in \ell^1(\Z)$, we deduce the uniform bound
\begin{align}\label{eq:est1J}
    \int_0^\infty \|\rho \dot{u}_J(t)\|_{L^2(\R^2)} \,\dd t
    \le \sum_{|j|\le J} \int_0^\infty \|\rho \dot{u}_j(t)\|_{L^2(\R^2)} \,\dd t
    \lesssim
    \|u_0\|_{\dot{B}^0_{2,1}(\R^2)}.
\end{align}
By \cref{prop:gluinglevel0} (see \eqref{eq:distributional-convergencesj}), we have
\begin{align}
\rho \dot{u}_J \to \rho \dot{u}
\qquad \text{in } \mathcal D'\bigl((0,\infty)\times\R^2\bigr).
\end{align}
Therefore, by lower semicontinuity (see \cref{rem:Immstrong}), we conclude
\begin{align}
\int_0^\infty \|\rho \dot{u}(t)\|_{L^2(\R^2)} \,\dd t
\le
\liminf_{J\to\infty}
\int_0^\infty \|\rho \dot{u}_J(t)\|_{L^2(\R^2)} \,\dd t
\lesssim \|u_0\|_{\dot{B}^0_{2,1}(\R^2)}.
\end{align}
This proves the first estimate in \eqref{eq:L1L2estimates}. By linear Stokes, we immediately obtain that
\begin{align}
    \int_0^\infty \|\nabla^2 u(t),\nabla P(t)\|_{L^2(\R^2)} \,\dd t \lesssim \int_0^\infty \|\sqrt{\rho} \dot{u}(t)\|_{L^2(\R^2)} \,\dd t \leq C \|u_0\|_{\dot B^0_{2,1}(\R^2)}.
\end{align}
The second estimate in \eqref{eq:L1L2estimates} follows in the same way by using \cref{lem:ancientlemma} (iii) with $\alpha_3 = \tfrac{1}{2}$ and $\beta_3 = 1$.

In order to prove the first estimate in \eqref{eq:addiestimatesu}, we use \cref{lem:ancientlemma} (i) with $\alpha_1=-1/2$ and $\beta_1=1$ and we obtain that, for every $j \in \Z$,
\begin{align}
    \int_0^\infty t^{-1/2} \norm{\nabla u_j(t)}_{L^2(\R^2)}  \,\dd t \lesssim c_j.
\end{align}
By summing over all $j \in \Z$ and the same arguments as above, we conclude the first bound in \eqref{eq:addiestimatesu}. To show the second bound in \eqref{eq:addiestimatesu}, we recall that by \eqref{eq:L2Lip} (and Stokes estimates), for every $j \in \Z$ and every $t>0$,
\begin{align}
    \|\nabla u_j(t)\|_{L^\infty(\R^2)}
    \lesssim
    \|\rho(t)\dot{u}_j(t)\|_{L^2(\R^2)}
    +
    \|\nabla u_j(t)\|_{L^2(\R^2)}^{1/2}
    \|\nabla \dot{u}_j(t)\|_{L^2(\R^2)}^{1/2}.
\end{align}

We have already estimated the first term. For the second one, we use Hölder's inequality with exponents $4/3$ and $4$, together with \cref{lem:ancientlemma} (i) with $\alpha_1=-2/3$ and $\beta_1=2/3$, and \cref{lem:ancientlemma} (iii) with $\alpha_3=2$ and $\beta_3=2$, to obtain
\begin{align}
    \int_0^\infty & 
    \|\nabla u_j(s)\|_{L^2(\R^2)}^{1/2} 
    \|\nabla \dot{u}_j(s)\|_{L^2(\R^2)}^{1/2}\, ds
    \\ &\le
    \left(\int_0^\infty s^{-2/3}\|\nabla u_j(s)\|_{L^2(\R^2)}^{2/3}\, ds\right)^{3/4} 
    \left(\int_0^\infty s^{2}\|\nabla \dot{u}_j(s)\|_{L^2(\R^2)}^{2}\, ds\right)^{1/4} \\
    &\lesssim
    \left(
    2^{-j/3}\|u_{0,j}\|_{\dot H^\eta(\R^2)}^{2/3}
    +
    2^{j/3}\|u_{0,j}\|_{\dot H^{-\eta}(\R^2)}^{2/3}
    \right)^{3/4} 
    \left(
    2^{-j}\|u_{0,j}\|_{\dot H^\eta(\R^2)}^{2}
    +
    2^{j}\|u_{0,j}\|_{\dot H^{-\eta}(\R^2)}^{2}
    \right)^{1/4} \lesssim c_j.
\end{align}
Since $(c_j)_{j\in\Z} \in \ell^1(\Z)$, we conclude as before.

To complete the proof of \eqref{eq:addiestimatesu}, it remains to show that
\begin{align}\label{eq:L1L4nablaudot}
    \int_0^\infty \|t^{3/4} \nabla \dot{u}_j(t)\|_{L^4(\R^2)} \,\dd t
    \lesssim c_j.
\end{align}
Fix $j \in \Z$. By Ladyzhenskaya's inequality and Hölder's inequality, we obtain
\begin{align}
    \int_0^\infty \|t^{3/4} \nabla \dot{u}_j(t)\|_{L^4(\R^2)} \,\dd t
    &\lesssim \int_0^\infty
    t^{3/4}
    \|\nabla \dot{u}_j(t)\|_{L^2(\R^2)}^{1/2}
    \|\nabla^2 \dot{u}_j(t)\|_{L^2(\R^2)}^{1/2}
    \,\dd t \\
    &\le
    \left(
    \int_0^\infty
    \|\nabla \dot{u}_j(t)\|_{L^2(\R^2)}^{2/3}
    \,\dd t
    \right)^{3/4}
    \left(
    \int_0^\infty
    t^{3}\|\nabla^2 \dot{u}_j(t)\|_{L^2(\R^2)}^2
    \,\dd t
    \right)^{1/4}.
\end{align}
By \cref{lem:exLS}, we have
\begin{align}
    \left(
    \int_0^\infty
    t^{3}\|\nabla^2 \dot{u}_j(t)\|_{L^2(\R^2)}^2
    \,\dd t
    \right)^{1/4}
    \lesssim 
    \|u_{0,j}\|_{L^2(\R^2)}^{1/2}
    \lesssim
    \left(
    2^{-j/4}\|u_{0,j}\|_{\dot H^\eta(\R^2)}^{1/2}
    +
    2^{j/4}\|u_{0,j}\|_{\dot H^{-\eta}(\R^2)}^{1/2}
    \right),
\end{align}
while by \cref{lem:ancientlemma} (iii), applied with $\alpha_3=0$ and $\beta_3=2/3$, we obtain
\begin{align}
    \left(
    \int_0^\infty
    \|\nabla \dot{u}_j(t)\|_{L^2(\R^2)}^{2/3}
    \,\dd t
    \right)^{3/4}
    \lesssim
    \left(
    2^{-j/4}\|u_{0,j}\|_{\dot H^\eta(\R^2)}^{1/2}
    +
    2^{j/4}\|u_{0,j}\|_{\dot H^{-\eta}(\R^2)}^{1/2}
    \right).
\end{align}
    Multiplying the above estimates yields \eqref{eq:L1L4nablaudot}. By \eqref{eq:mindec}, for $j\in\Z$, and almost every $t>0$,
    \begin{align}\label{eq:mindec_used}
        \sqrt{t}\,\|\nabla u_{j}(t)\|_{L^2(\R^2)}
        \lesssim \min \left\{
        t^{\eta/2}\|u_{0,j}\|_{\dot{H}^\eta(\R^2)},
        \;
        t^{-\eta/2}\|u_{0,j}\|_{\dot{H}^{-\eta}(\R^2)}
        \right\}.
    \end{align}
        Then, for every $j \in \Z$, we have
        \begin{align}
        \sqrt{t}\,\|\nabla u_{j}(t)\|_{L^2(\R^2)}
        &\lesssim t^{\eta/2}\|u_{0,j}\|_{\dot H^\eta(\R^2)}
        \le C t^{\eta/2} 2^{j/2} c_j, \label{eq:lowfr2}\\
        \sqrt{t}\,\|\nabla u_{j}(t)\|_{L^2(\R^2)}
        &\lesssim t^{-\eta/2}\|u_{0,j}\|_{\dot H^{-\eta}(\R^2)}
        \le C t^{-\eta/2} 2^{-j/2} c_j. \label{eq:highfr2}
    \end{align}
    We now choose
    \begin{align}
        N(t) := -\eta \log_2 t,
        \qquad \text{so that} \qquad
        t^{\eta/2} 2^{N(t)/2} = 1.
        \end{align}
        If $j \le N(t)$, we use \eqref{eq:lowfr2} and obtain
        \begin{align}
        \sqrt{t}\,\|\nabla u_{j}(t)\|_{L^2(\R^2)}
        \lesssim t^{\eta/2} 2^{N(t)/2} c_j
        =  c_j .
        \end{align}
        If $j \ge N(t)$, we use \eqref{eq:highfr2} and obtain
        \begin{align}
        \sqrt{t}\,\|\nabla u_{j}(t)\|_{L^2(\R^2)}
        \lesssim t^{-\eta/2} 2^{-N(t)/2} c_j
        =  c_j .
    \end{align}
    Therefore
    \begin{align}\label{eq:domCOn2}
        \sqrt{t}\,\|\nabla u_{j}(t)\|_{L^2(\R^2)}
        \lesssim \, c_j.
    \end{align}
    Since $c_j \in \ell^1(\Z)$ and for each $j$ we have (see \eqref{eq:scen2Hyp2})
    \begin{align}
        \lim_{t \to 0}   \sqrt{t}\,\|\nabla u_{j}(t)\|_{L^2(\R^2)} = 0
    \end{align}
    we conclude the continuity by dominated convergence. 
\end{proof}

\subsection{Weak--strong uniqueness for Besov data}

We start with the following lemma on $H^{-1}(\R^2)$-stability. It is inspired by \cite{CrinBaratSkondricViolini2025}. We stress that here we consider two immediately strong solutions, not necessarily constructed as in Section~4, and therefore not enjoying the additional properties discussed there. We simply recall that, by \cref{def:immstrongsol-INS}, there exist constants $C_{(\rho_1,u_1)}(u_1)$ and $C_{(\rho_2,u_2)}(u_2)$ such that
\begin{align}
\sup_{i \in \{0,1,2,3\}} \sup_{t>0} A_i^0(t,u_1)
\le C_{(\rho_1,u_1)}(u_1),
\qquad
\sup_{i \in \{0,1,2,3\}} \sup_{t>0} A_i^0(t,u_2)
\le C_{(\rho_2,u_2)}(u_2).
\end{align}

\begin{lemma}\label{lem:H-1stability}
Let $(\rho_1,u_1)$ and $(\rho_2,u_2)$ be two immediately strong solutions of \eqref{eq:INS}
arising from the initial data $(\rho_0,u_1(0))$ and $(\rho_0,u_2(0))$, both satisfying \eqref{ass:weakData2}. Assume moreover that
\begin{equation}\label{eq:Linftydef}
L_\infty
:=
\exp\!\left(
\int_0^{\infty}
\norm{\nabla u_2(s)}_{L^\infty(\R^2)}\,\dd s
\right)
<\infty .
\end{equation}
Fix $T>0$. Then there exist constants $C,N>0$\footnote{Here $N$ depends on $\|u_1(0)\|_{L^2(\R^2)}$, $C_{(\rho_1,u_1)}(u_1)$, $\|u_2(0)\|_{L^2(\R^2)}$, $C_{(\rho_2,u_2)}(u_2)$, $T$, and $D$, while $C$ depends on $N$ and $L_\infty$.} such that
\begin{equation}\label{eq:supdeltarho}
\supp \rho_1(t)\cup\supp \rho_2(t)
\subset D_{N,t} := D + B_{N\sqrt{t}}(0)
\quad \text{for all } t\in(0,T),
\end{equation}
and, for every $\phi\in H^1_{\loc}(\R^2)$ and every $t\in(0,T)$,
\begin{equation}\label{eq:negStab}
\left|
\int_{\R^2}
(\rho_1(t)-\rho_2(t))\,\phi \,\dd x
\right|
\le
C\, t 
\sup_{s\in(0,t)}
\norm{\rho_1(s)(u_1(s)-u_2(s))}_{L^2(\R^2)}
\,\norm{\nabla\phi}_{L^2(D_{N,t})}.
\end{equation}
\end{lemma}

\begin{proof}
Applying \cref{prop:ImprovedDecay} to $(\rho_1,u_1)$ we obtain that there exists
\begin{align}
N_1 = N_1\!\left(
\|u_1(0)\|_{L^2(\R^2)},
C_{(\rho_1,u_1)}(u_1),
T,
D
\right)
\end{align}
such that for every $t \in [0,T)$
\begin{align}
\supp\rho_1(t) \subset D_{N_1,t}.
\end{align}
Applying \cref{prop:ImprovedDecay} also to $(\rho_2,u_2)$ we obtain, for some $N_2>0$,
\begin{align}
\supp\rho_2(t)\subset D_{N_2,t}
\qquad\text{for all } t\in[0,T).
\end{align}
Setting $N=\max\{N_1,N_2\}$ yields \eqref{eq:supdeltarho}.

To prove \eqref{eq:negStab} we use assumption \eqref{eq:Linftydef}.  
If $X_2$ denotes the flow map associated with $u_2$, then
\begin{align}
\sup_{t>0} \norm{\nabla X_2(t)}_{L^\infty(\R^2)} \le L_\infty .
\end{align}
Denoting $(\delta \rho,\delta u)=(\rho_1-\rho_2,u_1-u_2)$ we observe that $\delta\rho$ satisfies
\begin{align}\label{eq:deltarhoeul}
\partial_t \delta \rho + \dive(\delta\rho\,u_2)
=
-\dive(\rho_1\delta u),
\qquad
\delta\rho(0)=0 .
\end{align}
For every $\phi\in H^1_{\loc}(\R^2)$ we use the Lagrangian formulation of \eqref{eq:deltarhoeul} (this computation can be justified by mollification and a commutator estimate; see \cite[Prop.~4.4]{CrinBaratSkondricViolini2025}):
\begin{align}
\int_{\R^2} \delta \rho(t,x)\phi(x)\,\dd x
&=
-\int_{\R^2}\int_0^t
\dive(\rho_1\delta u)(s,X_2(s,x))
\,\dd s\;
\phi(X_2(t,x))\,\dd x .
\end{align}
Using integration by parts we obtain
\begin{align}
\abs{\int_{\R^2} \delta \rho(t,x)\phi(x)\,\dd x}
&\le
\int_0^t
\norm{\rho_1\delta u(s)}_{L^2(\R^2)}
\norm{\nabla[\phi(X_2(t,\cdot))]}_{L^2(\supp(\rho_1(s)))}
\,\dd s
\\
&\le
\sup_{s\in(0,t)}
\norm{\rho_1\delta u(s)}_{L^2(\R^2)}
\int_0^t
\norm{\nabla[\phi(X_2(t,\cdot))]}_{L^2(\supp(\rho_1(s)))}
\,\dd s .
\end{align}
Let $X_1$ denote the flow map associated with $u_1$.  
Since $X_1(s,D)=\supp(\rho_1(s))$, we have
\begin{align}
\norm{\nabla[\phi(X_2(t,\cdot))]}_{L^2(X_1(s,D))}
&\le
\norm{(\nabla\phi)(X_2(t,\cdot))}_{L^2(X_1(s,D))}
\norm{\nabla X_2(t)}_{L^\infty(\R^2)}
\\
&\le
L_\infty
\norm{\nabla\phi}_{L^2(X_2(t,X_1(s,D)))} .
\end{align}
Now let $x\in D$. Using \eqref{eq:supdeltarho} and $s\le t$ we estimate
\begin{align}
|X_2(t,X_1(s,x))-x|
&\le
|X_2(t,X_1(s,x))-X_2(t,x)|
+
|X_2(t,x)-x|
\\
&\le
\norm{\nabla X_2(t)}_{L^\infty(\R^2)}
|X_1(s,x)-x|
+
N_2\sqrt{t}
\\
&\le
L_\infty N_1\sqrt{t}+N_2\sqrt{t}.
\end{align}
Hence, up to redefining $N$, we obtain
\begin{align}
\norm{\nabla\phi}_{L^2(X_2(t,X_1(s,D)))}
\le
\norm{\nabla\phi}_{L^2(D_{N,t})}.
\end{align}
Inserting this estimate into the previous inequality yields \eqref{eq:negStab}.
\end{proof}

We now state the uniqueness result in the Besov setting. This is a weak--strong uniqueness result: the solution $(\rho_2,u_2)$ has Besov regularity, while $(\rho_1,u_1)$ is only immediately strong and satisfies the regularity estimates propagated from $L^2$ data. In addition, the result yields a quantitative stability estimate between the two solutions.

\begin{theorem}[Uniqueness]
Assume that $\rho_0 = \mathbf{1}_D$, where $D$ is a bounded Lipschitz domain, and let
$u_1(0) \in L^2(\R^2)$ and $u_2(0) \in \dot{B}^0_{2,1}(\R^2)$.
Let $(\rho_1,u_1)$ and $(\rho_2,u_2)$ be two immediately strong solutions of \eqref{eq:INS}
arising from the initial data $(\rho_0,u_1(0))$ and $(\rho_0,u_2(0))$, respectively.

Then there exists a function $\gamma \in L^1(0,\infty)$ such that, for almost every $t>0$, there exists a constant
\begin{equation}\label{eq:Cdep}
C_t = C\!\left(t,\|u_1(0)\|_{L^2(\R^2)}, C_{(\rho_1,u_1)}(u_1), \|u_2(0)\|_{\dot{B}^0_{2,1}(\R^2)}, D\right),
\end{equation}
\footnote{One can verify that $C_t$ is non-decreasing with respect to $\|u_1(0)\|_{L^2(\R^2)}$, $C_{(\rho_1,u_1)}(u_1)$, and $\|u_2(0)\|_{\dot{B}^0_{2,1}(\R^2)}$.}
for which
\begin{equation}
\sup_{s\in(0,t)} \|\rho_1(s)[u_1(s)-u_2(s)]\|_{L^2(\R^2)}^2 +  \int_0^t \|\nabla \delta u(s)\|_{L^2(\R^2)}^2\,\dd s
\le
\|\rho_0[u_1(0)-u_2(0)]\|_{L^2(\R^2)}^2
\exp\!\Bigl(C_t\int_0^t \gamma(s)\,\dd s\Bigr).
\end{equation}

In particular, if $u_1(0)=u_2(0)$, then $(\rho_1,u_1)=(\rho_2,u_2)$ almost everywhere in $(0,\infty)\times \R^2$.
\end{theorem}
\begin{proof}
Throughout the proof, any constant with the dependence specified in \eqref{eq:Cdep} will be denoted by $C_t$ adn frequently omiteed by $\lesssim$.

Without loss of generality, we assume that $(\rho_2,u_2)$ is constructed as in \cref{prop:existenceL2} and satisfies all the properties stated in \cref{prop:existencebesov}. In particular, thanks to \cref{rem:Immstrong}, there exists a constant $C$ depending only on $\|u_2(0)\|_{L^2(\R^2)}$ such that
\begin{align}
C_{(\rho_2,u_2)}(u_2)
= C\bigl(\|u_2(0)\|_{L^2(\R^2)}\bigr)
\le
C\bigl(\|u_2(0)\|_{\dot{B}^0_{2,1}(\R^2)}\bigr)
\lesssim C_t.
\end{align}
    
Set
\begin{align}
\delta \rho \coloneqq \rho_1 - \rho_2,
\qquad
\delta u \coloneqq u_1 - u_2 .
\end{align}
For convenience, introduce
\begin{align}\label{eq:f_g_def}
f(s) &\coloneqq \sup_{\tau\in(0,s)} \|\rho_1(\tau)\,\delta u(\tau)\|_{L^2(\R^2)}, \qquad
g(s) \coloneqq \|\nabla \delta u(s)\|_{L^2(\R^2)} .
\end{align}
By \cref{prop:existencebesov}, the solution $(\rho_2,u_2)$ has sufficient regularity to apply
\cref{thm:relaEnergyNL} with $v_1=u_1$ and $v_2 = u_2$. Therefore, for almost every $t>0$, we obtain
\begin{align}\label{eq:rel_energy_ineq2}
\frac{1}{2}f^2(t) + \int_0^t g(s)^2\,\dd s
\le
\frac{1}{2}\|\rho_0\,\delta u(0)\|_{L^2(\R^2)}^2
+ \int_0^t I(s)\,\dd s
+ \int_0^t J(s)\,\dd s ,
\end{align}
where
\begin{align}\label{eq:I_J_def}
I(s) \coloneqq \Bigl| \int_{\R^2} \rho_1\,\delta u \otimes \delta u : \nabla u_2 \Bigr|, \qquad
J(s) \coloneqq \Bigl| \int_{\R^2} \delta \rho\,\delta u \cdot \dot u_2 \Bigr| .
\end{align}

\underline{Estimates on $\gamma$.}
Before studying \eqref{eq:rel_energy_ineq2} and attempting to recover a Gronwall-type inequality, we show that
\begin{align}\label{eq:gammadef}
\gamma(s) \coloneqq 
\|\nabla u_2(s)\|_{L^\infty(\R^2)}
+ s^{\frac52}\,\|\nabla \dot u_2(s)\|_{L^4(\R^2)}^2
+ s\,\|\rho_2(s)\,\dot u_2(s)\|_{L^2(\R^2)}^2
+ s^{\frac34}\,\|\nabla \dot u_2(s)\|_{L^4(\R^2)}
\end{align}
belongs to $L^1(0,\infty)$.

The first and the last terms belong to $L^1(0,\infty)$ thanks to \cref{prop:existencebesov}.
The third term belongs to $L^1(0,\infty)$ thanks to the bound $A_1(\rho_2,u_2)$.
It remains to control the second term.

Using the Gagliardo--Nirenberg inequality
\begin{align}
\|\nabla \dot u_2\|_{L^4}^2
\lesssim
\|\nabla \dot u_2\|_{L^2}\,\|\nabla^2 \dot u_2\|_{L^2},
\end{align}
we obtain
\begin{align}\label{eq:first_term_gamma}
\int_0^t s^{\frac52}\,\|\nabla \dot u_2(s)\|_{L^4(\R^2)}^2 \,\dd s
&\lesssim
\int_0^t s^{\frac52}\,\|\nabla \dot u_2(s)\|_{L^2(\R^2)}
\,\|\nabla^2 \dot u_2(s)\|_{L^2(\R^2)}\,\dd s \\
&\le
\frac12\int_0^t s^{2}\,\|\nabla \dot u_2(s)\|_{L^2(\R^2)}^2\,\dd s
+\frac12\int_0^t s^{3}\,\|\nabla^2 \dot u_2(s)\|_{L^2(\R^2)}^2\,\dd s .
\end{align}
Both terms are finite thanks to $A_2(\rho_2,u_2)$ and $A_3(\rho_2,u_2)$.
Therefore, $\gamma \in L^1(0,\infty)$.

\underline{Estimate of $I(s)$.}
By Hölder's inequality we obtain
\begin{align}\label{eq:I_est}
I(s)
\le
\|\rho_1(s)\,\delta u(s)\|_{L^2(\R^2)}^2
\|\nabla u_2(s)\|_{L^\infty(\R^2)}
\le
f(s)^2\,\gamma(s).
\end{align}

\underline{Estimate of $J(s)$.} Since $\nabla u_2 \in L^1((0,\infty); L^\infty(\R^2))$, we may apply \cref{lem:H-1stability} with 
\begin{align}
    \phi=\delta u(s)\cdot \dot u_2(s)
\end{align}
to obtain the existence of $N>0$ such that, for almost every $s\in(0,t)$,
\begin{align}\label{eq:J_start}
J(s)
\le C_t\, s\,
\sup_{\tau\in(0,s)} \|\rho_1(\tau)\,\delta u(\tau)\|_{L^2(\R^2)}\,
\|\nabla(\delta u(s)\cdot \dot u_2(s))\|_{L^2(D_{N,s})},
\end{align}
and
\begin{align}\label{eq:supp_in_DNs}
\supp \rho_1(\tau)\cup \supp \rho_2(\tau)\subset D_{N,s}
\qquad \text{for all } \tau\in(0,s).
\end{align}
Using the product rule
\begin{align}
\nabla(\delta u\cdot \dot u_2) = (\nabla \delta u)\,\dot u_2 + \delta u \cdot \nabla \dot u_2,
\end{align}
and \eqref{eq:f_g_def}, we deduce from \eqref{eq:J_start} that
\begin{align}\label{eq:J_split}
J(s)
&\lesssim s\,f(s)\,g(s)\,\|\dot u_2(s)\|_{L^\infty(D_{N,s})}
+ s\,f(s)\,\|\delta u(s)\|_{L^4(D_{N,s})}\,\|\nabla \dot u_2(s)\|_{L^4(\R^2)} .
\end{align}

By \cref{cor:Lpstrip} applied to $(\rho_1,u_1)$ with $v=\delta u(s)$ and $R=N$, we obtain for a.e.\ $s\in(0,t)$
\begin{align}\label{eq:strip_du}
s^{\frac14}\,\|\delta u(s)\|_{L^4(D_{N,s})}
\;\lesssim\;
s^{\frac12}\,\|\nabla\delta u(s)\|_{L^2(\R^2)}
+\|\rho_1(s)\,\delta u(s)\|_{L^2(\R^2)}
=
s^{\frac12} g(s)+f(s).
\end{align}
With the same argument applied to $(\rho_2,u_2)$ with $v=\dot u_2(s)$ (and the same $R=N$), we obtain
\begin{align}\label{eq:strip_dot_u2}
s^{\frac12}\,\|\dot u_2(s)\|_{L^\infty(D_{N,s})}
\;\lesssim\;
s^{\frac34}\,\|\nabla \dot u_2(s)\|_{L^4(\R^2)}
+\|\rho_2(s)\,\dot u_2(s)\|_{L^2(\R^2)} .
\end{align}
Inserting \eqref{eq:strip_du}--\eqref{eq:strip_dot_u2} into \eqref{eq:J_split} yields
\begin{align}\label{eq:J_inserted}
J(s)
&\le C_t\, f(s)\,s\,g(s)\Bigl[
s^{\frac14}\|\nabla \dot u_2(s)\|_{L^4(\R^2)}
+s^{-\frac12}\|\rho_2(s)\,\dot u_2(s)\|_{L^2(\R^2)}
\Bigr] \\
&\quad
+ C_t\, f(s)\,s^{\frac34}\Bigl[s^{\frac12}g(s)+f(s)\Bigr]\,
\|\nabla \dot u_2(s)\|_{L^4(\R^2)} .
\end{align}
Finally, by Young's inequality, we obtain (up to redefining $C_t$)
\begin{align}\label{eq:J_final}
J(s)
\le
\frac{1}{2} g(s)^2
+ C_t\, f(s)^2\,\gamma(s).
\end{align}

\underline{Conclusion (Gronwall).}
Combining \eqref{eq:rel_energy_ineq2}, \eqref{eq:I_est}, and \eqref{eq:J_final}, and absorbing the factor $g(s)^2$, we obtain for a.e.\ $t>0$,
\begin{align}\label{eq:gronwall_form}
f(t)^2 + \int_0^t g(s)^2\,\dd s
\le
2 \|\rho_0\,\delta u(0)\|_{L^2(\R^2)}^2
+ C_t \int_0^t \gamma(s)\,f(s)^2\,\dd s .
\end{align}
Now we define the relative energy as
\begin{align}
E_{\operatorname{rel}}(t)
=
\sup_{s\in(0,t)} \|\rho_1(s)\,\delta u(s)\|_{L^2(\R^2)}^2
+ \int_0^t \|\nabla \delta u(s)\|_{L^2(\R^2)}^2\,\dd s .
\end{align}
Since $f(t)^2 \le 2E_{\mathrm{rel}}(t)$, we obtain from \eqref{eq:gronwall_form}
\begin{align}\label{eq:gronwall_form_rel}
E_{\operatorname{rel}}(t)
\le
2 \|\rho_0\,\delta u(0)\|_{L^2(\R^2)}^2
+ C_t \int_0^t \gamma(s)\,E_{\operatorname{rel}}(s)\,\dd s .
\end{align}
By Gronwall's lemma we conclude
\begin{align}\label{eq:gronwall_conclusion}
E_{\operatorname{rel}}(t)
\le
\|\rho_0\,\delta u(0)\|_{L^2(\R^2)}^2
\exp\!\Bigl(C_t\int_0^t \gamma(s)\,\dd s\Bigr).
\end{align}

If the two initial data coincide, we obtain that
\begin{align}
\|\nabla \delta u(s)\|_{L^2(\R^2)} = 0
\quad \text{for a.e. } s\in(0,t),
\end{align}
which implies that
\begin{align}
u_1(s,x) - u_2(s,x) = c(s)
\end{align}
for some function $c:(0,t)\to\R^2$. Moreover,
\begin{align}
0
&=
\sup_{s\in(0,t)} \|\rho_1(s)\,\delta u(s)\|_{L^2(\R^2)}
=
\sup_{s\in(0,t)} |c(s)|\,\|\rho_1(s)\|_{L^2(\R^2)} =
|D|^{1/2}\sup_{s\in(0,t)} |c(s)| .
\end{align}
This implies that $c(s)=0$ for almost every $s\in(0,t)$, and therefore
\begin{align}
u_1 = u_2 \quad \text{a.e. in } (0,t)\times\R^2 .
\end{align}
By the freedom in the choice of $t>0$ we conclude
\begin{align}
u_1 = u_2 \quad \text{a.e. in } (0,\infty)\times\R^2 .
\end{align}
\end{proof}

\section{Additional properties of Besov solutions}
In this section, we study properties of the unique solution $(\rho,u)$ associated with the initial data
$\rho_0=\mathbf{1}_D$ and $u_0\in\dot B^0_{2,1}(\R^2)$, constructed in \cref{prop:gluinglevel0}.

By \cref{prop:existencebesov}, this solution is Lipschitz. Denoting by $X$ the measure-preserving flow associated with $u$, we have
\begin{align}\label{eq:61}
    \sup_{t \in (0,\infty)} \|\nabla X(t)\|_{L^\infty(\R^2)} \leq \exp\left(\int_0^\infty \|\nabla u(t)\|_{L^\infty(\R^2)} \,\dd t \right) =: L_\infty.
\end{align}

This estimate allows us to apply Gagliardo--Nirenberg-type inequalities on the evolving domains $D_t = X(t,D)$. We only state the specific cases needed in the sequel.

\begin{lemma}\label{lem:GNonDt}
    Let $(\rho,u)$ be the solution constructed in \cref{prop:existenceL2}, and assume in addition that $u_0 \in \dot{B}^0_{2,1}(\R^2)$. Denote by $X(t)$ its flow, and let $p \in [2,\infty)$.
    Then there exists a constant 
    \begin{align}
        C= C\bigl(p,\|u_0\|_{\dot{B}^0_{2,1}(\R^2)},D\bigr)
    \end{align}
    such that for almost every $t>0$ and every $v(t) \in L^2(D_t) \cap \dot{H}^1(\R^2)$, we have
    \begin{align}\label{eq:pcase}
        \|v(t)-M_t\|_{L^p(D_t)} \leq C \|v(t)-M_t\|_{L^2(D_t)}^{2/p} \|\nabla v(t)\|_{L^2(\R^2)}^{1-2/p}, 
        \qquad 
        M_t \coloneqq \fint_{D_t} v(t,x)\,\dd x.
    \end{align}
    Similarly, there exists a constant $C=C(\|u_0\|_{\dot{B}^0_{2,1}(\R^2)},D)$ such that
     \begin{align}\label{eq:pinftycase}
        \|v(t)-M_t\|_{L^\infty(D_t)} \leq C \|v(t)-M_t\|_{L^2(D_t)}^{1/2} \|\nabla v(t)\|_{L^\infty(\R^2)}^{1/2}.
    \end{align}
\end{lemma}

\begin{proof}
Fix $t>0$. Since $X(t,\cdot)$ is bi-Lipschitz with distortion uniformly bounded by $L_\infty$, the domain $D_t=X(t,D)$ is Lipschitz with uniformly controlled character. Hence, there exists an extension operator
\begin{align}
    E_t:H^1(D_t)\to H^1(\R^2), \qquad   E_t:W^{1,\infty}(D_t)\to W^{1,\infty}(\R^2),
\end{align}
whose operator norms are bounded independently of $t$. Set
\begin{align}
    f_t(x)\coloneqq v(t,x)-M_t,
    \qquad x\in D_t.
\end{align}
Then $\int_{D_t} f_t\,\dd x=0$. By the boundedness of $E_t$, the Gagliardo--Nirenberg inequality on $\R^2$, and the Poincar\'e inequality on $D_t$, we obtain
\begin{align}
\|f_t\|_{L^p(D_t)}
&\leq \|E_t f_t\|_{L^p(\R^2)}
\lesssim \|E_t f_t\|_{L^2(\R^2)}^{2/p}\|\nabla E_t f_t\|_{L^2(\R^2)}^{1-2/p} \lesssim \|f_t\|_{L^2(D_t)}^{2/p}\|\nabla f_t\|_{L^2(D_t)}^{1-2/p}.
\end{align}
Since $\nabla f_t=\nabla v(t)$ on $D_t$, this yields the first estimate.

The $L^\infty$ case follows in the same way, using the corresponding Gagliardo--Nirenberg inequality on $\R^2$:
\begin{align}
\|f_t\|_{L^\infty(D_t)}
\lesssim \|f_t\|_{L^2(D_t)}^{1/2}\|\nabla f_t\|_{L^\infty(D_t)}^{1/2},
\end{align}
which gives \eqref{eq:pinftycase}.
\end{proof}

\subsection{\texorpdfstring{$C^{1,\gamma}$}{C¹,γ} patches}

As in \cite{Danchin2025}, we can show that the flow is uniformly $C^1(\R^2)$ which implies that $C^1$ patches are preserved under flow. More precisely, we have the following.

\begin{proposition}\label{prop:C1patch}
    Let $(\rho,u)$ be the solution constructed in \cref{prop:existenceL2}, and assume in addition that $u_0 \in \dot{B}^0_{2,1}(\R^2)$. Then there exists a constant $C>0$ such that 
    \begin{align}\label{eq:refinedbesovbound}
        \int_0^\infty \norm{\nabla u(t)}_{\dot{B}^{1/2}_{4,1}(\R^2)} \dd t \leq C.
    \end{align}
    In particular, for every $t \geq 0,$ we have $X(t) \in C^1(\R^2)$ and, if $D$ is a $C^1$-domain, then $D_t$ is a $C^1$-domain as well.
\end{proposition}

\begin{proof}
    Every constant that may depend on $\norm{u_0}_{L^2(\R^2)}$, $\norm{u_0}_{\dot{B}^{0}_{2,1}(\R^2)}$ and $D$ is omitted. 
    
    Since $\dot{B}^{1/2}_{4,1}(\R^2) \hookrightarrow C_0(\R^2)$ by \cite[Proposition 2.39]{BCD2011}, it suffices to show \eqref{eq:refinedbesovbound} to prove \cref{prop:C1patch}. In order to prove \eqref{eq:refinedbesovbound}, recall that
    \begin{align}\label{eq:estimaterefbesov}
    \begin{aligned}
        \norm{\nabla u(t)}_{\dot{B}^{1/2}_{4,1}(\R^2)} \lesssim & \|\nabla u(t)\|_{L^4(\R^2)}^{1/2}\|\nabla^2 u(t)\|_{L^4(\R^2)}^{1/2}\\
        \lesssim & t^{-1/8}\|\nabla u(t)\|_{L^2(\R^2)}^{1/4}\|\nabla^2 u(t)\|_{L^2(\R^2)}^{1/4} \|\sqrt{\rho} \dot{u}(t)\|_{L^2(\R^2)}^{1/4} t^{1/8}\|\nabla \dot{u}(t)\|_{L^2(\R^2)}^{1/4}
    \end{aligned}
    \end{align}
    In the last step, we used the Stokes estimates, that $\sqrt{\rho} \dot{u}$ is mean free thanks to \cref{lem:MomentCons} and \cref{lem:GNonDt}, i.e.,
    \begin{align}
        \|\nabla^2 u(t)\|_{L^4(\R^2)} \lesssim \|\sqrt{\rho} \dot{u}(t)\|_{L^4(\R^2)} \lesssim \|\sqrt{\rho} \dot{u}(t)\|_{L^2(\R^2)}^{1/2} \|\nabla \dot{u}(t)\|_{L^2(\R^2)}^{1/2}
    \end{align}
    We see that each factor appearing on the right-hand side of \eqref{eq:estimaterefbesov} is in $L^4((0,\infty))$ thanks to \cref{prop:existencebesov} and we deduce \eqref{eq:refinedbesovbound} which completes the proof.
\end{proof}

Higher regularity of the patch is propagated under suitable assumptions on the initial velocity.

\begin{proposition}\label{prop:C1gammapatch}
    Let $(\rho,u)$ be the solution constructed in \cref{prop:existenceL2}, and assume in addition that $u_0 \in \dot{B}^0_{2,1}(\R^2)$. If additionally $u_0 \in \dot{B}^\gamma_{2,1}(\R^2)$, where $\gamma \in (0,1),$ then there exists a constant $C>0$ such that
    \begin{align}\label{eq:L1C1gamma}
        \int_0^\infty \norm{\nabla u(t)}_{C^\gamma(\R^2)} \dd t \leq C \left( \norm{u_0}_{\dot{B}^0_{2,1}(\R^2)} + \norm{u_0}_{\dot{B}^\gamma_{2,1}(\R^2)} \right).
    \end{align}
    In particular, for every $t\geq 0,$ we have that $X(t) \in C^{1,\gamma}(\R^2)$  and, if $D$ is a $C^{1,\gamma}$-domain, then $D_t$ is a $C^{1,\gamma}$-domain as well.
\end{proposition}

We recall the definition of the $C^{\gamma}$ norm for $\gamma \in (0,1)$:
\begin{align}
    \norm{f}_{C^{\gamma}(\R^2)} \coloneqq \norm{f}_{L^\infty(\R^2)}  +  \abs{f}_{C^{\gamma}(\R^2)}, \qquad \text{where } \quad \abs{f}_{C^{\gamma}(\R^2)} \coloneqq \sup_{x, y \in \R^2} \frac{\abs{f(x) - f(y)}}{\abs{x-y}^\gamma}.
\end{align}

\begin{proof}
     Every constant that may depend on $\norm{u_0}_{L^2(\R^2)}$, $\norm{u_0}_{\dot{B}^{0}_{2,1}(\R^2)}$ and $D$ is omitted. 
    
    Again it suffices to show \eqref{eq:L1C1gamma} in order to prove \cref{prop:C1gammapatch}. Since 
    \begin{align}
        \int_0^\infty \|\nabla u(t)\|_{L^\infty(\R^2)} \dd t \leq C 
    \end{align}
    by \cref{prop:existencebesov} it only remains to show that 
    \begin{align}\label{eq:L1C1gamudot}
        \int_0^\infty |\nabla u(t)|_{C^\gamma(\R^2)} \dd t \lesssim \norm{u_0}_{\dot{B}^\gamma_{2,1}(\R^2)}
    \end{align}
    To this end, recall that $\dot{B}^\gamma_{2,1} =[L^2(\R^2),\dot{H}^1(\R^2)]_{\gamma,1}$ which implies that there is a sequence $(u_{0,j})_{j \in \Z} \subset L_\sigma^2(\R^2) \cap \dot{H}^1(\R^2)$ such that
    \begin{align}
        u_0 = \sum_{j \in \Z} u_{0,j}, \quad \sum_{j \in \Z} 2^{\gamma j}\|u_{0,j}\|_{L^2(\R^2)} + 2^{-(1-\gamma) j}\|u_{0,j}\|_{\dot{H^1}(\R^2)} \leq 2 \norm{u_0}_{\dot{B}^\gamma_{2,1}(\R^2)}.
    \end{align}
    We fix $j \in \Z$ and let $u_j$ be a solution of the linearized system \eqref{eq:LSj} with initial value $u_{0,j}.$

    First, we show that, for $p = 2/(1-\gamma)$, 
    \begin{align}\label{eq:L1Lpudot}
        \int_0^\infty \|\sqrt{\rho} \dot{u}_j(t)\|_{L^p(\R^2)} \dd t \lesssim 2^{\gamma j}\|u_{0,j}\|_{L^2(\R^2)} + 2^{-(1-\gamma) j}\|u_{0,j}\|_{\dot{H^1}(\R^2)}.
    \end{align}
    Then, by linear Stokes estimates and Morrey's inequality,
    \begin{align}
        \int_0^\infty |\nabla u_j(t)|_{C^\gamma(\R^2)} \dd t \lesssim & \int_0^\infty \|\nabla^2 u_j(t)\|_{L^p(\R^2)} \dd t \lesssim \int_0^\infty \|\sqrt{\rho} \dot{u}_j(t)\|_{L^p(\R^2)} \dd t\\ \lesssim & 2^{\gamma j}\|u_{0,j}\|_{L^2(\R^2)} + 2^{-(1-\gamma) j}\|u_{0,j}\|_{\dot{H^1}(\R^2)}.
    \end{align}
    Summing over $j \in \Z$ implies the thesis. In order to show \eqref{eq:L1Lpudot}, observe that, due to \cref{lem:GNonDt}, for every $t>0,$
    \begin{align}
        \|\sqrt{\rho} \dot{u}_j(t)\|_{L^p(\R^2)} \lesssim \|\sqrt{\rho} \dot{u}_j(t)\|_{L^2(\R^2)}^{1-\gamma} \|\nabla \dot{u}_j(t)\|_{L^2(\R^2)}^\gamma.
    \end{align}
    Thanks to \cref{thm:sect2}, for almost every $t>0,$
    \begin{align}
        \|\sqrt{\rho} \dot{u}_j(t)\|_{L^p(\R^2)} \lesssim & \|\sqrt{\rho} \dot{u}_j(t)\|_{L^2(\R^2)}^{1-\gamma} \|\nabla \dot{u}_j(t)\|_{L^2(\R^2)}^\gamma\\
        \lesssim & \left(\|\sqrt{\rho_0} u_{0,j}\|_{L^2(\R^2)}t^{-1}\right)^{1-\gamma} \left(\|\sqrt{\rho_0} u_{0,j}\|_{L^2(\R^2)}t^{-3/2}\right)^{\gamma} \lesssim \|u_0\|_{L^2(\R^2)} t^{-1-\gamma/2},
    \end{align}
    and that, for every $j \in \Z$ and almost every $t>0,$
    \begin{align}
        \|\sqrt{\rho} \dot{u}_j(t)\|_{L^p(\R^2)} \lesssim & \|\sqrt{\rho} \dot{u}_j(t)\|_{L^2(\R^2)}^{1-\gamma} \|\nabla \dot{u}_j(t)\|_{L^2(\R^2)}^\gamma\\
        \lesssim & \left(\|\nabla u_{0,j}\|_{L^2(\R^2)}t^{-1/2}\right)^{1-\gamma} \left(\|\nabla u_{0,j}\|_{L^2(\R^2)}t^{-1}\right)^{\gamma} \lesssim \|\nabla u_0\|_{L^2(\R^2)} t^{-1/2-\gamma/2}.
    \end{align}
    Consequently, for every $j \in \Z$ and every $t>0,$
    \begin{align}
        \|\sqrt{\rho} \dot{u}_j(t)\|_{L^p(\R^2)} \lesssim \min\left\{\|u_0\|_{L^2(\R^2)} t^{-1-\gamma/2}, \|\nabla u_0\|_{L^2(\R^2)} t^{-1+\gamma/2}\right\}.
    \end{align}
    We infer that, for every $j \in \Z$ and every $A>0,$
    \begin{align}
        \int_0^\infty \|\sqrt{\rho} \dot{u}_j(t)\|_{L^p(\R^2)} \dd s
        \lesssim & \|\nabla u_0\|_{L^2(\R^2)} \int_0^A t^{-1/2-\gamma/2} \dd s +\|u_0\|_{L^2(\R^2)} \int_A^\infty t^{-1-\gamma/2} \dd s\\
        \lesssim & \|\nabla u_0\|_{L^2(\R^2)} A^{1/2-\gamma/2} +\|u_0\|_{L^2(\R^2)} A^{-\gamma/2}.
    \end{align}
    We choose $A=2^{-2j}$ and obtain \eqref{eq:L1Lpudot} which completes the proof.
\end{proof}

\subsection{Asymptotics.}
By \cref{lem:MomentCons}, the average momentum of $u$ on the patch is conserved for all times:
\begin{align}\label{eq:62}
M \coloneqq \frac{1}{|D|}\int_{\R^2}\rho_0(x)u_0(x)\,\dd x
= \frac{1}{|D_t|}\int_{\R^2}\rho(t,x)u(t,x)\,\dd x
=\fint_{D_t}u(t,x)\,\dd x.
\end{align}
The combination of finite distorsion \eqref{eq:61}, \eqref{eq:62}, and the energy equality allows us to prove exponential decay of $u(t)-M$ on the patch $D_t$ in the $L^2$ norm. To track this decay, we do not assume $\nu=1$, unlike in the rest of the paper, and instead keep the viscosity explicit.

\begin{lemma}\label{lem:expdecay}
In the previous setup, we have
\begin{align}
\|u(t)-M\|_{L^2(D_t)}
\le
\|u_0-M\|_{L^2(D)}\,e^{-\lambda t},
\qquad
\lambda \coloneqq \frac{\nu}{C_D^2 L_\infty^2},
\end{align}
where $C_D$ is the Poincar\'e constant of $D$.
\end{lemma}

\begin{proof}
We introduce the Galilean transform
\begin{align}
\rho_M(t,x)\coloneqq \rho(t,x+Mt),
\qquad
u_M(t,x)\coloneqq u(t,x+Mt)-M.
\end{align}
Since the system is Galilean invariant, $(\rho_M,u_M)$ satisfies the same equations as $(\rho,u)$, and therefore for every $t>s \geq 0$ we have the energy equality 
\begin{align}
\frac12\|\sqrt{\rho_M}u_M(t)\|_{L^2(\R^2)}^2
+\nu\int_s^t\|\nabla u_M(\tau)\|_{L^2(\R^2)}^2\,\dd \tau
=
\frac12\|\sqrt{\rho_M}u_M(s)\|_{L^2(\R^2)}^2
\end{align}
Using the change of variables $y=x+Mt$ and the invariance of spatial derivatives under translations, this identity becomes
\begin{align}\label{eq:Gcons}
\frac12\|u(t)-M\|_{L^2(D_t)}^2
+\nu\int_s^t\|\nabla u(\tau)\|_{L^2(\R^2)}^2\,\dd \tau
=
\frac12\|u(s)-M\|_{L^2(D_s)}^2.
\end{align}
By a change of variables using $D_\tau = X(\tau,D)$ 
\begin{align}\label{eq:char}
\begin{aligned}
\|u(\tau)-M\|_{L^2(D_\tau)}
&=
\Big\|u(\tau,\cdot)-\fint_{D_\tau}u(\tau,x)\,\dd x\Big\|_{L^2(D_\tau)} \\
&=
\Big\|u(\tau,X(\tau,\cdot))-\fint_D u(\tau,X(\tau,y))\,\dd y\Big\|_{L^2(D)}
\end{aligned}
\end{align}
By Poincar\'e inequality on $D$ and recalling that $\|\nabla X(t)\|_{L^\infty(\R^2)}\le L_\infty$, we infer that
\begin{align}
\Big\|u(\tau,X(\tau,\cdot))-\fint_D u(\tau,X(\tau,y))\,\dd y\Big\|_{L^2(D)} 
& \le
C_D \|\nabla (u(\tau)\circ X(\tau,\cdot))\|_{L^2(D)} \\
&\le
C_D \|\nabla u(\tau)\|_{L^2(\R^2)}\|\nabla X(\tau)\|_{L^\infty(\R^2)} \\  & \le
C_D L_\infty \|\nabla u(\tau)\|_{L^2(\R^2)}.
\end{align}
We have proved
\begin{align}
\|u(\tau)-M\|_{L^2(D_\tau)}
\le
C_D L_\infty \|\nabla u(\tau)\|_{L^2(\R^2)}.
\end{align}
Therefore,
\begin{align}\label{eq:thisexp}
\|\nabla u(\tau)\|_{L^2(\R^2)}^2
\ge
\frac{1}{C_D^2L_\infty^2}\|u(\tau)-M\|_{L^2(D_\tau)}^2.
\end{align}
Setting
\begin{align}
F(t)\coloneqq \frac12\|u(t)-M\|_{L^2(D_t)}^2,
\end{align}
and inserting  \eqref{eq:thisexp} into \eqref{eq:Gcons} gives for every $t>s\geq 0$
\begin{align}\label{eq:Gcons2}
F(t)
+2\lambda \int_s^tF(\tau)\,\dd \tau 
\leq F(s).
\end{align}
By Gronwall's lemma, it follows that
\begin{align}
F(t)\le F(0)e^{-2\lambda t}.
\end{align}
Taking square roots, we obtain the thesis.
\end{proof}

\begin{lemma}\label{lem:finaldomain}
In the previous setting, there exists a measure-preserving map $X_\infty$ such that
\begin{align}\label{eq:unifX}
   \lim_{t \to \infty} \bigl\|( X(t,\cdot ) - Mt) - X_\infty (\cdot) \bigr\|_{L^\infty(D)} = 0.
\end{align}
Moreover, for every $x \in D$, we have
\begin{align}\label{eq:last1}
    X_\infty(x)
    =
    x + \int_0^\infty \bigl(u(t, X(t,x)) - M\bigr)\,\dd t,
\end{align}
where
\begin{align}\label{eq:last2}
\sup_{x \in D}|u(t, X(t,x)) - M| \in L^1(0,\infty).
\end{align}
\end{lemma}

\begin{proof}
The goal is to show that
\begin{align}\label{eq:cauchy}
\bigl(X(t,\cdot ) - Mt\bigr)_{t>0}
\qquad\text{is Cauchy in } L^\infty(D),
\end{align}
which immediately implies \eqref{eq:unifX}. For every $x\in D$ and $0\le s<t$, we have
\begin{align}\label{eq:623}
|X(t,x)-Mt-(X(s,x)-Ms)|
&=
\left|\int_s^t \bigl(u(\tau,X(\tau,x))-M\bigr)\,\dd \tau\right| \le
\int_s^t \|u(\tau)-M\|_{L^\infty(D_\tau)}\,\dd \tau .
\end{align}
By \eqref{eq:pinftycase} with $v = u$ ther exist a constant $C=C(\norm{u_0}_{\dot{B}^0_{2,1}(\R^2)},D)$ such that 
\begin{align}\label{eq:622}
\|u(\tau)-M\|_{L^\infty(D_\tau)} \leq C
\|u(\tau)-M\|_{L^2(D_\tau)}^{\frac12}\,
\|\nabla u(\tau)\|_{L^\infty(\R^2)}^{\frac12}.
\end{align}
Putting together \eqref{eq:623} and \eqref{eq:622}, and using \cref{lem:expdecay}, we obtain
\begin{align}
\norm{X(t,x)-Mt-(X(s,x)-Ms)}_{L^\infty(D)}
&\leq C
\int_s^t
\|u(\tau)-M\|_{L^2(D_\tau)}^{\frac12}
\|\nabla u(\tau)\|_{L^\infty(\R^2)}^{\frac12}\,\dd \tau \\
&\leq C
\int_s^t
e^{-\frac{\lambda \tau}{2}}
\|\nabla u(\tau)\|_{L^\infty(\R^2)}^{\frac12}\,\dd \tau .
\end{align}
Since $\nabla u\in L^1((0,\infty);L^\infty(\R^2))$, Hölder's inequality yields
\begin{align}
H(\tau)\coloneqq e^{-\frac{\lambda \tau}{2}}\|\nabla u(\tau)\|_{L^\infty(\R^2)}^{\frac12}\in L^1(0,\infty).
\end{align}
Therefore,
\begin{align}
\norm{X(t,x)-Mt-(X(s,x)-Ms)}_{L^\infty(D)}
\le
\int_s^t H(\tau)\,\dd \tau,
\end{align}
which proves \eqref{eq:cauchy} and \eqref{eq:last2}. The solution formula \eqref{eq:last1} is obatine dy passing to the limit for $t \to \infty$ the equality 
\begin{align}
X(t,x)-Mt
&=\int_0^t \bigl(u(\tau,X(\tau,x))-M\bigr)\,\dd \tau.
\end{align}
\end{proof}
\begin{remark}
The last estimate of the previous lemma implies that there exists a
bi-Lipschitz and measure-preserving map $X_\infty$ such that
\begin{align}
\|X(t,\cdot)-Mt - X_\infty(\cdot)\|_{L^\infty(D)} \to 0
\qquad \text{as } t\to\infty .
\end{align}
In particular,
\begin{align}
D_t - Mt = X(t,D)-Mt \to X_\infty(D) \eqqcolon D_\infty .
\end{align}
\end{remark}
\textbf{Comments.}
\begin{enumerate}
\item We have shown that, for large times, the patch converges to a Lipschitz domain $D_\infty$ which moves with constant velocity $M$ (the average initial momentum). In other words, viscosity gradually removes the deformations of the patch, and the only motion that remains is a rigid translation. This behavior is expected: viscosity smooths the velocity field, while the conservation of momentum prevents the flow from stopping.

\item Notice that $L_\infty$ is dimensionless and therefore
\begin{align}
\lambda \approx \frac{\nu}{C_D^2}
\approx
\frac{\nu}{\operatorname{diam}(D)^2}
\approx
\left[\frac{1}{T}\right].
\end{align}
Hence the characteristic time scale of the system is
\begin{align}
T_* \sim \frac{1}{\lambda} \sim \frac{\operatorname{diam}(D)^2}{\nu}.
\end{align}
This is the time needed for viscosity to smooth the velocity inside the patch. For $t \ll T_*$ the velocity may still vary significantly in the patch, while for $t \gg T_*$ these variations become small and the motion approaches the asymptotic regime described above. This time scale is the same as the classical diffusive time scale of the heat equation.
\end{enumerate}

\appendix
\section{The relative energy inequality}
\subsection{The linear case}
We first study the following scenario.

\noindent\textbf{Scenario 1}\label{sce:scenario1}.  
\begin{itemize}
    \item[(i)] $(\rho,u)$ is an immediately strong solution of \eqref{eq:INS}
    with initial data $(\rho_0,u_0)$ satisfying \eqref{ass:weakData2};
    \item[(ii)]  $v$ is an immediately strong solution of the linearized system
    \eqref{eq:LSs} advected by $(\rho,u)$, such that
    $\rho v \to \rho_0 v_0 \in L^2(\R^2)$.
\end{itemize}

\begin{theorem}\label{thm:relaEnergyL}
Under the assumptions of Scenario~1, for almost every $t> 0$ we have
\begin{align}
    \|\rho(t)\,(u(t)-v(t))\|_{L^2(\R^2)}^{2}
    + 2\int_0^t \|\nabla (u(s)-v(s))\|_{L^2(\R^2)}^{2}\,ds
    \le \|\rho_0\,(u_0-v_0)\|_{L^2(\R^2)}^{2}.
\end{align}
\end{theorem}

\begin{proof}
Set $w:=u-v$.  
Using the energy inequalities for $u$ and $v$, yields
\begin{align}
    \|\rho(t) w(t)\|_{L^2(\R^2)}^{2}
    &= \|\rho(t)u(t)\|_{L^2(\R^2)}^{2}
       - 2\int_{\R^2} \rho(t)\,u(t)\cdot v(t)\,dx
       + \|\rho(t)v(t)\|_{L^2(\R^2)}^{2} \nonumber\\
    &\le \|\rho_0u_0\|_{L^2(\R^2)}^{2}
       + \|\rho_0v_0\|_{L^2(\R^2)}^{2}
       - 2\int_0^t \Bigl(
           \|\nabla u(s)\|_{L^2(\R^2)}^{2}
           + \|\nabla v(s)\|_{L^2(\R^2)}^{2}
         \Bigr)\,ds \nonumber\\
    &\hspace{2cm}
       - 2\int_{\R^2} \rho(t)\,u(t)\cdot v(t)\,dx .
\end{align}

Using \cref{prop:admuMomL}, we infer that for almost every $t>0$,
\begin{align}\label{eq:momm}
    \int_{\R^2} \rho(t)\,u(t)\cdot v(t)\,dx
    - \int_{\R^2} \rho_0\,u_0\cdot v_0\,dx
    = -2\int_0^t\!\!\int_{\R^2} \nabla u(s,x):\nabla v(s,x)\,dx\,ds .
\end{align}
Inserting \eqref{eq:momm} into the previous inequality and rearranging the terms we conclude the proof.
\end{proof}

To show \cref{prop:admuMomL} we need some preliminary result.  We first show an admissibility result for the transport equation.

\begin{proposition}\label{prop:admuTrasp}
Assume that $(\rho,u)$ is an immediately strong solution of \eqref{eq:INS} with initial data satisfying \eqref{ass:weakData2}. Let $\varphi \in W^{1,1}_{\loc}((0,\infty)\times\R^2)$ be such that
\begin{align}
\rho \varphi,\ \rho \nabla \varphi \in L_\loc^1([0,\infty);L^2(\R^2)), \quad \ \rho\,\partial_t \varphi \in L_\loc^1([0,\infty);L^1(\R^2)),  
\quad 
\rho \varphi \in C([0,\infty);L^1(\R^2)).
\end{align}
Then, for almost every $t>0$, we have 
\begin{align}\label{eq:weakTr}
    \int_{\R^2} \rho(t)\,\varphi(t)
    - \int_{\R^2} \rho_0\,\varphi(0)
    = \int_0^t\!\!\int_{\R^2} 
      \rho\,\partial_t \varphi 
      + \rho \,u  \cdot \nabla \varphi.
\end{align}
\end{proposition}
\begin{proof}
Fix $\eps>0$ and let $\eta \in C_c^\infty(\R^2)$ be such that $0\le \eta \le 1$, 
$\eta \equiv 1$ on $B_1(0)$ and $\supp \eta \subset B_2(0)$.
For $n\ge 1$ set
\begin{align}
    \chi_n(x) := \eta\Bigl(\frac{x}{n}\Bigr),
    \qquad  \varphi^n := \varphi\,\chi_n.
\end{align}
By assumption we have $\varphi^n \in W^{1,1}((\varepsilon,\infty);W^{1,1}(\R^2))$ and, as in \cite[Lemma~3.4]{CrinBaratSkondricViolini2025}, one can show that it is an admissible test function for the weak formulation of the 
transport equation of $(\rho,u)$. 
Therefore, for a.e.\ $s,t\in(\eps,\infty)$,
\begin{align}\label{eq:Transportw_new}
   \int_{\R^2} \rho(t)\,\varphi^n(t)
    - \int_{\R^2} \rho(s)\,\varphi^n(s)
    = \int_s^t\!\!\int_{\R^2}
      \rho\,\partial_t\varphi^n 
      + \rho \,u\cdot \nabla \varphi^n .
\end{align}
By the continuity assumption on $\rho u $ the left hand side of \eqref{eq:Transportw_new}  converges to the left hand side of \eqref{eq:weakTr} for $n \to \infty$ and $s \to 0 $. Since $\rho \partial_t \varphi \in L^1_tL^1_x$ also the first term of the right hand side converges. For the last term we have 
\begin{align}
\int_s^t\!\!\int_{\R^2} \rho\,u \cdot(\nabla \varphi^n - \nabla \varphi)
& \leq \norm{\rho u }_{L^\infty((0,t);L^2(\R^2)} \norm{ \rho \nabla ( \varphi^n - \varphi)}_{L^1((0,t);L^2(\R^2)}
\end{align}
The last factor is finite since $\rho \nabla \varphi \in L^1((0,t);L^2(\R^2))$. 
Moreover, since $\rho(t)=\mathbf{1}_{D_t}$ with $D_t$ bounded, for $n$ large enough we have $\chi_n \equiv 1$ on $\supp \rho(s)$ for all $s\in[0,t]$, hence
\begin{align}
    \rho \nabla(\varphi^n - \varphi)
    = \rho\,\varphi \nabla \chi_n .
\end{align}
Therefore,
\begin{align}
    \norm{\rho \nabla ( \varphi^n - \varphi)}_{L^1((0,t);L^2(\R^2))}
    \leq \frac{C}{n}\norm{\rho \varphi}_{L^1((0,t);L^2(\R^2))} \to 0,
\end{align}
which concludes the proof.
\end{proof}

\cref{prop:admuTrasp} allows us to show that in the first Scenario the admissibility of $u \cdot v  $ for the transport equation of $(\rho,u)$ in the time interval $(\eps,\infty)$.

\begin{corollary}\label{cor:admutrasp}
Under the hypothesis of Scenario 1 for any $\eps>0$ and almost every $t \in (\eps,\infty)$, we have 
\begin{align}\label{eq:traspvu}
  \int_{\R^2} \rho (t) u(t) \cdot v(t)
    - \int_{\R^2} \rho (\eps) u(\eps) \cdot v(\eps)
    = \int_\eps^t\!\!\int_{\R^2} 
      \rho v \cdot  \dot{u} + \rho u \cdot D_u (v).
\end{align}
Moreover we have the following convergence
\begin{align}
    \lim_{\eps \to 0} \frac{1}{2 }\int_{\R^2} \rho (\eps) u(\eps) \cdot v(\eps) =  \frac{1}{2 }\int_{\R^2} \rho_0 u_0 \cdot v_0.
\end{align}
\end{corollary}
\begin{proof}
     We simply check the assumptions of \cref{prop:admuTrasp} to show that $u\cdot v$ is an admissible test function. 
    By \cref{lem:B1} $u \cdot v $ belongs to $W^{1,1}_{\loc}((0,\infty)\times\R^2)$ and we have 
    \begin{align}
    \begin{aligned}\label{eq:bds}
        \norm{ \rho v \cdot u}_{L^1_tL_x^2} & \leq   \norm{ \rho u }_{L^\infty_tL_x^2}   \norm{ \rho v}_{L^1_tL_x^\infty} \\ 
          \norm{ \rho \nabla (v \cdot u)}_{L^1_tL_x^2} & \leq  \norm{ \nabla v}_{L^2_tL_x^2}  \norm{ \rho u }_{L^2_tL_x^\infty} +  \norm{ \rho v }_{L^2_tL_x^\infty} \norm{  \nabla u}_{L^2_tL_x^2} \\ 
          \norm{\rho \partial_t ( u \cdot v )}_{L^1_tL^1_x } & \leq   \norm{\rho \partial_t u}_{L^1_tL^2_x } \norm{\rho v }_{L_t^\infty L^2_x } +  \norm{\rho u}_{L_t^\infty L^2_x } \norm{ \rho \partial_t v}_{L_t^1L^2_x }.
    \end{aligned}
    \end{align}
    By energy inequality $ \rho u , \; \rho v \in L^\infty_tL^2_x$ and $ \nabla u , \; \nabla v \in L^2_t L^2_x$. Away from initial time, by definitions of \cref{def:immstrongsol-INS} and \cref{def:immstrongsol-LS} we have 
    \begin{align}
      \rho \partial_t u , \; \rho \partial_t v \in L^1((\eps,\infty); L^2(\R^2)).
    \end{align}
Fix $\varepsilon>0$ and $T>0$. 
We do not keep track of the precise dependence of the constants,
since it is not needed; all implicit constants may depend on $T$. By \cref{prop:ImprovedDecay}, for every $t \in [0,T]$,
\begin{align}
\operatorname{supp}\rho(t)
\subset D + B_{C\sqrt{t}}(0).
\end{align}
Applying \cref{cor:Lpstrip} with $v = v(t)$ and $R = C$, we obtain
\begin{equation}\label{eq:vinfty_local}
t^{\frac12}
\|v(t)\|_{L^\infty(D_{C,t})}
\lesssim 
t^{\frac34}\|\nabla v(t)\|_{L^4(\mathbb{R}^2)}
+
\|\rho(t)\,v(t)\|_{L^2(\mathbb{R}^2)} .
\end{equation}
Combining the support inclusion with \eqref{eq:vinfty_local}, we deduce
\begin{align}
t^{\frac12}
\|\rho(t)v(t)\|_{L^\infty(\mathbb{R}^2)}
&\lesssim
t^{\frac34}\|\nabla v(t)\|_{L^4(\mathbb{R}^2)}
+
\|\rho(t)\,v(t)\|_{L^2(\mathbb{R}^2)} \\
&\lesssim
t^{\frac34}
\|\nabla v(t)\|_{L^2(\mathbb{R}^2)}^{\frac12}
\|\nabla^2 v(t)\|_{L^2(\mathbb{R}^2)}^{\frac12}
+
\|\rho(t)\,v(t)\|_{L^2(\mathbb{R}^2)} \lesssim_{\varepsilon}
C_{(\rho,u)}(v).
\end{align}
In particular,
\begin{align}
\rho v \in L^\infty_{\loc}\bigl((0,\infty); L^\infty(\mathbb{R}^2)\bigr).
\end{align}
A similar computation shows that $ \rho u \in L^\infty_{\loc}\bigl((0,\infty); L^\infty(\mathbb{R}^2)\bigr) $. Hence we have proved that the right hand sides of \eqref{eq:bds} are bounded provided the time intervals are compact subsets of $(0,\infty)$. 

  We are left to study the continuity. First, we decompose 
   \begin{align}
       \int_{\R^2}  \rho(\eps) u(\eps) \cdot v(\eps) - \rho_0 u_0 \cdot v_0 & =  \int_{\R^2}  [ \rho(\eps) u(\eps) - \rho_0 u_0 ] \cdot v_0 + \int_{\R^2} \rho(\eps) u(\eps) \cdot [ v(\eps) -v_0]
   \end{align}
  where the first addend converges to zero because $\mathbb{P}(\rho u) \in C L^2_w$, for the second addend we write
  \begin{align}
      \int_{\R^2} \rho u(\eps) \cdot [ v(\eps) -v_0)] & \leq \norm{\rho_0 u_0}_{L^2(\R^2)} \norm{\rho(\eps) ( v(\eps) - v_0}_{L^2(\R^2)} \\ & \lesssim  \underbrace{\norm{ \rho(\eps) v(\eps) - \rho_0 v_0}_{L^2(\R^2)}}_{I(\eps)} + \underbrace{\norm{ (\rho(\eps) -\rho_0) v_0}_{L^2(\R^2)}}_{J(\eps)} .
  \end{align}
 $I(\eps) \to 0 $ because $\rho v \to \rho_0 v_0$ in $L^2(\R^2)$ by assumption.  For the convergence of $J(\eps)$ let $\delta >0 $ and choose $M=M(\delta)$ and $\Tilde{\eps}= \Tilde{\eps}(M,\delta) = \Tilde{\eps}(\delta) $ such that 
 \begin{align}
     \norm{ v_0}_{L^2(\abs{v_0}>M)} < \delta, \qquad 
     M \norm{\rho(\Tilde{\eps}) - \rho_0}_{L^2(\R^2)} < \delta.
 \end{align}
 The second choice is possible since $\rho \in C_t L^2_x$ by \cref{cor:lingrowthu}. Then for $\eps < \Tilde{\eps}$ we have 
 \begin{align}
     J(\eps) &  \leq \norm{ (\rho(\eps) - \rho_0) v_0}_{L^2(\abs{v_0}<M)} + \norm{ (\rho(\eps) - \rho_0) v_0}_{L^2(\abs{v_0} > M)} \\ &  \leq M  \norm{ \rho(\eps) - \rho_0 }_{L^2(\R^2)} + \norm{\rho(\eps) - \rho_0}_{L^\infty(\R^2))}\norm{  v_0}_{L^2(\abs{v_0} > M)}   \leq 2\delta 
 \end{align}
 which proves the convergence.
\end{proof}

\begin{proposition}\label{prop:admuMomL}
Under the hypothesis of Scenario 1 we have for almost every $t>0$
\begin{align}
   \int_{\R^2} \rho(t)\,u(t)\cdot v(t)
   - \int_{\R^2} \rho_0\,u_0\cdot v_0
   &= -2\int_0^t\!\!\int_{\R^2} \nabla u : \nabla v.
\end{align}

\end{proposition}
\begin{proof}
By \cref{prop:tailest} for almost every time in $t >0 $ we have, as $\abs{x} \to \infty$, 
\begin{align}
    \abs{u(t,x)}=\abs{v(t,x)} = O(1).
\end{align}
Then by \cref{cor:tailEst} with $v  = u$ and $h = v(t)$ we get for almost every $t > 0$ that
\begin{align}
\int_{\mathbb{R}^2} \rho \dot u(t,x)\cdot v(t,x)\,\dd x
=
- \int_{\mathbb{R}^2} \nabla u(t,x):\nabla v(t,x)\,\dd x .
\end{align}
and exchanging $u$ and $v$ in \cref{cor:tailEst} (using the shortcut $\dot{v}= D_u(v)$) we also obtain 
\begin{align}
\int_{\mathbb{R}^2} \rho \dot{v}(t,x)\cdot u(t,x)\,\dd x
=
- \int_{\mathbb{R}^2} \nabla v(t,x):\nabla u(t,x)\,\dd x .
\end{align}
Now we sum the two expressions and we integrate in $(\eps,t)$ with $\eps>0$ 
\begin{align}
    \int_\eps^t\int_{\R^2}\rho \dot{u} \cdot v + \rho \dot{v} \cdot u = - 2\int_\eps^t \int_{\R^2} \nabla u : \nabla v
\end{align}
and we use \cref{cor:admutrasp}, in particular \eqref{eq:traspvu}, to say that
\begin{align}
  \int_{\R^2} \rho (t) u(t) \cdot v(t)
    - \int_{\R^2} \rho (\eps) u(\eps) \cdot v(\eps)= - 2\int_\eps^t \int_{\R^2} \nabla u : \nabla v.
\end{align}
The continuity for $\eps \to 0$ of the left-hand side is proved in \cref{cor:admutrasp} and the convergence of the right-hand side is trivial since $\nabla u, \; \nabla v \in L^2_tL^2_x$. 
\end{proof}

\subsection{The non linear case}
We now focus on the following setting.

\noindent\textbf{Scenario 2}\label{sce:scenario2}.
\begin{itemize}
    \item[(i)] $(\rho_1,u_1)$ and $(\rho_2,u_2)$ are immediately strong solutions of \eqref{eq:INS} with initial data $(\rho_0,u_1(0))$ and $(\rho_0,u_2(0))$, respectively, satisfying \eqref{ass:weakData2}. Moreover, $v_1$ and $v_2$ are immediately strong solutions of \eqref{eq:LSs}, advected by $(\rho_1,u_1)$ and $(\rho_2,u_2)$, respectively, with initial data $v_1(0)$ and $v_2(0)$.
    
    \item[(ii)] We assume that $\rho_2 v_2 \to \rho_0 v_2(0)$ in $L^2_x$ as $t \to 0$, and that
    \begin{align}
        \sqrt{t}\norm{\nabla v_2(t)}_{L^2(\R^2)} \to 0, \qquad
        \rho_2 D_{u_2}(v_2),\; \sqrt{t}\,\nabla D_{u_2}(v_2) \in L^1_t L^2_x, \quad \nabla v_2 \in L^1_t L^\infty_x,
    \end{align}
    where all spaces are understood locally in time.
\end{itemize}
For notational convenience, we denote $\dot v_i := D_{u_i}v_i$ and $\dot u_i := D_{u_i}u_i$, unless this causes ambiguity.

\begin{theorem}\label{thm:relaEnergyNL}
Under the assumptions of Scenario~2, for almost every $t >0 $ we have
\begin{align}
   \|(\rho_1\,\delta v)(t)\|_{L^2(\R^2)}^2
    & + 2 \int_0^t \|\nabla \delta v(s)\|_{L^2(\R^2)}^2 \,ds
    \le  \|\rho_0 \,\delta v(0)\|_{L^2(\R^2)}^2  \\
    &\quad
    - 2\int_0^t\!\!\int_{\R^2} \rho_1\,\delta u\!\otimes\!\delta v : \nabla v_2
    - 2\int_0^t\!\!\int_{\R^2} \delta \rho\,\delta v \cdot \dot{v}_2 .
\end{align}
\end{theorem}

\begin{proof}
Using the energy inequality for both $(\rho_1,v_1)$ and $(\rho_2,v_2)$, we obtain
\begin{align}\label{eq:Erel-ineq-1}
\begin{aligned}
  \|(\rho_1\,\delta v)(t)\|_{L^2(\R^2)}^2
&= \|(\rho_1 v_1)(t)\|_{L^2(\R^2)}^{2}
   - 2\int_{\R^2} (\rho_1 v_1)(t)\cdot v_2(t)
   +  \|(\rho_1 v_2)(t)\|_{L^2(\R^2)}^{2} \\
&=  \|(\rho_1 v_1)(t)\|_{L^2(\R^2)}^{2}
   - 2\int_{\R^2} (\rho_1 v_1)(t)\cdot v_2(t)
   +  \|(\rho_2 v_2)(t)\|_{L^2(\R^2)}^{2}
   +  \int_{\R^2} \delta\rho(t)\,|v_2(t)|^2 \\
&\le  \|\rho_0 v_1(0)\|_{L^2(\R^2)}^{2}
   +  \|\rho_0 v_2(0)\|_{L^2(\R^2)}^{2}
   - 2 \int_0^t \Bigl( \|\nabla v_1(s)\|_{L^2(\R^2)}^{2}
                     + \|\nabla v_2(s)\|_{L^2(\R^2)}^{2} \Bigr)\,ds \\
&\quad - 2\int_{\R^2} (\rho_1 v_1)(t)\cdot v_2(t)
   +  \int_{\R^2} \delta\rho(t)\,|v_2(t)|^2 .
\end{aligned}
\end{align}
For the last two terms in \eqref{eq:Erel-ineq-1}, we use \cref{prop:Scen2}. Combining \eqref{eq:Erel-ineq-1} with \eqref{eq:scen2cor}, we obtain
\begin{align}
    \|(\rho_1 \delta v)(t)\|_{L^2(\R^2)}^2 \le\,
    & \|\rho_0 v_1(0)\|_{L^2(\R^2)}^{2}
     +  \|\rho_0 v_2(0)\|_{L^2(\R^2)}^{2}
     - 2 \int_{\R^2} \rho_0 v_1(0) \cdot v_2(0) \\
    & - 2\int_0^t \Bigl( \|\nabla v_1(s)\|_{L^2(\R^2)}^{2}
                       + \|\nabla v_2(s)\|_{L^2(\R^2)}^{2} \Bigr)\,ds
      + 4 \int_0^t \int_{\R^2} \nabla v_1 : \nabla v_2 \\
    & - 2 \int_0^t \int_{\R^2} \rho_1 \delta u \otimes \delta v : \nabla v_2
      - 2 \int_0^t \int_{\R^2} \delta \rho\,\delta v \cdot \dot{v}_2 .
\end{align}
By observing that
\begin{align}
\|\nabla v_1\|_{L^2}^2 + \|\nabla v_2\|_{L^2}^2
- 2 \int_{\R^2} \nabla v_1 : \nabla v_2
= \|\nabla \delta v\|_{L^2}^2,
\end{align}
we conclude the proof.
\end{proof}

To prove \cref{prop:Scen2}, we first need to establish regularity properties for cross quantities such as $\rho_1 v_2$.

\begin{lemma}\label{lem:cross}
Under the assumptions of Scenario~2, we have
\begin{align}
    \delta \rho \,\dot{v}_2 \in L_\loc^1([0,\infty);L^2(\R^2)), \qquad \delta \rho \,\delta v \in L_\loc^\infty([0,\infty);L^2(\R^2)).
\end{align}
Moreover, the following convergences hold:
\begin{align}\label{eq:conv}
    \lim_{\eps \to 0} \int_{\R^2} (\rho_1v_1)(\eps) \cdot v_2(\eps)
    = \int_{\R^2} \rho_0 v_1(0) \cdot v_2(0)\qquad
    \lim_{\eps \to 0} \int_{\R^2} \delta \rho(\eps)\,|v_2(\eps)|^2 = 0.
\end{align}
\end{lemma}

\begin{proof}
Fix $T>0$. Using \cref{prop:ImprovedDecay} for both $\rho_1$ and $\rho_2$ and picking the maximum among the two constants we get that for almost every $t\in (0,T)$,
\begin{align}\label{eq:startPoint}
    \supp \rho_i(t) \subseteq D + B_{N\sqrt{t}}(0).
\end{align}
In particular, for every measurable function $v$,
\begin{align}
    \norm{\delta \rho(t)\, v}_{L^2(\R^2)} \le \norm{v}_{L^2(D_{N,t})}.
\end{align}
Choosing $R = N$ in \cref{cor:Lpstrip}, we obtain for both $i=1$ and $i=2$, and for almost every $t \in (0,T)$,
\begin{align}\label{eq:lastcite}
    \norm{v}_{L^2(D_{N,t})} \lesssim \sqrt{t}\norm{\nabla v}_{L^2(\R^2)} + \norm{\rho_i(t)\, v}_{L^2(\R^2)}.
\end{align}
Choosing $v = \dot{v}_2(t)$ and $i=2$, we infer
\begin{align}
    \norm{\delta \rho(t)\, \dot{v}_2(t)}_{L^2(\R^2)} \lesssim \sqrt{t}\norm{\nabla \dot{v}_2(t)}_{L^2(\R^2)} + \norm{\rho_2(t)\, \dot{v}_2(t)}_{L^2(\R^2)},
\end{align}
which proves that $\delta \rho\, \dot{v}_2 \in L_\loc^1([0,\infty);L^2(\R^2))$. Choosing instead $v = v_1(t)$ and $v = v_2(t)$, respectively, we obtain
\begin{align}\label{eq:crosu}
\begin{aligned}
    \norm{\delta \rho(t)\, v_1(t)}_{L^2(\R^2)} &\lesssim \sqrt{t}\norm{\nabla v_1(t)}_{L^2(\R^2)} + \norm{\rho_1(t)\, v_1(t)}_{L^2(\R^2)}, \\
    \norm{\delta \rho(t)\, v_2(t)}_{L^2(\R^2)} &\lesssim \sqrt{t}\norm{\nabla v_2(t)}_{L^2(\R^2)} + \norm{\rho_2(t)\, v_2(t)}_{L^2(\R^2)}.
\end{aligned}
\end{align}
This proves that $\delta \rho\, \delta v \in L_\loc^\infty([0,\infty);L^2(\R^2))$.
The continuity at the initial time is nontrivial. We start by proving the second convergence in \eqref{eq:conv}. Recall that $\delta\rho$ solves
\begin{align}\label{eq:deltarhop}
    \partial_t \delta \rho + \dive(\rho_1 u_1) = \dive(\rho_2 u_2),
    \qquad \delta\rho(0)=0 .
\end{align}
Testing \eqref{eq:deltarhop} with a regular enough $\varphi$ and integrating in time, we obtain for a.e.\ $t>0$
\begin{align}
   \Big|\int_{\R^2} \delta \rho(t)\,\varphi \,dx\Big|
   &= \Big|\int_0^t \int_{D_{N,s}} \big(\rho_1 u_1 - \rho_2 u_2\big)(s)\cdot \nabla \varphi \,dx\,ds \Big| \\ 
   &\le \int_0^t\Big( \| u_1(s)\|_{L^\infty (D_{N,s})}
           + \| u_2(s)\|_{L^\infty (D_{N,s})}\Big)
        \|\nabla\varphi\|_{L^1(D_{N,s})}\,ds ,
\end{align}
where we used \eqref{eq:startPoint} to restrict the domain of integration to $D_{N,s}$. By \cref{lem:Linfty} we have
\begin{align}
    \| u_i(s)\|_{L^\infty (D_{N,s})} \lesssim s^{-1/2},
\end{align}
and therefore
\begin{align}\label{eq:transTrick}
    \Big| \int_{\R^2} \delta \rho(t)\,\varphi \,dx \Big|
    \lesssim \sqrt{t}\,\|\nabla\varphi\|_{L^1(D_{N,t})}.
\end{align}
Inserting $\varphi = |v_2(t)|^2$ in \eqref{eq:transTrick}, we obtain
\begin{align}
 \Big|\int_{\R^2} \delta \rho(t)\,|v_2(t)|^2\,dx\Big|
   \lesssim  \|v_2(t)\|_{L^2(D_{N,t})}
      \big(\sqrt t\,\|\nabla v_2(t)\|_{L^2(\R^2)}\big).
\end{align}
By \cref{cor:Lpstrip}, the first factor is bounded in $L^\infty(0,t)$, and by assumption
\begin{align}
\sqrt t\,\|\nabla v_2(t)\|_{L^2(\R^2)} \longrightarrow 0
\qquad\text{as } t\downarrow 0.
\end{align}
Hence we conclude the second convergence in \eqref{eq:conv}.

We now choose $\varphi = v_2(t)\cdot \psi$ in \eqref{eq:transTrick}, with $\psi \in H^1(\R^2)$, and obtain
\begin{align}
 \Big|\int_{\R^2} \delta \rho(t) v_2(t) \cdot \psi\,dx\Big|
   \lesssim  \|\psi\|_{L^2(\R^2)}
      \big(\sqrt t\,\|\nabla v_2(t)\|_{L^2(\R^2)}\big)
      + \sqrt t\,\|v_2(t)\|_{L^2(D_{N,t})}
      \,\|\nabla \psi\|_{L^2(\R^2)} ,
\end{align}
which converges to $0$ as $t\to 0$. This proves that
\begin{align}
\delta \rho(t) v_2(t)\rightharpoonup 0 \quad \text{in } \dot{H}^{-1}(\R^2).
\end{align}
By \eqref{eq:crosu} and a standard density argument we get
\begin{align}
\delta \rho(t) v_2(t)\rightharpoonup 0 \quad \text{in } L^2(\R^2).
\end{align}

The convergence of $\delta \rho\,|v_2|^2$ at the initial time, combined with the fact that
$\rho_2 v_2 \to \rho_0 v_2(0)$ in $L^2(\R^2)$, also yields
\begin{align}\label{eq:crosu2}
    \int_{\R^2} \rho_1(\eps)\,|v_2(\eps)|^2
    = \int_{\R^2} \delta \rho(\eps)\,|v_2(\eps)|^2
      + \int_{\R^2} \rho_2(\eps)\,|v_2(\eps)|^2
    \longrightarrow \int_{\R^2} \rho_0\,|v_2(0)|^2 .
\end{align}
Moreover, for every $\varphi \in L^2(\R^2)$ we have
\begin{align}
    \int_{\R^2} \rho_1(\eps) v_2(\eps) \cdot \varphi
    = \int_{\R^2} \delta \rho (\eps) v_2(\eps) \cdot \varphi
      + \int_{\R^2} \rho_2 (\eps) v_2(\eps) \cdot \varphi
    \longrightarrow \int_{\R^2} \rho_0 v_2(0) \cdot \varphi .
\end{align}
Hence $(\rho_1v_2)(\eps) \to \rho_0 v_2(0)$ in $L^2(\R^2)$.

We are now ready to prove the first convergence in \eqref{eq:conv}. We write
\begin{align}
     \int_{\R^2} (\rho_1 v_1)(\eps) \cdot v_2(\eps)
    - \rho_0\, v_1(0) \cdot v_2(0)
    &=  \int_{\R^2} (\rho_1v_1)(\eps) \cdot \big[v_2(\eps)-v_2(0)\big] \\
    &\quad +  \int_{\R^2} \big[(\rho_1 v_1)(\eps) 
    - \rho_0\, v_1(0)\big] \cdot v_2(0).
\end{align}
The second term converges to zero since $\mathbb{P}(\rho_1v_1) \in C([0,T];L^2_w)$ and $v_2(0)$ is divergence free. For the first term we estimate
\begin{align}
    \Big|\int_{\R^2} (\rho_1v_1)(\eps) \cdot \big[v_2(\eps)-v_2(0)\big]\Big|
    &\le \|\rho_1 v_1\|_{L^\infty_t L^2_x}
       \|\rho_1(\eps) v_2(\eps) - \rho_1(\eps) v_2(0)\|_{L^2_x} \\
    &\lesssim \|\rho_1(\eps) v_2(\eps) - \rho_0 v_2(0)\|_{L^2_x}
      + \|(\rho_1-\rho_0) v_2(0)\|_{L^2_x}.
\end{align}
The first term converges to zero by the strong convergence
$(\rho_1v_2)(\eps) \to \rho_0 v_2(0)$ in $L^2(\R^2)$, while the second one is handled exactly as $J(\varepsilon)$ in \cref{cor:admutrasp}.

\end{proof}

\begin{proposition}\label{prop:Scen2}
Under the assumptions of Scenario~2, for almost every $t>0 $ we have
\begin{align}\label{eq:scen2cor}
\begin{aligned}
       - \int_{\R^2} \rho_1(t)\,v_1(t)&\cdot v_2(t)
     + \frac12 \int_{\R^2} \delta\rho(t)\,|v_2(t)|^2
    = - \int_{\R^2} \rho_0 v_1(0) \cdot v_2(0) \\
    &\quad - \int_0^t \int_{\R^2} \rho_1\, \delta u \otimes \delta v : \nabla v_2
          - \int_0^t \int_{\R^2} \delta\rho\, \delta v \cdot \dot v_2
          + 2 \int_0^t \int_{\R^2} \nabla v_1 : \nabla v_2 .
          \end{aligned}
\end{align}
\end{proposition}

\begin{proof}
We split the proof into three steps. In the first two steps, we prove two inequalities which, when combined in Step~3, yield the claim.

\underline{Step~1.} By \cref{prop:tailest}, for almost every $t > 0$ and for $i=1,2$, we have $\abs{v_i(t,x)} = O(1)$ as $\abs{x}\to\infty$. We may therefore apply \cref{cor:tailEst} to $(\rho_1,u_1)$ with $v=v_1(t)$ and $h(t)=v_2(t)$. After integrating over $(\eps,t)$, we obtain 
\begin{align}\label{eq:tosum1}
0 = -\int_\eps^{t}\!\!\int_{\R^2} \rho_1\, \dot v_1 \cdot v_2 
    - \int_\eps^{t}\!\!\int_{\R^2} \nabla v_1 : \nabla v_2,
\end{align}
and, in the same way,
\begin{align}\label{eq:tosum2}
0 = - \int_\eps^{t}\!\!\int_{\R^2} \rho_2\, \dot v_2 \cdot v_1
    - \int_\eps^{t}\!\!\int_{\R^2} \nabla v_1 : \nabla v_2.
\end{align}
A crucial step is to test the weak formulation of the transport equation associated with $(\rho_1,u_1)$ against the function $v_1\cdot v_2$ and to obtain
\begin{align}\label{eq:quickgoal}
  \int_{\R^2} \rho_1(t)\, v_1(t)\cdot v_2(t)
  - \int_{\R^2} \rho_1(\varepsilon)\, v_1(\varepsilon)\cdot v_2(\varepsilon)
  =
  \int_\varepsilon^t \!\!\int_{\R^2}
  \rho_1\, v_2\cdot \dot v_1
  + \rho_1\, v_1\cdot D_{u_1}v_2 .
\end{align}
This identity is precisely \eqref{eq:traspvu} with the corresponding replacements.

As in the proof of \cref{cor:admutrasp}, the main difficulty in establishing \eqref{eq:quickgoal} is to verify that $v_1\cdot v_2$ is an admissible test function. The argument follows the same lines as in \cref{cor:admutrasp}; however, since $v_2$ is not advected by $(\rho_1,u_1)$, quantities such as $\rho_1 v_2$ or $\rho_1 D_{u_1}v_2$ may a priori fail to have the required spatial integrability.

This issue can be handled exactly as in the proof of \cref{lem:cross}, by repeatedly exploiting localized $L^p$ bounds for $v_2$ and $D_{u_1}v_2$. As a representative example, for almost every $t> 0$ one has
\begin{align}\label{eq:rep}
  \|\rho_1(t)\,v_2(t)\|_{L^\infty(\R^2)}
  \le \|v_2(t)\|_{L^\infty(D_{N,\sqrt t})}
  \lesssim
  t^{\frac34}\|\nabla v_2(t)\|_{L^4(\R^2)}
  + t^{-\frac12}\|\rho_2(t)\,v_2(t)\|_{L^2(\R^2)} .
\end{align}
Similar estimates apply to all the other terms arising in the weak formulation, and ensure that $v_1\cdot v_2$ is an admissible test function.

Once \eqref{eq:quickgoal} is established, we add \eqref{eq:tosum1}, \eqref{eq:tosum2}, and \eqref{eq:quickgoal} to obtain
\begin{align}\label{eq:scen21}
\begin{aligned}
 \int_{\R^2} \rho_1(t)\,v_1(t) \cdot v_2(t) - & \int_{\R^2} \rho_1(\eps)\,v_1(\eps)\cdot v_2(\eps)\\
 = & \int_\eps^t \int_{\R^2} \rho_1 v_1 \cdot D_{u_1}v_2 - \rho_2 v_1 \cdot D_{u_2}v_2- 2 \int_\eps^t \int_{\R^2} \nabla v_1 : \nabla v_2.
\end{aligned}
\end{align}

\underline{Step~2.} We now turn to a different identity, following a similar strategy. The density difference $\delta\rho := \rho_1-\rho_2$ satisfies
\begin{align}\label{eq:deltarho}
  \partial_t \delta\rho + \dive(\delta\rho\,u_2)
  = - \dive(\rho_1\,\delta u),
\end{align}
in the sense of distributions. Therefore, provided that $\lvert v_2\rvert^2/2$ is an admissible test function on compact subsets of $(0,\infty)$, we obtain
\begin{align}
  \frac12 \int_{\R^2} \delta\rho(t)\,\lvert v_2(t)\rvert^2
  - \frac12 \int_{\R^2} \delta\rho(\varepsilon)\,\lvert v_2(\varepsilon)\rvert^2
  &=
  \int_\varepsilon^t \!\!\int_{\R^2}
  \delta\rho\, v_2\cdot \partial_t v_2
  + \int_\varepsilon^t \!\!\int_{\R^2}
  \delta\rho\, u_2\cdot \nabla\!\left(\frac{\lvert v_2\rvert^2}{2}\right) \\
  &\quad
  + \int_\varepsilon^t \!\!\int_{\R^2}
  \rho_1\, \delta u \cdot \nabla\!\left(\frac{\lvert v_2\rvert^2}{2}\right) .
\end{align}
After straightforward manipulations, this identity yields
\begin{align}\label{eq:scen22}
\frac12 \int_{\R^2} \delta\rho(t)\,\lvert v_2(t)\rvert^2
- \frac12 \int_{\R^2} \delta\rho(\eps)\,\lvert v_2(\eps)\rvert^2
&= \int_\eps^t \int_{\R^2}
   \delta\rho\, v_2 \cdot D_{u_2}v_2
 + \rho_1\, \delta u \otimes v_2 : \nabla v_2 .
\end{align}
As before, the only nontrivial point is to verify that $\lvert v_2\rvert^2/2$ is an admissible test function for \eqref{eq:deltarho}. In this case, the required $L^p$ bounds for the products $\delta\rho\,v_2$ and $\delta\rho\,\dot v_2$ are again obtained by exploiting the finite speed of propagation of the density, namely
\begin{align}
  \supp \delta\rho(t) \subset D + B_{N\sqrt t}(0),
\end{align}
and by repeating the localized estimates described above in \eqref{eq:rep}.

\underline{Step~3.} Subtracting \eqref{eq:scen22} from \eqref{eq:scen21}, we obtain
\begin{align}
    - \int_{\R^2}& \rho_1(t)\, v_1(t) \cdot v_2(t)
    + \frac12 \int_{\R^2} \delta\rho(t)\,|v_2(t)|^2
    = - \int_{\R^2} \rho_1(\eps)\,v_1(\eps)\cdot v_2(\eps)
       + \frac12 \int_{\R^2} \delta\rho(\eps)\,|v_2(\eps)|^2 \\
    &\quad + \int_\eps^t \int_{\R^2}
    \underbrace{-\rho_1 v_1 \cdot D_{u_1}v_2
    + \rho_2 v_1 \cdot D_{u_2}v_2
    + \delta\rho\, v_2 \cdot D_{u_2}v_2
    + \rho_1\, \delta u \otimes v_2 : \nabla v_2}_{=:K}  + 2 \int_\eps^t \int_{\R^2} \nabla v_1 : \nabla v_2 .
\end{align}
By straightforward algebraic manipulations, we can rewrite
\begin{align}
    K = - \rho_1\, \delta u \otimes \delta v : \nabla v_2
        - \delta\rho\, \delta v \cdot D_{u_2}v_2 .
\end{align}
We now study the limit as $\eps \to 0$. By \eqref{eq:conv}, we have
\begin{align}
- \int_{\R^2} \rho_1(\eps)\,v_1(\eps) \cdot v_2(\eps)\,dx
+ \frac12 \int_{\R^2} \delta \rho(\eps)\,|v_2(\eps)|^2\,dx
\longrightarrow
- \int_{\R^2} \rho_0 v_1(0) \cdot v_2(0)\,dx .
\end{align}
For the term $K$, we estimate
\begin{align}
    \|K\|_{L^1(\R^2)}
    \le \|\rho_1\, \delta u\|_{L^2(\R^2)}\,\|\rho_1\, \delta v\|_{L^2(\R^2)}\, \|\nabla v_2\|_{L^\infty(\R^2)}
       + \|\delta\rho\, \delta v\|_{L^2(\R^2)}\, \| \delta \rho\dot  v_2\|_{L^2(\R^2)},
\end{align}
which belongs to $L^1(0,t)$ by \cref{lem:cross} and the assumptions on $v_2$ (the inclusion $\rho_1\, \delta u \in L^\infty_t L^2_x$ can be proved by choosing $v = \delta u(t)$ and $i=1$ in \eqref{eq:lastcite}). Since $\nabla v_1, \nabla v_2 \in L^2_t L^2_x$, by dominated convergence we may pass to the limit as $\eps \to 0$, and we obtain
\begin{align}
    - \int_{\R^2} \rho_1(t)\,v_1(t)&\cdot v_2(t)
     + \frac12 \int_{\R^2} \delta\rho(t)\,|v_2(t)|^2
    = - \int_{\R^2} \rho_0 v_1(0) \cdot v_2(0) \\
    &\quad - \int_0^t \int_{\R^2} \rho_1\, \delta u \otimes \delta v : \nabla v_2
          - \int_0^t \int_{\R^2} \delta\rho\, \delta v \cdot \dot v_2
          + 2 \int_0^t \int_{\R^2} \nabla v_1 : \nabla v_2 .
\end{align}
\end{proof}

\subsection*{Acknowledgments}

The authors are supported by the Deutsche Forschungsgemeinschaft (DFG) through the project \emph{Inhomogeneous and compressible fluids: statistical solutions and dissipative anomalies} within the SPP 2410 \emph{Hyperbolic Balance Laws in Fluid Mechanics: Complexity, Scales, Randomness} (CoScaRa).

Part of this work was carried out during S.~\v{S}kondri\'c’s visit to the Department of Mathematics and Computer Science at the University of Basel.

We thank G.~Crippa, M.~Nesi and E.~Wiedemann for helpful discussions.

\citation

\bibliographystyle{plain} 
\bibliography{bibFin}

\end{document}